\documentclass[11pt]{article}
\usepackage{epsfig,mst-stylefile,amssymb,amsmath,harvard,epstopdf}
\usepackage{multirow}
\usepackage{hhline}
\usepackage[svgnames]{xcolor}

\newcommand{\bbZ}{{\Bbb Z}}
\newcommand{\bbR}{{\Bbb R}}
\newcommand{\bbN}{{\Bbb N}}
\newcommand{\bbC}{{\Bbb C}}

\newcommand{\bbE}{{\Bbb E}}

\newcommand{\Cov}{\textnormal{Cov}}
\newcommand{\Var}{\textnormal{Var\hspace{.5mm}}}
\renewcommand{\P}{\mathbb{P}}
\renewcommand{\R}{\mathbb{R}}
\newcommand{\E}{\mathbb{E}}
\newcommand{\N}{\mathbb{N}}
\renewcommand{\Z}{\mathbb{Z}}
\newcommand{\pto}{\stackrel{P}\rightarrow}
\DeclareMathOperator*{\argmin}{arg\,min}

\newcommand{\bth}{\boldsymbol \theta}

\newcommand{\norm}[1]{\left\|#1\right\|}

\newcommand{\Edj}{\E_{\bth_0} d^2(2^j,0)} 
\newcommand{\Edjp}{\E_{\bth_0} d^2(2^{j'},0)}
\newcommand{\Edjf}{\E_{\bth_0}  d^2_{\textnormal{fBm}}(2^j,0)} 
\newcommand{\Edjpf}{\E_{\bth_0}  d^2_{\textnormal{fBm}}(2^{j'},0)}
\newcommand{\Edjfa}{\E_{\bth'}  d^2_{\textnormal{fBm}}(2^j,0)} 
\newcommand{\Edjpfa}{\E_{\bth'}  d^2_{\textnormal{fBm}}(2^{j'},0)}

\renewcommand{\cite}{\citeyear}
\begin{document}

\title{Tempered fractional Brownian motion: wavelet estimation, modeling and testing
\thanks{The second author was partially supported by the prime award no.\ W911NF-14-1-0475 from the Biomathematics subdivision of the Army Research Office, USA. The authors would like to thank Laurent Chevillard for his comments on this work.}
}

\author{B.\ Cooper Boniece, Gustavo Didier \\ Mathematics Department\\ Tulane University
\and Farzad Sabzikar\\ Department of Statistics\\ Iowa State University}


\bibliographystyle{agsm}

\maketitle

\begin{abstract}
The Davenport spectrum is a modification of the classical Kolmogorov spectrum for the inertial range of turbulence that accounts for non-scaling low frequency behavior. Like the classical fractional Brownian motion vis-\`a-vis the Kolmogorov spectrum, tempered fractional Brownian motion (tfBm) is a canonical model that displays the Davenport spectrum. The autocorrelation of the increments of tfBm displays semi-long range dependence (hyperbolic and quasi-exponential decays over moderate and large scales, respectively), a phenomenon that has been observed in wide a range of applications from wind speeds to geophysics to finance. In this paper, we use wavelets to construct the first estimation method for tfBm and a simple and computationally efficient test for fBm vs tfBm alternatives. The properties of the wavelet estimator and test are mathematically and computationally established. An application of the methodology to the analysis of geophysical flow data shows that tfBm provides a much closer fit than fBm.
\end{abstract}

\section{Introduction}

A \textit{tempered fractional Brownian motion} (tfBm) $B_{H,\lambda} = \{B_{H,\lambda}(t)\}_{t \in \bbR}$ with Hurst parameter $H > 0$ and tempering parameter $\lambda > 0$ is the stochastic process defined by the moving average representation
\begin{equation}\label{e:tfBm_MA}
B_{H,\lambda}(t) = \int_{\bbR} \{e^{- \lambda(t-u)_{+}}(t - u)^{H-1/2}_{+} - e^{- \lambda(-u)_{+}}(- u)^{H-1/2}_{+} \} B(du),
\end{equation}
where $u_{+} = u 1_{\{u > 0\}}$ and $B(du)$ is an independently scattered Gaussian random measure satisfying $\bbE B(du)^2 = \sigma^2du$. In the boundary case $\lambda = 0$ (and $0 < H < 1$), tfBm reduces to a fractional Brownian motion (fBm), namely, a Gaussian, stationary-increment, self-similar process (e.g., Embrechts and Maejima \cite{embrechts:maejima:2002}, Taqqu \cite{taqqu:2003}). TfBm is a new canonical model introduced in Meerschaert and Sabzikar \cite{meerschaert:sabzikar:2013,meerschaert:sabzikar:2014}, Sabzikar et al.\ \cite{sabzikar:meerschaert:chen:2015} that displays the so-named \textit{Davenport spectrum} (Davenport \cite{davenport:1961}). The latter is a modification of the classical Kolmogorov spectrum for the inertial range of turbulence that accounts for low frequency behavior and has been successfully applied in wind speed modeling (Norton and Wolff \cite{norton:wolff:1981}, Li and Kareem \cite{li:kareem:1990}, Beaupuits et al.\ \cite{beaupuits:etal:2004}). On the other hand, a \textit{wavelet} is a unit $L^2(\bbR)$-norm function that annihilates a certain number of polynomials (see \eqref{e:N_psi}). In this paper, we use wavelets to construct the first estimator for tfBm and a simple and computationally efficient test for fBm vs tfBm alternatives. The asymptotic properties of the wavelet estimator and test are mathematically established and their finite sample performance is computationally studied. An application of the methodology in geophysical flow data shows that tfBm provides a much closer fit than fBm. 

Classical models of turbulence describe how kinetic energy at the largest length scales is progressively transferred down to smaller scales. In the complete Kolmogorov spectral model for turbulence (Kolmogorov \cite{kolmogorov:1940,kolmogorov:1941}, Friedlander and Topper \cite{friedlander:topper:1961}, Shiryaev \cite{shiryaev:1999}), large eddies are produced in the low frequency range of scales, whereas in the inertial range (moderate frequencies), larger eddies are continuously broken down into smaller eddies, until they eventually dissipate (high frequencies). In the landmark paper Mandelbrot and Van Ness \cite{mandelbrot:vanness:1968}, Kolmogorov's model was revisited when fBm was proposed as a framework for the analysis of \textit{scale invariant}, or non-Markovian, phenomena (see Graves et al.\ \cite{graves:gramacy:watkins:franzke:2017}). A system is called scale invariant if its dynamics are driven by a continuum of time scales instead of a few characteristic scales. Within the fBm family, the instance $H = 1/3$ corresponds to the Kolmogorov spectrum for the inertial range.

Unlike fBm, tfBm is not self-similar. Instead, it satisfies the scaling property
$$
\{B_{H,\lambda}(ct)\}_{t \in \bbR} \stackrel{{\mathcal L}}= \{c^{H}B_{H,c\lambda}(t)\}_{t \in \bbR}, \quad c > 0,
$$
where $\stackrel{{\mathcal L}}=$ indicates equality of finite dimensional distributions. This is a consequence of the presence of the extra (tempering) parameter $\lambda > 0$, which controls the deviation from a fBm's power law spectrum at low frequencies. Notably, even though a stationary-increment process, tfBm converges almost surely to stationarity as $t\to\infty$. This property of tfBm makes it suitable for modeling data that looks stationary or approximately so. TfBm also exhibits \emph{semi-long range dependence}, i.e., the correlation between its increments decays essentially like a power law over fine/moderate scales (fractional or scale invariant behavior), but quasi-exponentially over large scales (see \eqref{e:semi-lrd}; cf.\ Giraitis et al.\ \cite{giraitis:kokoszka:leipus:2000}).

Research on tempered fractional processes has been expanding at a fast pace. Several models (ARTFIMA, tempered diffusion, tempered stable motions, tempered L\'{e}vy flights) have recently been studied and used in a wide range of modern applications such as in the physics and modeling of transient anomalous diffusion (Piryatinska et al.\ \cite{piryatinska:sanchev:woyczynski:2005}, Stanislavsky et al.\ \cite{stanislavsky:weron:weron:2008}, Sandev et al.\ \cite{sandev:chechkin:kantz:metzler:2015}, Wu et al.\ \cite{wu:deng:barkai:2016}, Chen et al.\ \cite{chen:wang:deng:2017}, Liemert et al.\ \cite{liemert:sandev:kantz:2017}, Chen et al.\ \cite{chen:wang:deng:2018}), geophysical flows (Meerschaert et al.\ \cite{meerschaert:zhang:baeumer:2008}, Meerschaert et al.\ \cite{meerschaert:sabzikar:phanikumar:zeleke:2014}) and finance (Dacorogna et al.\ \cite{dacorogna:muller:nagler:olsen:pictet:1993}, Granger and Ding \cite{granger:ding:1996}, Cont et al.\ \cite{cont:potters:bouchaud:1997}, Ling and Li \cite{ling:li:2001}, Zhang and Xiao \cite{zhang:xiao:2017}). See also Chevillard \cite{chevillard:2017} on turbulence modeling based on regularization. In addition, tempered fractional processes are closely related to near integrated models in econometrics (Phillips et al.\ \cite{phillips:moon:xiao:2001}). However, in spite of the fast growth of the literature on tempered fractional processes, little work has been done, in general, on their statistical methodology. In Meerschaert et al.\ \cite{meerschaert:sabzikar:phanikumar:zeleke:2014}, turbulence data is modeled in the Fourier domain within an ARTFIMA framework. In Anh, Heyde and Tieng \cite{anh:heyde:tieng:1999} and Gao et al.\ \cite{gao:anh:heyde:tieng:2001}, respectively, wavelet log-regression and Whittle-type methods are constructed for subclasses of tempered fractional processes assuming the tempering parameter $\lambda$ is known (see also Anh, Angulo and Ruiz-Medina \cite{anh:angulo:ruiz-medina:1999}). 

The wavelet transform is a powerful tool for the analysis of scale invariant phenomena (Flandrin \cite{flandrin:1992}, Wornell and Oppenheim \cite{wornell:oppenheim:1992}, Abry et al.\ \cite{abry:flandrin:taqqu:veitch:2003}, Percival and Walden \cite{percival:walden:2006}). Given an observed fractional time series $X$, estimation of the parameter $H$ can be conducted by a log-regression procedure that draws upon the scaling property of the sample wavelet variance, i.e.,
\begin{equation}\label{e:linear_wavelet_spectrum}
\frac{1}{n_j}\sum^{n_j}_{k=1} d_X^2(2^j,k) \simeq C 2^{j(2H+1)},
\end{equation}
where $d_X(2^j,k)$, $k = 1,\hdots,n_j$, is the wavelet transform at scale $2^j$ and shift $k$ (Veitch and Abry \cite{veitch:abry:1999}, Bardet et al.\ \cite{bardet:lang:moulines:soulier:2000}; see \eqref{e:nj=n/2^j} on $n_j$). Wavelet-based statistical inference has well-documented benefits, such as: $(i)$ for a sample of size $n$, the fast wavelet transform may reach computational complexity $O(n)$, which is even lower than that of the fast Fourier transform (Mallat \cite{mallat:1999}); $(ii)$ robustness with respect to contamination by polynomial trends (Craigmile et al.\ \cite{craigmile:guttorp:Percival:2005}); $(iii)$ modeling of stationary or stationary increment processes of any order in the same framework (Moulines et al.\ \cite{moulines:roueff:taqqu:2007:Fractals,moulines:roueff:taqqu:2007:JTSA,moulines:roueff:taqqu:2008}); $(iv)$ quasi-decorrelation of several families of stochastic processes (Masry \cite{masry:1993}, Bardet \cite{bardet:2002}), which often leads to Gaussian confidence intervals.

Following up on the work Boniece et al.\ \cite{boniece:sabzikar:didier:2018}, presented without proofs and containing preliminary computational studies, in this paper we construct the wavelet analysis of tfBm and the statistical theory for the first estimator for tfBm. We make the realistic assumption that only discrete time measurements are available. We characterize the phenomenon of semi-long range dependence in the wavelet domain, whereby the presence of the parameter $\lambda$ determines that, for large octaves, the wavelet spectrum of tfBm deviates from that of fBm (see Figure \ref{fig:semilrd}, left plot, and \eqref{e:approx_decorrelation}). In particular, the wavelet (or Fourier) spectrum of tfBm is not, in general, log-linear over a wide range of scales (cf.\ \eqref{e:linear_wavelet_spectrum}). Hence, the proposed estimator is based on a nonlinear regression procedure in the wavelet domain (see \eqref{e:def_estimator}; cf.\ Frecon et al.\ \cite{frecon:didier:pustelnik:abry:2016}). Its consistency and asymptotic normality are mathematically established (Theorem \ref{t:asympt}). Monte Carlo studies demonstrate the estimator's efficacy over finite samples for different instances of tfBm, at an acceptable computational cost.

In physical and modeling practice, it is also of great interest to identify the estimable range of $\lambda$ for given sample sizes (cf.\ Sabzikar et al.\ \cite{sabzikar:wang:phillips:2018}). We further propose to investigate this issue from a hypothesis testing perspective (cf.\ Giraitis et al.\ \cite{giraitis:kokoszka:leipus:teyssiere:2003}). We construct a simple and computationally efficient test for fBm vs tfBm alternatives based on comparing Hurst exponent estimates obtained from different regions of the sample wavelet spectrum. The asymptotic distribution of the test statistic is mathematically established under both the null and the alternative hypotheses (Theorem \ref{t:test}). Starting from the latter, we computationally develop power curves as a function of the sample size that hence quantify how distinguishable tfBm is from fBm for each value of $\lambda$, given $H$. We apply the estimation and testing methodology to geophysical flow data from the Red Cedar river in Michigan, USA, and conclude that tfBm is a significantly better model than fBm. This complements and justifies the exploratory data analysis presented in Boniece et al.\ \cite{boniece:sabzikar:didier:2018} (see Figure \ref{fig:semilrd}, right plot).

\begin{figure}[!htb]
    \centering
    \begin{minipage}{.5\linewidth}
        \centering
        \includegraphics[width=\linewidth]{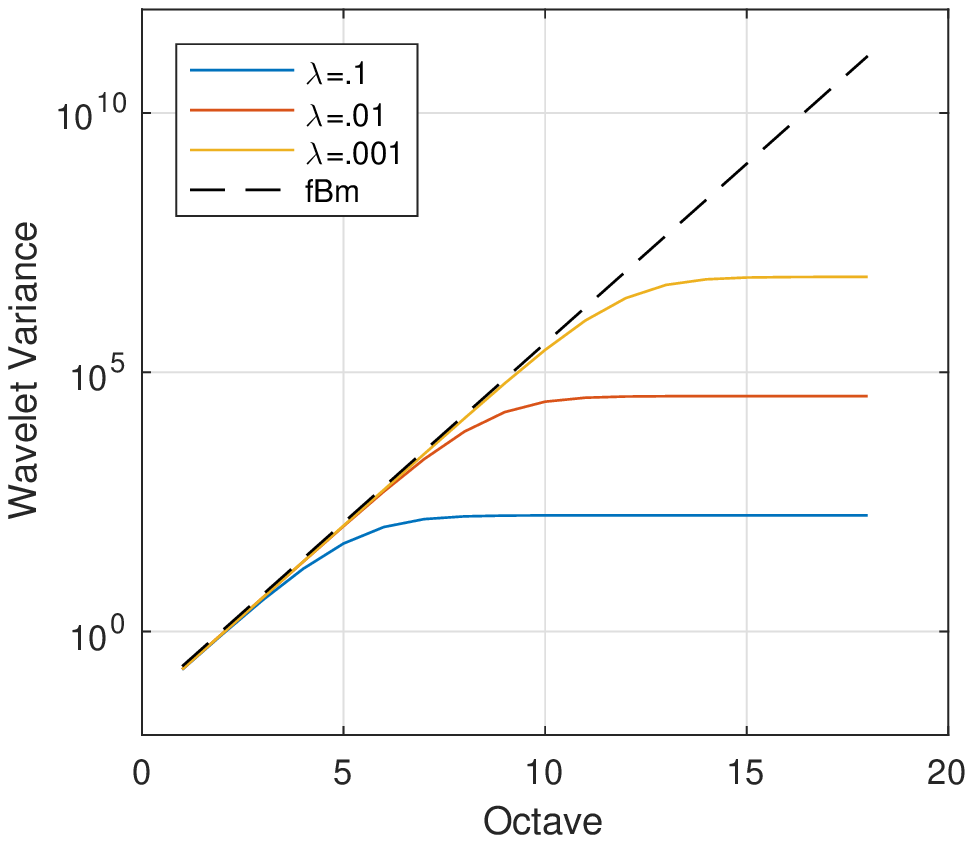}
        \label{fig:prob1_6_2}
    \end{minipage}%
    \begin{minipage}{.5\textwidth}
        \centering
        \includegraphics[width=.7\linewidth]{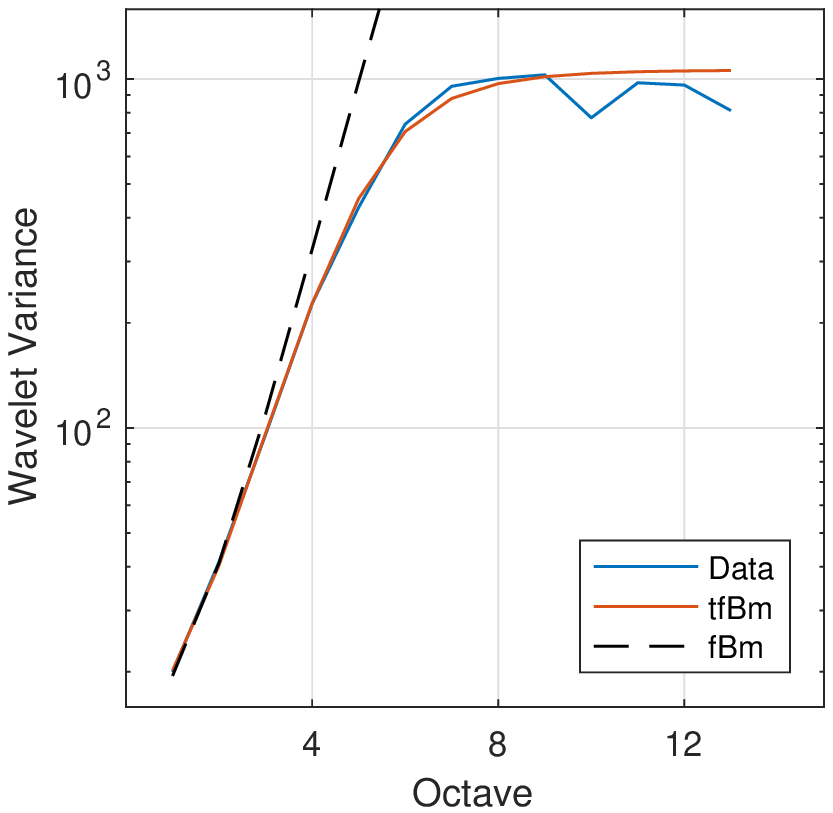}
        \label{fig:prob1_6_1}
    \end{minipage}
    \caption{\label{fig:semilrd} \textbf{Semi-long range dependence in the wavelet domain.} 
    The smaller the value of $\lambda>0$, the wider the range of octaves $j$ over which the wavelet spectrum of tfBm resembles that of fBm. Left plot: wavelet spectra of fBm and tfBm for various values of $\lambda$ (continuous time).  Right plot: sample wavelet spectra of the Red Cedar river data and best fBm and tfBm fits (plots reproduced from Boniece et al.\ \protect\cite{boniece:sabzikar:didier:2018}; see also Section \ref{s:application} for the data analysis).}
    \end{figure}

This work leads to a number of issues: in view of the convergence to stationarity of tfBm, is there a way of constructing a powerful test of tfBm vs ARMA-like alternatives?; is the parameter range $H > 1$, which is absent in the fBm framework, of physical interest? In addition, few statistical methods are available for other Gaussian or non-Gaussian tempered fractional models such as tempered fractional stable motion (Sabzikar and Surgailis \cite{sabzikar:surgailis:2018}), tempered Hermite processes (Sabzikar \cite{sabzikar:2015}), tempered stable processes (Cohen and Rosi\'nski \cite{cohen:rosinski:2007}, Rosi\'nski \cite{rosinski:2007}, Baeumer and Meerschaert \cite{baeumer:meerschaert:2010}, Bianchi et al.\ \cite{bianchi:rachev:kim:fabozzi:2010}, Gajda and Magdziarz \cite{gajda:magdziarz:2010}, Kienitz \cite{kienitz:2010}, Rosi\'nski and Sinclair \cite{rosinski:sinclair:2010}, Kawai and Masuda \cite{kawai:masuda:2012}, K\"{u}chler and Tappe \cite{kuchler:tappe:2013}) or tempered fractional Langevin dynamics (Zeng et al.\ \cite{zeng:yang:chen:2016}).

This paper is organized as follows. In Section \ref{s:preliminaries}, we recall basic properties of tfBm and construct its wavelet analysis. Section \ref{s:main} contains the main results of the paper. In Section \ref{s:asymptotic_theory_discrete}, we develop asymptotic theory for the point estimator and also the testing framework. In Section \ref{s:MC}, we present the Monte Carlo studies, and in Section \ref{s:application}, we model geophysical flow data. All proofs can be found in the Appendix, together with auxiliary results.

\section{Preliminaries}\label{s:preliminaries}

\subsection{Basic properties of tfBm}

Let $B_{H,\lambda}$ be a tfBm as in \eqref{e:tfBm_MA}. Starting from its corresponding moving average representation, it can be shown that its covariance function is given by
\begin{equation}\label{e:cov}
\textnormal{Cov}(B_{H,\lambda}(s),B_{H,\lambda}(t)) = \frac{\sigma^2}{2}\{C^2_t |t|^{2H}+ C^2_s |s|^{2H} - C^2_{t-s} |t-s|^{2H}\}, \quad s,t \in \bbR,
\end{equation}
where $C^2_t = \frac{2 \Gamma(2H)}{(2 \lambda |t|)^{2H}} - \frac{2 \Gamma(H+1/2)}{\sqrt{\pi}}\frac{1}{(2 \lambda|t|)^{H}} K_{H}(\lambda|t|)$, $t \neq 0$, $K_{\cdot}(\cdot)$ is the modified Bessel function of the second kind, and $C^2_0 = 0$ (Meerschaert and Sabzikar \cite{meerschaert:sabzikar:2013}, Proposition 2.3). TfBm admits the harmonizable (Fourier domain) representation
\begin{equation}\label{e:tfbm_harmonizable}
\Big\{ B_{H,\lambda}(t) \Big\}_{t \in \bbR} \stackrel{{\mathcal L}}= \Big\{\frac{\Gamma(H+\frac{1}{2})}{\sqrt{2\pi}}\int_\bbR \frac{e^{-ixt}-1}{(\lambda+ix)^{H+\frac{1}{2}}}\widetilde{B}(dx)\Big\}_{t \in \bbR},
\end{equation}
where $\widetilde{B}(dx)$ is a complex-valued Gaussian random measure such that $\widetilde{B}(-dx)= \overline{\widetilde{B}(dx)}$ and $\bbE |\widetilde{B}(dx)|^2 = \sigma^2 dx$. Expression \eqref{e:tfBm_MA} or \eqref{e:tfbm_harmonizable} implies that tfBm is Gaussian and has stationary increments. The increment process of tfBm, namely, $\{X(t) \}_{t\in\bbZ}:= \{B_{H,\lambda}(t+1)-B_{H,\lambda}(t) \}_{t\in\bbZ}$, is called \textit{tempered fractional Gaussian noise} (tfGn). Its covariance function $\gamma_X(h) = \bbE X(t)X(t+h)$, $h \in \bbZ$, has decay
\begin{equation}\label{e:semi-lrd}
\gamma_X(h) \sim \frac{\sigma^2\Gamma(H+\frac{1}{2})(e^{-\lambda}+e^{\lambda}-2)}{(2\lambda)^{H+1/2}} \frac{h^{H-1/2}}{ e^{\lambda h}}, \quad h \rightarrow \infty,
\end{equation}
(semi-long range dependence; cf.\ Chen et al.\ \cite{chen:wang:deng:2017}, Appendix 2).

\subsection{Wavelet analysis}

Throughout the paper, we make use of a wavelet multiresolution analysis (MRA; see Mallat \cite{mallat:1999}, chapter 7), which decomposes $L^2(\mathbb{R})$ into a sequence of \textit{approximation} (low-frequency) and \textit{detail} (high-frequency) subspaces $V_j$ and $W_j$, respectively, associated with different scales of analysis $2^{j}$. We always assume the MRA satisfies the technical conditions ($W1-W4$) (corresponding to expressions \eqref{e:N_psi}--\eqref{e:sum_k^m_phi(.-k)} in the Appendix), so such conditions are omitted in statements. Note that we work under assumptions that are closely related to the broad wavelet framework for the analysis of Gaussian stochastic processes laid out in Moulines et al.\ \cite{moulines:roueff:taqqu:2007:Fractals,moulines:roueff:taqqu:2007:JTSA,moulines:roueff:taqqu:2008}.
We further suppose the wavelet coefficients stem from Mallat's pyramidal algorithm (Mallat \cite{mallat:1999}, chapter 7). Initially, suppose an infinite time series
\begin{equation}\label{e:infinite_sample}
\{B_{H,\lambda}(\ell)\}_{\ell \in \bbZ},
\end{equation}
associated with the starting scale $2^j = 1$ (or octave $j = 0$), is available. Then, we can apply Mallat's algorithm to extract the so-named \textit{approximation} $(a(2^{j+1},\cdot))$ and \textit{detail} $(d(2^{j+1},\cdot))$ coefficients at coarser scales $2^{j+1}$ by means of an iterative procedure. In fact, as commonly done in the wavelet literature, we initialize the algorithm with the process
\begin{equation}\label{e:Btilde}
\widetilde{B}_{H,\lambda}(t) := \sum_{k \in \bbZ }B_{H,\lambda}(k)\phi(t-k) \in V_0, \quad  t\in \R.
\end{equation}
By the orthogonality of the shifted scaling functions $\{\phi(\cdot - k)\}_{k \in \bbZ}$,
\begin{equation}\label{e:a(0,k)}
a(2^0,k)= \int_\R \widetilde{B}(t)\phi(t-k)dt= B_{H,\lambda}(k), \quad k \in \bbZ,
\end{equation}
(see Stoev et al.\ \cite{stoev:pipiras:taqqu:2002}, proof of Lemma 6.1, or Moulines et al.\ \cite{moulines:roueff:taqqu:2007:JTSA}, p.\ 160; cf.\ Abry and Flandrin \cite{abry:flandrin:1994}, p.\ 33). In other words, the initial sequence, at octave $j = 0$, of approximation coefficients is given by the original time series. To obtain approximation and detail coefficients at coarser scales, we use Mallat's iterative procedure
\begin{equation}\label{e:Mallat}
a(2^{j+1},k) = \sum_{k'\in \mathbb{Z}} u_{k'-2k}a(2^j,k'),\quad d(2^{j+1},k) =\sum_{k'\in \mathbb{Z}}v_{k'-2k} a(2^j,k'), \quad j \in \mathbb{N} \cup \{0\}, \quad k \in \mathbb{Z},
\end{equation}
where the filter sequences $\{ u_k:=2^{-1/2}\int_\bbR \phi(t/2)\phi(t-k)dt \}_{k\in\bbZ}$, $\{v_k:=2^{-1/2}\int_\bbR\psi(t/2)\phi(t-k)dt\}_{k\in\bbZ}$ are called low- and high-pass MRA filters, respectively. Due to the assumed compactness of the supports of $\psi$ and the associated scaling function $\phi$ (see condition \eqref{e:supp_psi=compact}), only a finite number of filter terms is nonzero, which is convenient for computational purposes (Daubechies \cite{daubechies:1992}). {Hereinafter, we assume without loss of generality that $\text{supp}(\phi) = \text{supp}(\psi)=[0,T]$ (cf.\ Moulines et al \cite{moulines:roueff:taqqu:2007:JTSA}, p.\ 160).} Moreover, the wavelet (detail) coefficients $d(2^j,k)$ of tfBm can be expressed as
\begin{equation}\label{e:disc2}
d(2^j,k) = \sum_{\ell\in\bbZ}B_{H,\lambda}(\ell)h_{j,2^jk-\ell},
\end{equation}
where the filter terms are defined by
\begin{equation}\label{e:hj,l}
h_{j,\ell} =2^{-j/2}\int_\bbR\phi(t+\ell)\psi(2^{-j}t)dt.
\end{equation}
If we replace \eqref{e:infinite_sample} with the realistic assumption that only a finite length time series
\begin{equation}\label{e:finite_sample}
\{B_{H,\lambda}(k)\}_{k=1,\hdots,n}
\end{equation}
is available, writing $\widetilde{B}_{H,\lambda}^{(n)}(t) := \sum_{k=1}^{n}B_{H,\lambda}(k)\phi(t-k)$, the finite-sample wavelet coefficients $d^{(n)}(2^j,k)$ of $\widetilde{B}_{H,\lambda}^{(n)}(t) $ satisfy
\begin{equation}\label{e:d^n=d}
d^{(n)}(2^j,k) = d(2^j,k),\qquad \forall (j,k)\in\{(j,k): 2^{-j}T\leq k\leq 2^{-j}(n+1)-T\}.
\end{equation}
 In other words, such subset of finite-sample wavelet coefficients is not affected by the so-named border effect (cf.\ Craigmile et al.\ \cite{craigmile:guttorp:Percival:2005}, Percival and Walden \cite{percival:walden:2006}, Didier and Pipiras \cite{didier:pipiras:2010}). Moreover, by \eqref{e:d^n=d} the number of such coefficients at octave $j$ is given by $n_j = \lfloor 2^{-j}(n+1-T)-T\rfloor$. Hence, $n_j \sim 2^{-j}n$ for large $n$. Thus, for notational simplicity we suppose
 \begin{equation}\label{e:nj=n/2^j}
 n_{j} = \frac{n}{2^j}
 \end{equation}
 holds exactly and only work with wavelet coefficients unaffected by the border effect.

From the above assumptions, it can be shown that, for fixed scales $j$ and $j'$ and some constant $C_{N_\psi} > 0$,
\begin{equation}\label{e:approx_decorrelation}
|\Cov (d(2^j,k),d(2^{j'},k') |\leq C_{N_\psi}2^{(j+j')/2}\frac{|2^jk-2^{j'}k'|^{H-\frac{1}{2}}}{e^{\lambda{|2^jk-2^{j'}k'|}}},
\end{equation}
when $|2^jk-2^{j'}k'|$ is not too small (see Proposition \ref{p:tfbm_decorr}, $(i)$). This is the algebraic expression of the semi-long range dependence phenomenon in the wavelet domain. Let
\begin{equation}\label{e:Hj}
{\mathcal H}_j(x)=\sum_{\ell\in\bbZ}h_{j,\ell}e^{ix \ell} \in \bbC
\end{equation}
be the discrete Fourier transform of the sequence $\{h_{j,\ell}\}_{\ell \in\bbZ}$ appearing in \eqref{e:disc2}. The harmonizable representation of tfBm \eqref{e:tfbm_harmonizable}, expression \eqref{e:disc2} and the periodicity of ${\mathcal H}_j$ imply that
\begin{equation}\label{e:discspec}
\bbE d^2(2^j,k)= \frac{\sigma^2 \Gamma^2\left( H+\frac{1}{2} \right)}{2\pi}\int^{\pi}_{-\pi} \sum_{\ell \in \bbZ}\frac{|{\mathcal H}_j(x)|^2}{|\lambda^2 + (x+2\pi \ell)^2|^{H+\frac{1}{2}}}dx,
\end{equation}
regardless of $k \in \bbZ$ (Proposition \ref{p:wavelet_spectrum}). The wavelet spectrum of tfBm is depicted in Figure \ref{fig:semilrd}, left plot. One consequence of the semi-long range dependence phenomenon is the following. For small $\lambda>0$, the wavelet spectrum of $B_{H,\lambda}(t)$ mimics that of fBm for moderate scales $2^j$, namely, it is approximately linear with slope $2H+1$.  However,
\begin{equation}\label{e:disc_slrd2}
\lim_{j\to\infty}\bbE d^2(2^j,0) = \sum_{\ell\in\bbZ}\frac{\sigma^2 \Gamma^2(H+\frac{1}{2})}{|\lambda^2+(2\pi \ell)^2|^{H+\frac{1}{2}}},
\end{equation}
i.e., it tends to a constant at large octaves, where the transition from linearity to constancy is controlled by the tempering parameter $\lambda>0$ (Proposition \ref{p:wavelet_spectrum}).

\begin{remark}
Assuming a continuous time tfBm sample path is available, its wavelet transform is given by
\begin{equation}\label{e:d(2^j,k)_cont}
d(2^j,k) = \int_{\bbR} 2^{-j/2} \psi(2^{-j}t - k)B_{H,\lambda}(t) dt, \quad j \in \bbN \cup \{0\}, \quad k \in \bbZ.
\end{equation}
The integral in \eqref{e:d(2^j,k)_cont} is well-defined in the $L^2(P)$ sense because the following two conditions are met (Cram\'{e}r and Leadbetter \cite{cramer:leadbetter:1967}, p.\ 86). First, the covariance function $\bbE B_{H,\lambda}(s)B_{H,\lambda}(t)$ is continuous. Second, by an adaptation of the argument in Abry and Didier \cite{abry:didier:2018}, Proposition 3.1, it can also be shown that $\int_\R\int_{\bbR} |\psi(2^{-j}s - k)\psi(2^{-j}t - k) \bbE B_{H,\lambda}(s)B_{H,\lambda}(t)| ds dt < \infty$.

The continuous time wavelet spectrum is
\begin{equation}\label{e:cont_time_wavelet_spec}
\bbE d^2(2^j,k) = \frac{\Gamma^2(H+1/2)}{2\pi} \int_{{\Bbb R}} \frac{\sigma^2}{|\lambda^2 + (\frac{x}{2^j})^2|^{H + \frac{1}{2}}} |\widehat{\psi}(x)|^2 dx,
\end{equation}
regardless of $k$ (Proposition \ref{p:cont_time_wavelet_spec}).
However, unlike with fBm and related processes, the fundamental trait of the (wavelet or Fourier) log-spectrum of tfBm is not its slope over large scales. Hence, estimation based on the continuous time spectrum introduces non-negligible biases. Figure \ref{fig:disc_v_cont} illustrates the difference between the continuous and discrete time wavelet spectra of tfBm.

\begin{figure}[h]
\begin{center}
\includegraphics[width=.40\linewidth]{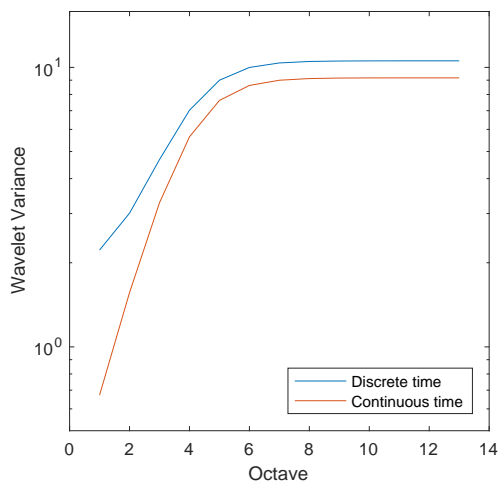}
\end{center}
\caption{\label{fig:disc_v_cont}{\textbf{\boldmath  A comparison of the discrete and continuous time wavelet spectra of tfBm for $H=0.15$, $\lambda=0.3$.}  Unlike with fBm, the discrepancy between the two spectra introduces non-negligible biases in continuous time-based modeling (the same conclusion holds in the Fourier domain).}}
\end{figure}
\end{remark}

\section{Main results}\label{s:main}

\subsection{Estimation and testing}\label{s:asymptotic_theory_discrete}

With classical fBm and related processes, only coarse scale information is generally of interest, which is given by the Hurst parameter $H$. By contrast, with tfBm, fine/moderate scale information is also relevant and is provided by $H$, whereas on coarse scales it is given by both $H$ and $\lambda$ (see \eqref{e:discspec} and \eqref{e:disc_slrd2}). This points to parametric estimation. An $M$-estimator (Van der Vaart \cite{vandervaart:1998}) can be constructed based on the log-wavelet spectrum of tfBm.

\begin{definition}\label{def:estimator}
For a number $n$ of observations of a tfBm, define the sample wavelet variance
\begin{equation}\label{e:W(2j)_tfBm}
{W(2^j) = \frac{1}{n_j}\sum^{n_j}_{k=1}d^2(2^j,k), \quad j = j_1,\hdots,j_m,}
\end{equation}
for some octave range $1 \leq j_1 \leq j_m \leq \lfloor \log_2 n \rfloor$, where the sample wavelet coefficients $d(2^j,k)$ are computed by means of expression \eqref{e:disc2} and $n_j$ is given by \eqref{e:nj=n/2^j}.

Let ${\boldsymbol \theta} = (H,\lambda,\sigma^2)$ be the parameter vector. We define a wavelet-based estimator by means of a weighted nonlinear log-regression
\begin{equation}\label{e:def_estimator}
\widehat{{\boldsymbol \theta}}_n :=\argmin_{{\boldsymbol \theta}}\sum_{j=j_1}^{j_m} w_j \left(\log_2 W(2^j)- \eta_j({\boldsymbol \theta})\right)^2 = \argmin_{{\boldsymbol \theta}} f_n({\boldsymbol \theta}),
\end{equation}
where the octaves $j_1,\ldots,j_m$ are chosen so as to capture both the fine/moderate-scale behavior and the limiting behavior \eqref{e:disc_slrd2}. In \eqref{e:def_estimator}, $\eta_j$ and $w_j$, $j = j_1,\hdots,j_m$, are appropriate choices of log-wavelet spectrum functions and regression weights.
\end{definition}

There are two effects to consider when picking $\eta_j$ and $w_j$ in \eqref{e:def_estimator}. First, estimation based on minimizing the distance between $\log_2 W(2^j)$ and $\log_2 \bbE d^2(2^j,0)$ is biased due to the fact that $\bbE \log_2 d^2(2^j,0)\neq \log_2 \bbE d^2(2^j,0)$. Second, the variance of the sample wavelet variance $W(2^j)$ changes across scales. In view of the near decorrelation property \eqref{e:approx_decorrelation}, we propose using bias-corrected and rescaled wavelet spectrum expressions in \eqref{e:def_estimator} by applying standard approximations to $\bbE \log W(2^j)$ and $\Var \log W(2^j)$ (e.g., Veitch and Abry \cite{veitch:abry:1999}, Wendt et al.\ \cite{wendt:didier:combrexelle:abry:2017}). More precisely, on one hand we set the wavelet spectrum functions to
\begin{equation}\label{e:bias_function}
\eta_j({\boldsymbol \theta}) = \log_2 \bbE d^2(2^j,0) + B(j), \quad j = j_1,\hdots,j_m,
\end{equation}
where
\begin{equation}\label{e:B(j)}
B(j) = \frac{\Psi(n_j/2)}{\log 2} - \log_2 \Big(\frac{n_j}{2}\Big)
\end{equation}
for $\Psi(z) = \Gamma'(z)/\Gamma(z)$, $z > 0$. On the other hand, we choose the weights $w_j = 1/2^{(j-1)/2}$, $j = j_1,\hdots,j_m$.

The consistency and asymptotic normality of the estimator \eqref{e:def_estimator} is established in the following theorem.
\begin{theorem}\label{t:asympt}
Let
\begin{equation}\label{e:theta0_in_intXi}
\Xi = \{(H,\lambda,\sigma^2): \min\{H,\lambda, \sigma^2\}>0\}
\end{equation}
be the parameter space, and let $\bth_0 \in U_0 \subset \Xi$ be the true parameter value, where $U_0$ is a bounded vicinity.
\begin{enumerate}
\item [$(i)$] Suppose the wavelet spectrum $(\bbE_{{\boldsymbol \theta} } W(2^{j_1}),\hdots, \bbE_{{\boldsymbol \theta} } W(2^{j_m}))$ is identifiable for ${\boldsymbol \theta} \in \overline{U}_0$, i.e.,
\begin{equation}\label{e:local_identifiability}
{\boldsymbol \theta} \neq {\boldsymbol \theta}' \Rightarrow (\bbE_{{\boldsymbol \theta} } W(2^{j_1}),\hdots, \bbE_{{\boldsymbol \theta} } W(2^{j_m})) \neq (\bbE_{{\boldsymbol \theta}' } W(2^{j_1}),\hdots, \bbE_{{\boldsymbol \theta}' } W(2^{j_m})).
\end{equation}
Then, for $f_n$ as in \eqref{e:def_estimator}, there is a sequence $\widehat{{\boldsymbol \theta}}_n$ of restricted minima of $f_n$ to $\overline{U}_0$ such that
\begin{equation}\label{e:consistency}
\widehat{{\boldsymbol \theta}}_n \stackrel{P}\rightarrow {\boldsymbol \theta}_0, \quad n \rightarrow \infty;
\end{equation}
\item [$(ii)$] if, in addition to the above,
\begin{equation}\label{e:det>0}
\det \left(\sum_{j=j_1}^{j_m} \frac{w_j\partial_\ell\Edj\partial_i\Edj }{[\Edj]^2} \right)_{i,\ell=1,2,3}> 0,
\end{equation}
then
\begin{equation}\label{e:asympt_normality}
\sqrt{n}(\widehat{{\boldsymbol \theta}}_n - {\boldsymbol \theta}_0) \stackrel{d}\rightarrow {\mathcal N}(0,\Sigma({\boldsymbol \theta}_0)), \quad n \rightarrow \infty,
\end{equation}
for some symmetric positive semidefinite matrix $\Sigma({\boldsymbol \theta_0})$.
\end{enumerate}
\end{theorem}

\begin{remark}\label{r:identifiable}
Note that the wavelet spectrum (variance) $\E_{{\boldsymbol \theta}} W(2^j)$ is a function of the octave $j$ and of the parameter vector ${\boldsymbol \theta} =(H,\lambda,\sigma^2)$. The identifiability condition \eqref{e:local_identifiability} means that there is no pair of parameter vectors ${\boldsymbol \theta}_0$ and ${\boldsymbol \theta}_1$ for which $\E_{{\boldsymbol \theta}_0} W(2^j)  = \E_{{\boldsymbol \theta}_1} W(2^j)$ over the available octave range. Computational studies indicate that \eqref{e:local_identifiability} is mild and generally satisfied in practice (see, for example, Figure \ref{fig:id_surf}).

In turn, condition \eqref{e:det>0} is quite mild, and it amounts to requiring the full rank of a sum of rank 1 terms. Note that, for fixed $j$, the corresponding matrix in the sum in \eqref{e:det>0} can be written as $u_ju_j^\top$, where $u_j = c_j\left(\partial_1 \Edj,\partial_2 \Edj,\partial_3 \Edj\right)^\top$ for some constant $c_j$. So long as the vectors $u_j$ are linearly independent,  \eqref{e:det>0} holds (see Appendix D in Abry et al.\ \cite{frecon:didier:pustelnik:abry:2016} for a more detailed discussion in a similar context).
\end{remark}

\begin{figure}[h!]
\centering
\includegraphics[width=\linewidth]{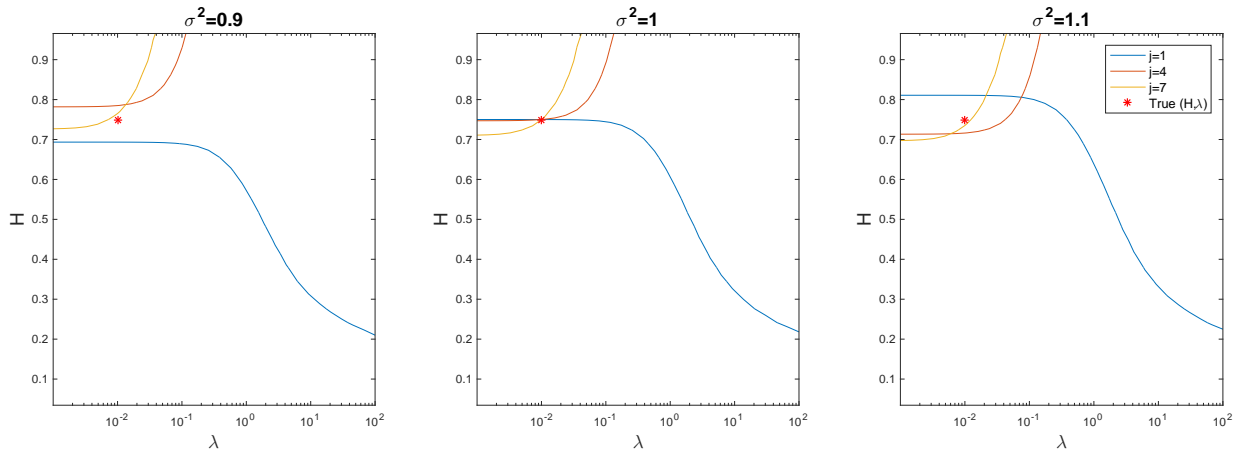}
\caption{\textbf{Wavelet domain identifiability.} {For a fixed, true parameter vector $\bth_0 =(H_0,\lambda_0,\sigma_0^2)= (.75,.01,1)$, each plot is associated with one value of $\sigma^2\in\{0.9,1.0,1.1\}$. The lines are the level curves of $\E_{\bth} W(2^j)- \E_{\bth_0} W(2^j)=0$ as a function of $(H,\lambda)$ for $j=1,4,7$. In the middle plot, associated with the true value $\sigma^2_0 = 1$, the three level curves only intersect at the true pair $(H_0,\lambda_0) = (.75,.01)$. Moreover, the intersection of the level curves in the first and third plots is empty. Hence, the inclusion of $m\geq 3$ octaves guarantees the intersection of the level surfaces $\cap_{j=j_1}^{j_m}\{\bth: \E_{\bth}W(2^j)- \E_{\bth_0} W(2^j)=0\}$ is empty for $\bth\in V\setminus\{\bth_0\}$, i.e., condition \eqref{e:local_identifiability} is satisfied locally.}}\label{fig:id_surf}
\end{figure}

In practice, it is of interest to test whether a sample path comes from a fBm as opposed to a tfBm alternative, which can be naturally framed as
\begin{equation}\label{e:H0}
H_0: \textnormal{ the process is a fBm} \quad \textnormal{vs} \quad H_1: \textnormal{ the process is a tfBm}.
\end{equation}
If the null hypothesis of a fBm is rejected, then estimation of tfBm can be conducted by means of the estimator \eqref{e:def_estimator}. A simple test consists of comparing Hurst exponent estimates obtained over different octave ranges. For fBm, the estimates should be approximately constant, whereas for tfBm, they should significantly differ assuming $\lambda$ is not too small.

As to maximize test power, we propose to estimate the Hurst exponent over a fixed pair of small octaves and compare it with another estimate obtained over a range of large octaves. In regard to the latter estimate, note that for both fBm and tfBm the slope of the log-spectrum stabilizes in the limit. Hence, a traditional wavelet log-regression procedure is suitable. However, over small scales, either for fBm or tfBm, there is significant discrepancy between discrete and continuous time spectra (see Figure \ref{fig:disc_v_cont}). For this reason, over small scales, we propose to parametrically fit the discrete time spectrum by means of wavelet-based $M$-estimation.

Starting from the null hypothesis of a fBm, one can define an $M$-estimator $(\widehat{H},\widehat{\sigma^2})$ of $(H,\sigma^2)$ over (small) octaves $j_1,\ldots,j_m$ via
\begin{equation}\label{eq:def_Mest_fbm}
(\widehat{H}(2^{j_1},\ldots,2^{j_m}),\widehat{\sigma^2}(2^{j_1},\ldots,2^{j_m})):= \argmin_{H,\sigma^2}\sum_{j=j_1}^{j_m}\left(\log_2W(2^j) - \widetilde{\eta}_j(H,\sigma^2)\right)^2.
\end{equation}
In \eqref{eq:def_Mest_fbm}, we set
\begin{equation}\label{e:eta-tilde}
\widetilde{\eta}_j(H,\sigma^2) =\log_2(\sigma^2/C^2(H)) + \log_2\int_{-\pi}^\pi \sum_{\ell\in\Z}\frac{|\mathcal{H}_j(x)|^2}{|x+2\pi \ell|^{2H+1}}dx + B(j),
\end{equation}
where $B(j)$ is given by \eqref{e:B(j)} and $C(H) = \sqrt{\pi/\left( H\Gamma(2H)\sin(\pi H)\right)}$. In other words, $\widetilde{\eta}_j$ is the bias-corrected log-spectrum function of a fBm observed in discrete time. Moreover, for large scales, let $\widetilde{H}(a(n)2^{j_3},a(n)2^{j_4})$ be a traditional wavelet log-regression estimator
\begin{equation}\label{eq:def_wavereg}
\widetilde{H}(a(n)2^{j_3},a(n)2^{j_4}):= \frac{1}{2}\left(\sum^{j_4}_{j=j_3} \varpi_{j} \log_2 W(a(n)2^j) - 1\right), \quad 0 < j_3 < j_4 \leq \lfloor \log_2 n \rfloor.
\end{equation}
In \eqref{eq:def_wavereg}, the linear regression weights satisfy
\begin{equation}\label{e:sum-varpi=0,sum-j*varpij=1}
\sum^{j_4}_{j=j_3} \varpi_j = 0, \quad \sum^{j_4}_{j=j_3}j \varpi_j = 1
\end{equation}
(Veitch and Abry \cite{veitch:abry:1999}, Abry et al.\ \cite{abry:flandrin:taqqu:veitch:2000}, Bardet \cite{bardet:2002}, Stoev et al.\ \cite{stoev:pipiras:taqqu:2002}) and $a(n)$ is a slow dyadic scaling factor (see \eqref{e:seq_decay_assumption}). Consider the test statistic defined by
\begin{equation}\label{e:test_stat}
T_n = \widehat{H}(2^{j_1},2^{j_2})-\widetilde{H}(a(n)2^{j_3},a(n)2^{j_4}), \quad j_1 < j_2< j_3 < j_4,
\end{equation}
where large values of $T_n$ are associated with evidence against $H_0$ in \eqref{e:test_stat}. In the following proposition, the asymptotic distribution of $T_n$ is established under both $H_0$ and $H_1$. {Under $H_1$, the observed stochastic process is a tfBm. Hence, the objective function \eqref{eq:def_Mest_fbm} is misspecified. So, we make the additional assumption that $j_1,j_2$ are chosen such that
\begin{equation}\label{e:slope_assumption}
\inf_{H\in(0,1)}\frac{\E_{H} d_{\textnormal{fBm}}^2(2^{j_2},0)}{\E_{H}d_{\textnormal{fBm}}^2(2^{j_1},0)}\leq \frac{\E_{H,\lambda} d_{\textnormal{tfBm}}^2(2^{j_2},0)}{\E_{H,\lambda}d_{\textnormal{tfBm}}^2(2^{j_1},0)} \leq\frac{\E_{H} d_{\textnormal{fBm}}^2(2^{j_2},0)}{\E_{H}d_{\textnormal{fBm}}^2(2^{j_1},0)}.
\end{equation}
Condition \eqref{e:slope_assumption} ensures that there is a pair
\begin{equation}\label{e:(H',sigma'^2)}
(H',\sigma'^2)
\end{equation}
in the fBm parametrization such that
\begin{equation}\label{e:Ed2_fbm=Ed2_tfBm}
\log_2\E_{H',\sigma'^2} d^2_{\text{fBm}}(2^j,0) = \log_2\E_{\bth_0}d^2_{\text{tfBm}}(2^j,0), \quad j=j_1,j_2,
\end{equation}
where $\bth_0=(H,\lambda,\sigma^2)$ is the true parameter vector of the underlying tfBm (see Lemma \ref{Hest_AN} and also Remark \ref{r:on_conditions_for_HT}).}
\begin{theorem}\label{t:test}
Consider the hypotheses \eqref{e:H0} and let $T_n$ be the test statistic defined by \eqref{e:test_stat}.  Assume that the condition \eqref{e:Npsi>alpha} holds (for $\alpha$ as in \eqref{e:psihat_is_slower_than_a_power_function}), and that the dyadic scaling factor $a(n)$ satisfies
\begin{equation}\label{e:seq_decay_assumption}
\frac{a(n)}{n} + \frac{n}{a(n)^{1 + 2 \beta}} \to 0, \quad n \rightarrow \infty,
\end{equation}
for some $\beta \in (1,\min\{\alpha - \varepsilon_0,2\}]$ and some small $\varepsilon_0 > 0$.
\begin{enumerate}
\item [$(i)$] Under $H_0$, {suppose in addition that the rank condition
\begin{equation}\label{e:det>0_hyptest}
\det\left(\sum_{j=j_1,j_2} \frac{w_j\partial_\ell\E_{\bth} d_{\textnormal{fBm}}^2(2^j,0)\partial_i\E_{\bth} d_{\textnormal{fBm}}^2(2^j,0) }{[\E_{\bth} d_{\textnormal{fBm}}^2(2^j,0)]^2}\right)_{i,\ell={1,2}}> 0
\end{equation}
holds at $\bth=(H,\sigma^2)$.} Then,
\begin{equation}\label{e:sqrt(n)T->N(0,tau^2)}
\sqrt{n} \hspace{1mm}T_n \stackrel{d}\rightarrow {\mathcal N}(0,\tau^2_0(H)), \quad n \rightarrow \infty,
\end{equation}
for some $\tau^2_0(H) > 0$;
\item [$(ii)$] under $H_1$ {and \eqref{e:slope_assumption}, suppose in addition that \eqref{e:det>0_hyptest} holds at $\bth=(H',\sigma'^2)$ as in \eqref{e:(H',sigma'^2)}.} Then, there are $\vartheta \in (\frac{1}{2},\frac{3}{2})$ and $\tau^2_1(H,\lambda) > 0$ such that
\begin{equation}\label{e:sqrt(n)(T-vartheta)->N(0,tau^2)}
\sqrt{n} \hspace{1mm}(T_n - \vartheta) \stackrel{d}\rightarrow {\mathcal N}(0,\tau^2_1(H,\lambda)), \quad n \rightarrow \infty,
\end{equation}
for some $\tau^2_1(H,\lambda) > 0$.
\end{enumerate}
\end{theorem}
The weak convergence \eqref{e:sqrt(n)T->N(0,tau^2)} suggests the asymptotically valid rejection region
\begin{equation}\label{e:Ralpha}
R_{\epsilon}: T_n > z_{\epsilon}\frac{\tau_0(H)}{\sqrt{n}}
\end{equation}
at level $0 < \epsilon < 1$ (see Section \ref{s:MC}, \textbf{Hypothesis testing}, on finding the standard error).

\begin{remark}\label{r:on_conditions_for_HT}
Condition \eqref{e:det>0_hyptest}, akin to \eqref{e:det>0}, is necessary to establish the asymptotic distribution of $T_n$ under each hypothesis, at different values of $\bth$. Assumption \eqref{e:slope_assumption} is mild and easily met in practice (i.e., for physically relevant tfBm parameter values) for $j_1=1$, $j_2=2$. The lefthand inequality states that the initial slope of the log-wavelet spectrum of tfBm is comparable to that of the whole class of fBm (any $H \in (0,1)$). This ensures the tempering parameter $\lambda$ is not too large; otherwise, the log-wavelet spectrum of tfBm quickly flattens as a function of the octave $j$, which makes tfBm and fBm very distinguishable from one another. Moreover, the lefthand inequality in \eqref{e:slope_assumption} guarantees a sequence of minima in \eqref{eq:def_Mest_fbm} exists. On the other hand, as a result of tempering, for any fixed $j_1$ there always exists a $j_2$, usually small, so that the righthand inequality is satisfied. This is a consequence of the fact that the slope of the log-wavelet spectrum of tfBm converges to zero (cf.\ Figure \ref{fig:semilrd}, left plot).
\end{remark}

\begin{remark}
Picking the scales $2^{j_1}$, $2^{j_2}$, $a(n)2^{j_3}$, $a(n)2^{j_4}$ is simple in practice. Computational experiments confirm that setting $j_1 = 1$ and $j_2 = 2$ is a good choice, in view of the large number of terms that go into the sample wavelet variances $W(2^{1})$ and $W(2^{2})$. On the other hand, based on related work on the wavelet analysis of fractional processes, the choice of scales for the wavelet regression component of the test statistic \eqref{e:test_stat} is well-understood. An example of a scaling sequence satisfying \eqref{e:seq_decay_assumption} for large enough $n$ is
$$
a(n) := 2 \lfloor n^{\mu} \rfloor, \quad \frac{1}{1+2 \beta} < \mu < 1.
$$
In other words, a low parameter value $H$ implies that $a(n)$ must grow slowly by comparison to $n$. For a fixed octave range $\{j_3,\hdots,j_4\}$ associated with an initial scaling factor value $a(n_0) =1$ and sample size $n_0$, define the scale range $\{2^{j_3(n)},\hdots,2^{j_4(n)}\} = \{a(n)2^{j_3}, \hdots,a(n)2^{j_4}\}$ for a general sample size $n$. Then, under \eqref{e:seq_decay_assumption}, for every $n$ the range of useful octaves is constant and given by $j_4(n) - j_3(n) = j_4 - j_3$, where the new octaves are $j_\ell(n) = \log_2 a(n) + j_\ell \in \bbN$, $\ell = 3,4$. Selecting $\mu $ close to $1$ leads to using the coarsest available range of scales, hence favoring low bias in the estimation of the Hurst eigenvalues at the price of a larger variance. In contrast, picking $\mu$ close to the lower bound $1/(1+2 \beta)$ minimizes the variance at the cost of a larger bias. For further discussion of the choice of $j_3$ and $j_4$ and the scaling sequence $a(n)$ in practice, see, for instance, Wendt et al.\ \cite{wendt:didier:combrexelle:abry:2017}, Remark 3.1, or Abry and Didier \cite{abry:didier:2018}, Remark 4.2.
\end{remark}

\subsection{Monte Carlo studies}\label{s:MC}


\noindent \textbf{Numerical experiment setting.} The performance of the wavelet estimator $\widehat{{\boldsymbol \theta}} = (\widehat{H},\widehat{\lambda},\widehat{\sigma^2})$ was assessed through Monte Carlo studies over 5000 independent realizations of tfBm paths. The chosen parameter values were taken from $(H,\lambda, \sigma^2) \in \{0.15,0.35,0.65,0.85\} \times \{0.001,0.01,0.1,1\} \times \{1,10\}$. In all cases, each tfBm path was generated by circulant matrix embedding using tfGn (Davies and Harte \cite{davies:harte:1987}, Wood and Chan \cite{wood:chan:1994}). The analysis was conducted using orthogonal least asymmetric Daubechies wavelets computed through Mallat's pyramidal algorithm \eqref{e:Mallat}. Unless otherwise indicated, we set $N_\psi=2$ in computational studies.


To implement the discrete time wavelet variances \eqref{e:discspec} appearing in the objective function \eqref{e:def_estimator}, several numerical considerations were made. First, the filter sequences $\{h_{j,\ell}\}_{\ell \in\Z}$ require computation and numerical integration of the wavelet and scaling functions $\psi$ and $\phi$. However, this needs only be done once, and the computed filters can be stored.

Second, for larger $j$, the filters $\mathcal{H}_j(x)$ have a larger number of non-vanishing Fourier coefficients (related to high frequency behavior) and so for increasing $j$, successively finer partitions $P(j)$ of $[-\pi,\pi]$ for the numerical integration of the wavelet spectrum \eqref{e:discspec} were implemented. Unreported computational experiments suggested that a mesh size $\|P(j)\|=2\pi/2^{\max\{10,j-3\}}$ was sufficiently large to yield negligible quadrature error, which was used for this study. Additionally, the infinite sum appearing in the integrand \eqref{e:discspec} was truncated at $|\ell| \leq 10$. To approximate truncated terms, noting that for $K$ sufficiently large we have  $\sum_{|\ell|>K}{|\lambda^2 + (x+2\pi \ell)^2|^{-(H+\frac{1}{2})}}\approx \sum_{|\ell|>K}{|\lambda^2 + (2\pi \ell)^2|^{-(H+\frac{1}{2})}}$ for all $x\in [-\pi,\pi]$, we used the integral ${2\int_{10}^\infty \Gamma^2(H+\frac{1}{2})(2\pi)^{-1}|\lambda^2+(2\pi \ell)^2|^{-(H+\frac{1}{2})} d\ell}$. Minimization of the objective function \eqref{e:def_estimator} was implemented using the Matlab routine \texttt{fmincon} with constraints $0 < H < 1$, $0 < \lambda < 5$.  In all cases the initial values were set to $(\widehat{H}_{\textnormal{init}},\widehat{\lambda}_{\textnormal{init}}) = (1/2,0.03)$. With regards to the octaves $j_1,\ldots,j_m$ in \eqref{e:def_estimator}, the largest octave $j_m$ is especially important in determining the lower bound of possible $\lambda$ that can be estimated; larger values of $j_m$ correspond to improved estimation of increasingly smaller $\lambda$ (cf.\ Figure \ref{fig:prob1_6_2}, left plot). However, due to the aforementioned increasingly finer partitions $P(j)$, the inclusion of larger values of $j_m$ adds to the computational cost of the optimization. Thus, for each value of $n$, we set $j_m=\min\{j_{\max},12\}$, where $j_{\max}$ is the largest available octave in the sample. 
In any case, all the remaining octaves are given by $(j_1,j_2,\ldots,j_{m-1})=(1,2,\ldots, j_{m}-1)$.\\

\noindent \textbf{Bias correction and number of vanishing moments.} In Table \ref{MC_212_tables}, one can see that the estimator performance is, in general, overall best with both bias correction and $N_\psi = 2$. Similarly to the classical fBm case (cf.\ Abry et al.\ \cite{abry:flandrin:taqqu:veitch:2003}), picking $N_\psi = 2$ instead of $N_\psi = 1$ is more crucial when $H>1/2$ and small $\lambda$. This is due to the fact that, for these parameter values, tfBm behaves like fBm with long memory over a large number octaves. Unreported numerical results indicated that estimation with $N_\psi > 2$ tends to diminish the estimation accuracy for most parameter values, especially for small $\lambda$.\\

	\begin{table}[tb]
	 \begin{minipage}{.5\linewidth}
	   \centering
	   \resizebox{\linewidth}{!}{%
\begin{tabular}{|l|c|c|c|c|}
\hline
\multicolumn{5}{|c|}{$N_\psi=1$, no bias correction}\\\hline&$H=0.15$&$H=0.35$&$H=0.65$&$H=0.85$\\\hline
\multirow{2}{*}{$\lambda=0.001$}&(.1511, .0032)&(.3542, .0025)&(.6671, .0023)&(.8837, .0023)\\
&(.0041, .0032)&(.0156, .0024)&(.0232, .0017)&(.0229, .0013)\\\hline
\multirow{2}{*}{$\lambda= 0.01$}&(.1507, .0134)&(.3533, .0127)&(.6585, .0123)&(.8599, .0119)\\
&(.0035, .0052)&(.0130, .0041)&(.0286, .0037)&(.0264, .0031)\\\hline
\multirow{2}{*}{$\lambda=  0.1$}&(.1498, .1103)&(.3499, .1067)&(.6642, .1076)&(.8824, .1111)\\
&(.0034, .0159)&(.0080, .0107)&(.0330, .0119)&(.0622, .0173)\\\hline
\multirow{2}{*}{$\lambda=    1$}&(.1480, 1.122)&(.3459, 1.046)&(.6421, 1.024)&(.8387, 1.019)\\
&(.0053, .2161)&(.0103, .0791)&(.0186, .0434)&(.0260, .0325)\\\hline
\end{tabular}

}\\ \ \\
\resizebox{\linewidth}{!}{%
\begin{tabular}{|l|c|c|c|c|}
\hline
\multicolumn{5}{|c|}{$N_\psi=1$, with bias correction}\\\hline&$H=0.15$&$H=0.35$&$H=0.65$&$H=0.85$\\\hline
\multirow{2}{*}{$\lambda=0.001$}&(.1521, .0018)&(.3570, .0015)&(.6674, .0015)&(.8822, .0017)\\
&(.0047, .0025)&(.0165, .0020)&(.0231, .0015)&(.0230, .0012)\\\hline
\multirow{2}{*}{$\lambda= 0.01$}&(.1502, .0103)&(.3504, .0103)&(.6508, .0103)&(.8517, .0102)\\
&(.0036, .0047)&(.0136, .0039)&(.0282, .0034)&(.0259, .0028)\\\hline
\multirow{2}{*}{$\lambda=  0.1$}&(.1502, .1007)&(.3499, .1001)&(.6517, .1005)&(.8538, .1013)\\
&(.0034, .0153)&(.0079, .0105)&(.0313, .0116)&(.0618, .0169)\\\hline
\multirow{2}{*}{$\lambda=    1$}&(.1502, 1.022)&(.3503, 1.004)&(.6502, 1.000)&(.8501, 1.001)\\
&(.0051, .1739)&(.0103, .0738)&(.0188, .0415)&(.0264, .0312)\\\hline
\end{tabular}
}\\
\end{minipage}
	 \begin{minipage}{.5\linewidth}
	   \centering
	   \resizebox{\linewidth}{!}{%
\begin{tabular}{|l|c|c|c|c|}
\hline
\multicolumn{5}{|c|}{$N_\psi=2$, no bias correction}\\\hline&$H=0.15$&$H=0.35$&$H=0.65$&$H=0.85$\\\hline
\multirow{2}{*}{$\lambda=0.001$}&(.1513, .0054)&(.3528, .0043)&(.6538, .0032)&(.8538, .0028)\\
&(.0037, .0060)&(.0112, .0045)&(.0207, .0035)&(.0204, .0029)\\\hline
\multirow{2}{*}{$\lambda= 0.01$}&(.1506, .0144)&(.3516, .0133)&(.6537, .0125)&(.8541, .0122)\\
&(.0033, .0077)&(.0104, .0059)&(.0228, .0048)&(.0242, .0041)\\\hline
\multirow{2}{*}{$\lambda=  0.1$}&(.1502, .1105)&(.3505, .1075)&(.6554, .1066)&(.8631, .1072)\\
&(.0031, .0176)&(.0070, .0130)&(.0180, .0110)&(.0352, .0131)\\\hline
\multirow{2}{*}{$\lambda=    1$}&(.1487, 1.102)&(.3468, 1.042)&(.6443, 1.023)&(.8425, 1.017)\\
&(.0047, .193)&(.0091, .0756)&(.0161, .0432)&(.0218, .0328)\\\hline
\end{tabular}
}\\ \ \\
	   \resizebox{\linewidth}{!}{%
\begin{tabular}{|l|c|c|c|c|}
\hline
\multicolumn{5}{|c|}{$N_\psi=2$, with bias correction}\\\hline&$H=0.15$&$H=0.35$&$H=0.65$&$H=0.85$\\\hline
\multirow{2}{*}{$\lambda=0.001$}&(.1523, .0031)&(.3564, .0025)&(.6584, .0020)&(.8576, .0017)\\
&(.0041, .0048)&(.0124, .0037)&(.0206, .0030)&(.0199, .0025)\\\hline
\multirow{2}{*}{$\lambda= 0.01$}&(.1503, .0102)&(.3499, .0101)&(.6495, .0101)&(.8494, .0101)\\
&(.0034, .0073)&(.0110, .0057)&(.0231, .0046)&(.0244, .0039)\\\hline
\multirow{2}{*}{$\lambda=  0.1$}&(.1504, .1003)&(.3501, .1002)&(.6500, .1003)&(.8504, .1002)\\
&(.0031, .0170)&(.0070, .0128)&(.0184, .0111)&(.0349, .0129)\\\hline
\multirow{2}{*}{$\lambda=    1$}&(.1504, 1.016)&(.3502, 1.004)&(.6502, 1.001)&(.8505, 1.000)\\
&(.0046, .1648)&(.0090, .0715)&(.0161, .0417)&(.0220, .0318)\\\hline
\end{tabular}
}
\end{minipage}
\caption{ \textbf{\boldmath  Bias correction and $N_{\psi}$.} Each of the entries in the above tables corresponds to the Monte-Carlo averages (top row of each cell) and standard deviations (bottom row of each cell), based on 5000 independent realizations per $(H,\lambda)$ pair with $\sigma^2 = 1$ and $n=2^{12}$. The best estimation performance is given with $N_\psi$ and bias correction. The table also illustrates the fact that smaller values of $\lambda$ require larger sample sizes. In particular, at the sample size $n=2^{12}$, the instance $\lambda = 0.001$ is not yet estimable, though it does become estimable at larger sample sizes (see Boniece et al.\ \protect\cite{boniece:sabzikar:didier:2018}). See also Table \ref{table:pwr} on the related issue of test power.}\label{MC_212_tables}
\end{table}

\noindent \textbf{Bias and standard deviation.} Figure \ref{fig:skew_kurt_bias_plots} shows that, for all estimator entries $\widehat{{\boldsymbol \theta}}_n = (\widehat{H},\widehat{\lambda},\widehat{\sigma^2})$, the bias becomes negligible as the sample size grows. This illustrates the estimator's consistency. The plots further show that standard deviations for all estimator entries decrease as $n^{-1/2}$ (the latter trend being plotted as superimposed red dashed lines), which illustrates the theoretical $\sqrt{n}$ convergence rate to normality.\\

\noindent \textbf{Asymptotic normality.} Figure \ref{fig:skew_kurt_bias_plots} displays the skewness and (excess) kurtosis of the finite sample distribution
of the estimator vector $\widehat{{\boldsymbol \theta}}_n = (\widehat{H}, \widehat{\lambda}, \widehat{\sigma^2})$. Both measures decrease as the sample size increases. Moreover, the plots provide a measure of the sample sizes needed for an accurate Gaussian approximation to the distribution of each entry of $\widehat{{\boldsymbol \theta}}_n$. In particular, it can be seen that normality is reached faster (i.e., at smaller sample sizes) for $\widehat{H}$ than for $\widehat{\lambda}$ and $\widehat{\sigma^2}$.\\

\noindent \textbf{Computational cost (comparison with maximum likelihood estimation).} To gauge the computational cost of the wavelet estimator \eqref{e:def_estimator}, we implemented maximum likelihood estimation (of tfGn) in \texttt{Matlab} using the minimization routine \texttt{fmincon}. Unsurprisingly, maximum likelihood is extremely computationally inefficient, clocking in at roughly 15 minutes 23 seconds per estimate at the moderate sample size $2^{12}$. By contrast, the wavelet method takes 3.86 seconds on average per estimate. Implementing maximum likelihood-based estimation beyond this sample size quickly becomes infeasible, whereas the proposed wavelet method can be used on large sample sizes at moderate computational cost. For example, at $n=2^{20}$ computing the wavelet estimator takes on average $10.34$ seconds per estimate.

\begin{remark}
For a comparative simulation study of the proposed wavelet and the Whittle estimators both in terms of statistical and computational performance, see Boniece et al.\ \cite{boniece:sabzikar:didier:2018}. 
\end{remark}

\begin{figure}[!htb]
    \centering
    \begin{minipage}{.24\linewidth}
        \centering
        \includegraphics[width=\linewidth]{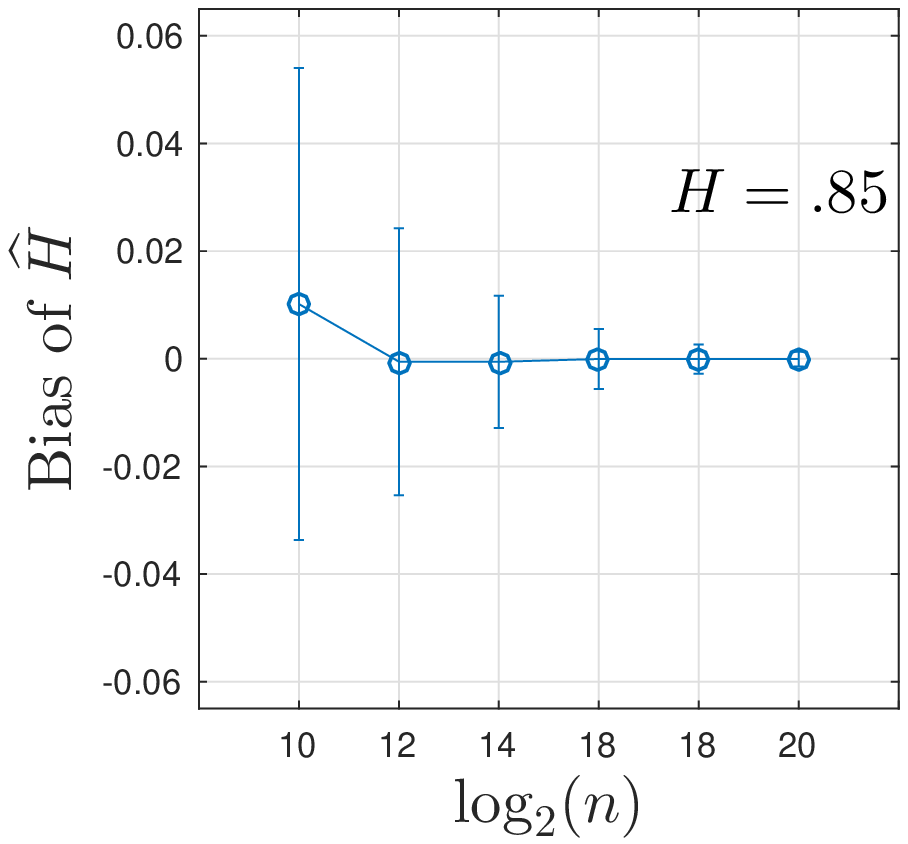}
    \end{minipage}
    \begin{minipage}{.24\linewidth}
        \centering
        \includegraphics[width=\linewidth]{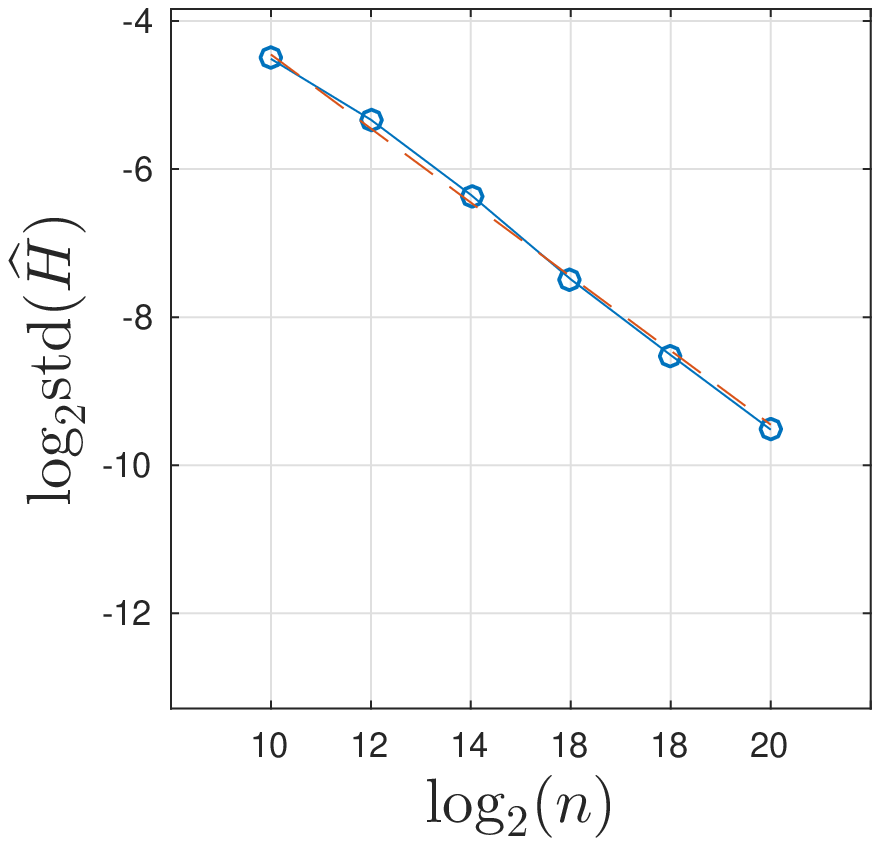}
    \end{minipage}
    \begin{minipage}{.24\linewidth}
        \centering
        \includegraphics[width=\linewidth]{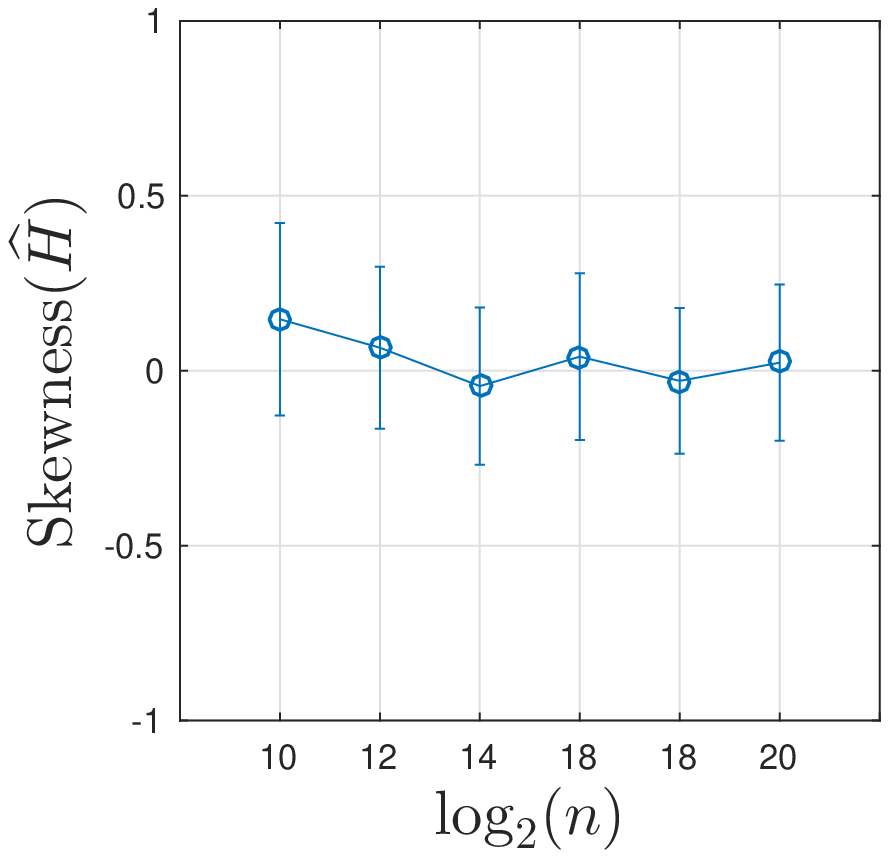}
    \end{minipage}
    \begin{minipage}{.24\linewidth}
        \centering
        \includegraphics[width=\linewidth]{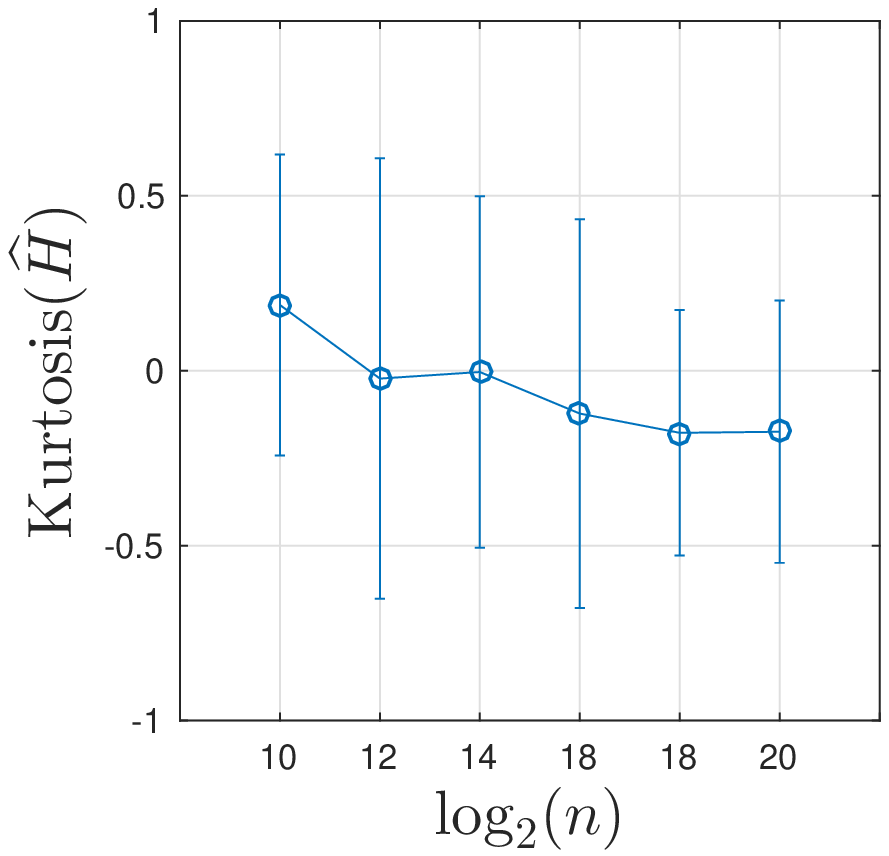}
    \end{minipage}
    \begin{minipage}{.24\linewidth}
        \centering
        \includegraphics[width=\linewidth]{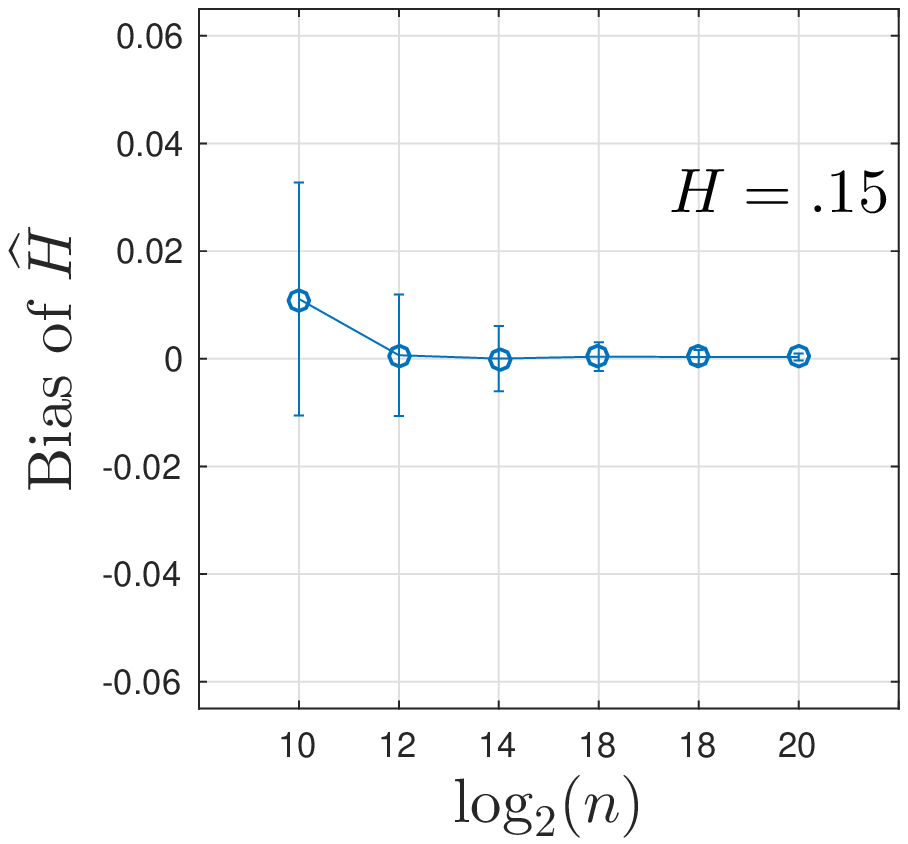}
    \end{minipage}
    \begin{minipage}{.24\linewidth}
        \centering
        \includegraphics[width=\linewidth]{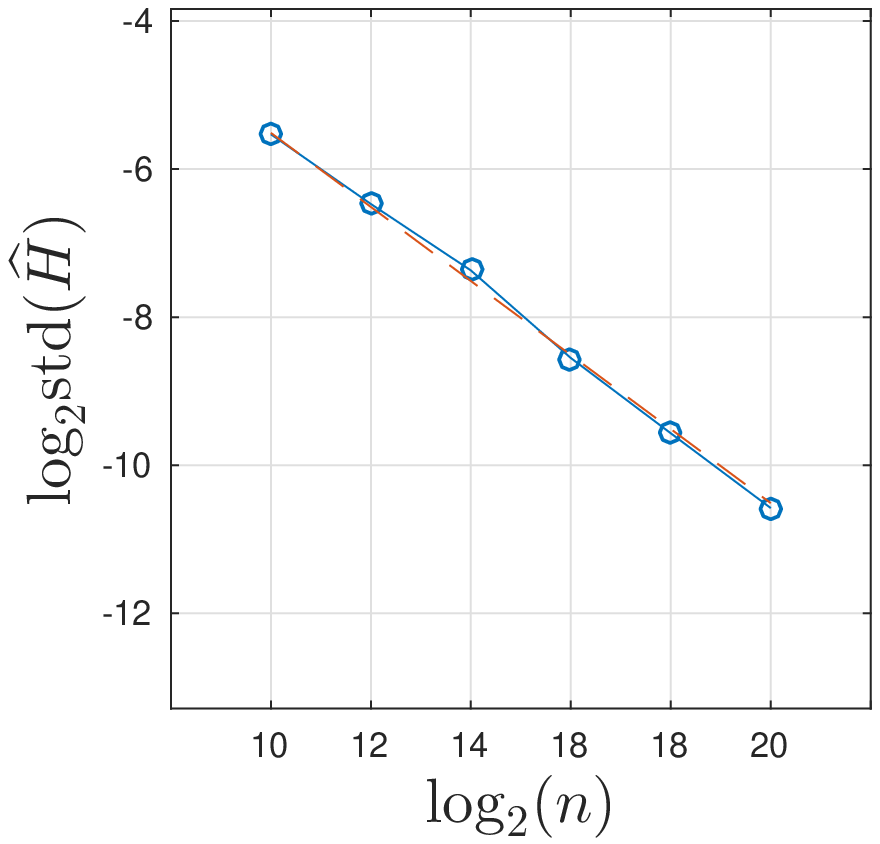}
    \end{minipage}
    \begin{minipage}{.24\linewidth}
        \centering
        \includegraphics[width=\linewidth]{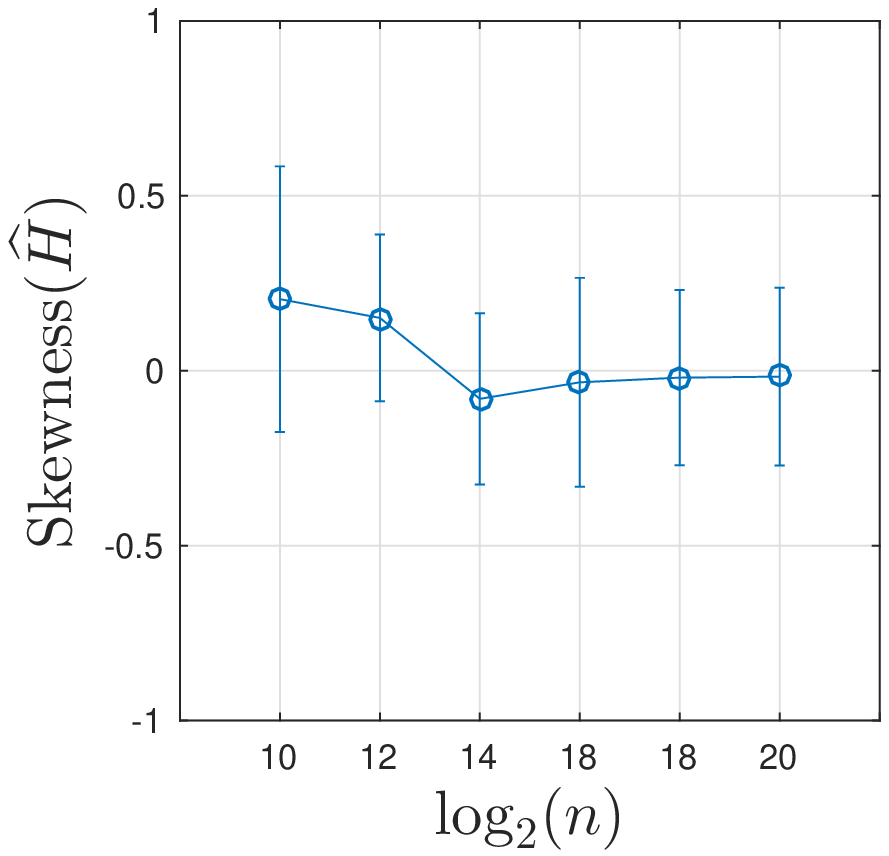}
    \end{minipage}
    \begin{minipage}{.24\linewidth}
        \centering
        \includegraphics[width=\linewidth]{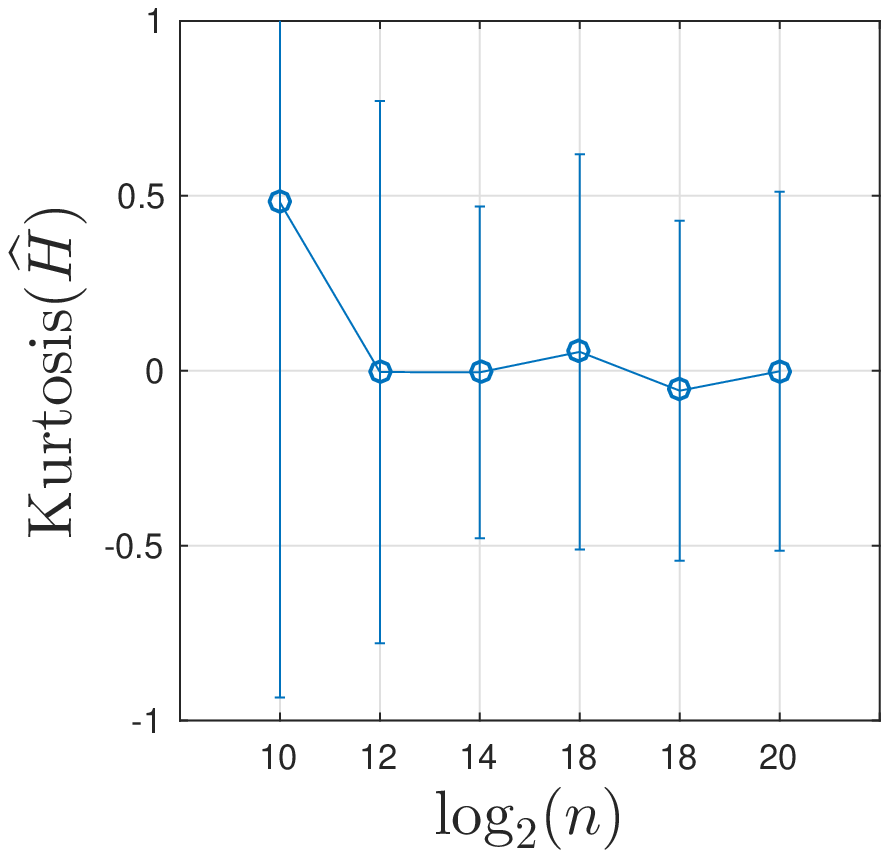}
    \end{minipage}
    \begin{minipage}{.24\linewidth}
        \centering
        \includegraphics[width=\linewidth]{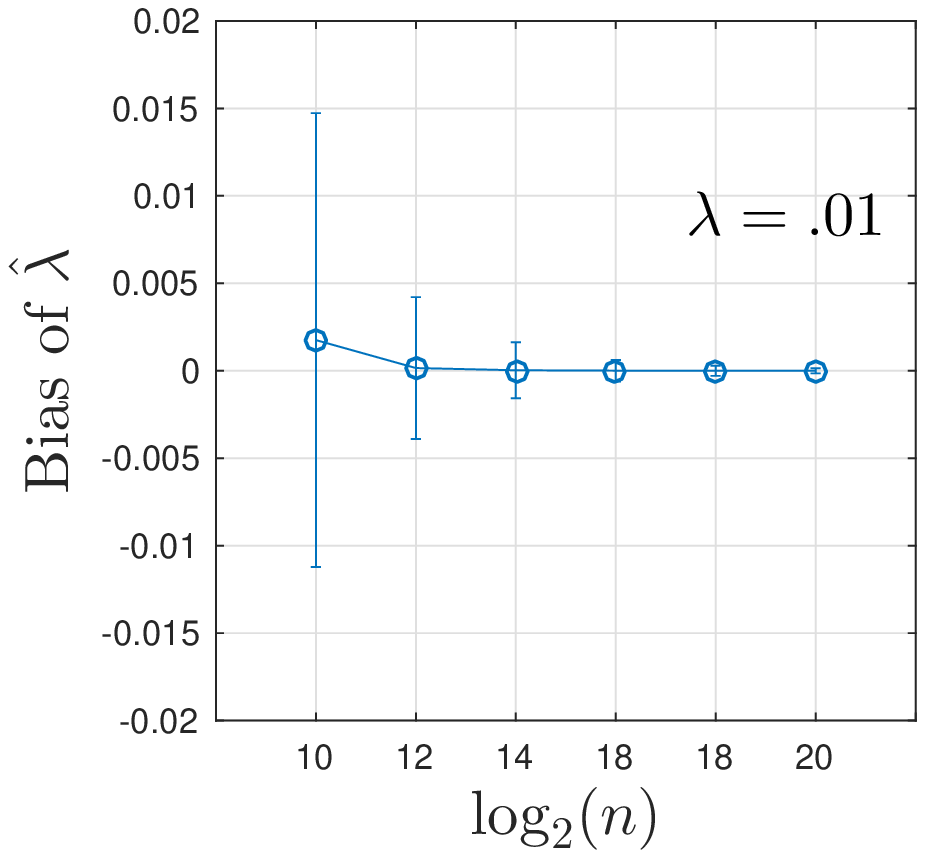}
    \end{minipage}
    \begin{minipage}{.24\linewidth}
        \centering
        \includegraphics[width=\linewidth]{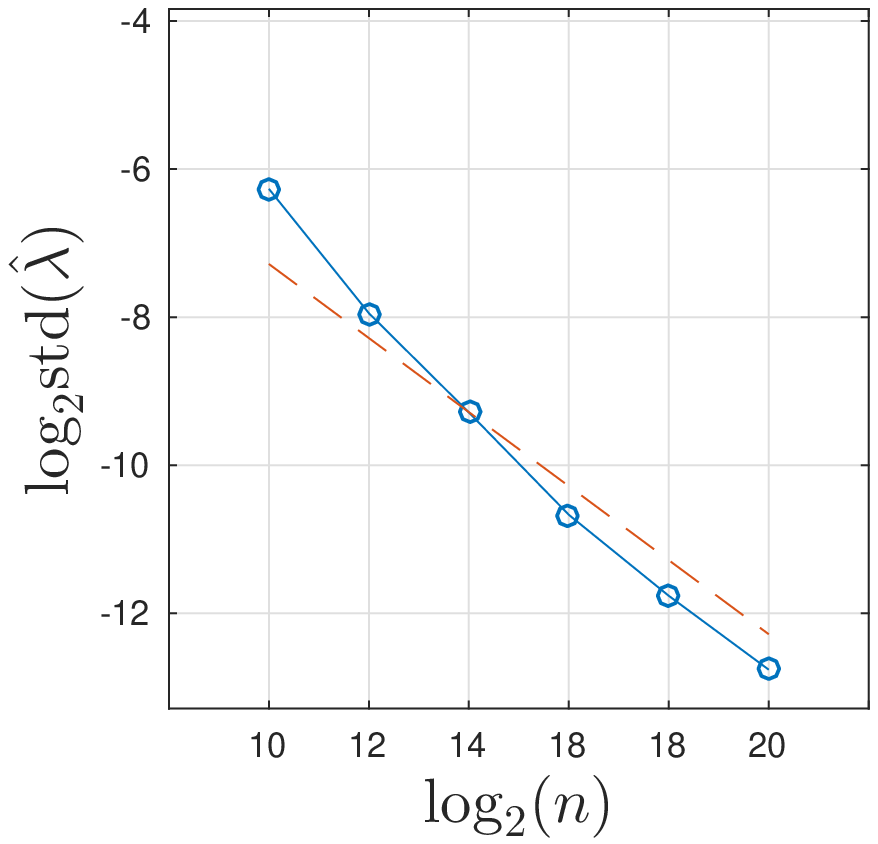}
    \end{minipage}
    \begin{minipage}{.24\linewidth}
        \centering
        \includegraphics[width=\linewidth]{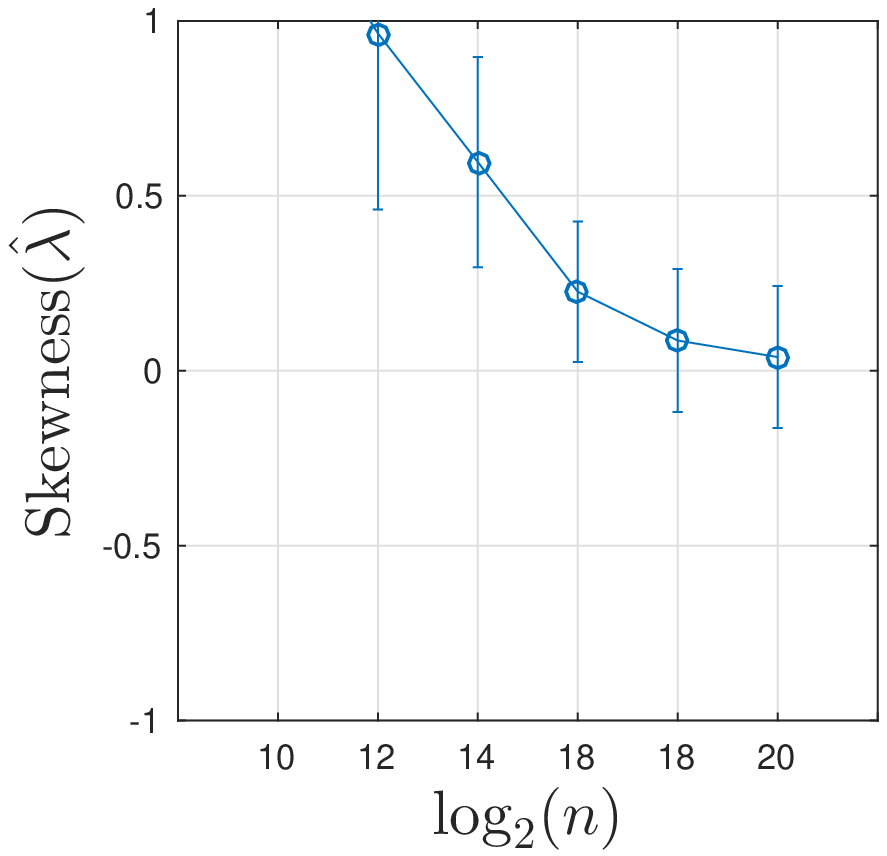}
    \end{minipage}
    \begin{minipage}{.24\linewidth}
        \centering
        \includegraphics[width=\linewidth]{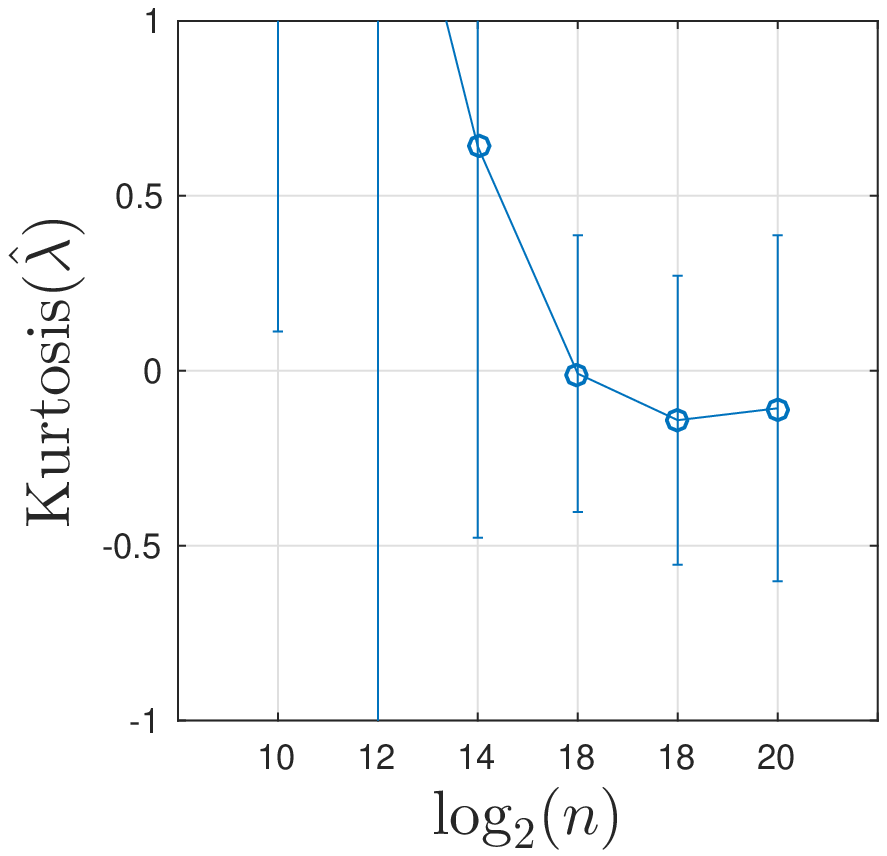}
    \end{minipage}

       \begin{minipage}{.24\linewidth}
        \centering
        \includegraphics[width=\linewidth]{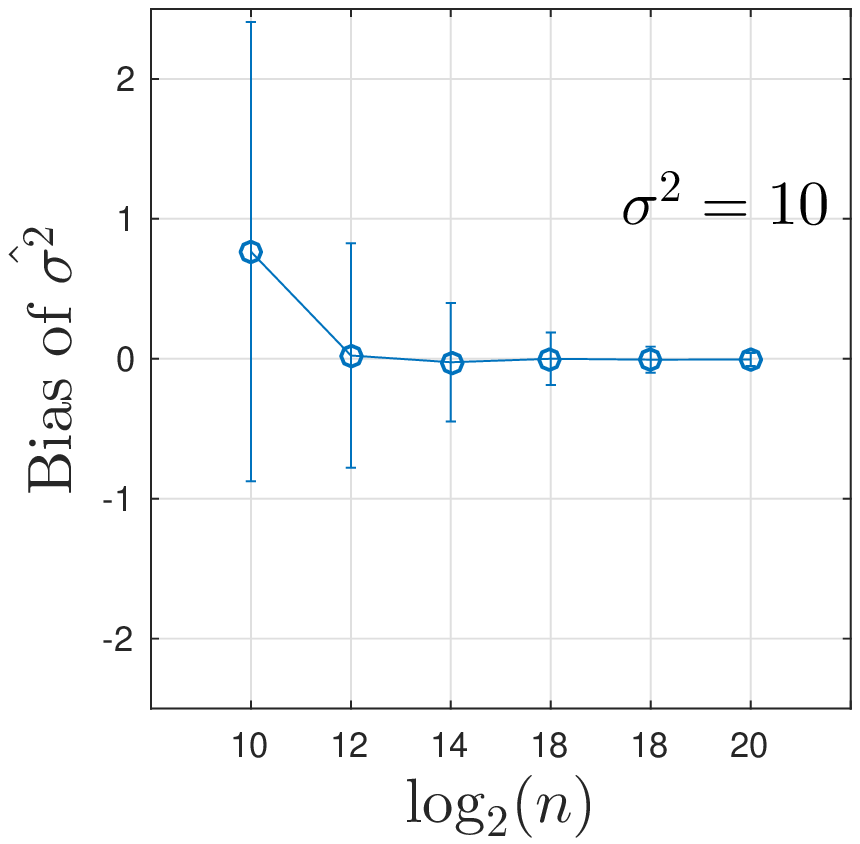}
    \end{minipage}
    \begin{minipage}{.24\linewidth}
        \centering
        \includegraphics[width=\linewidth]{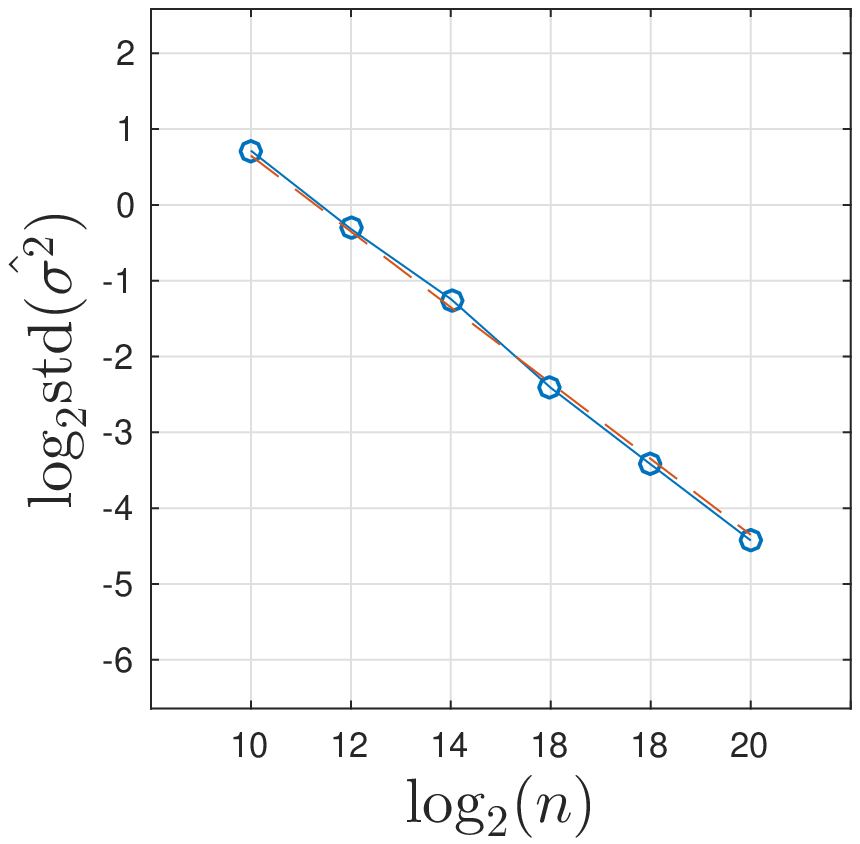}
    \end{minipage}
    \begin{minipage}{.24\linewidth}
        \centering
        \includegraphics[width=\linewidth]{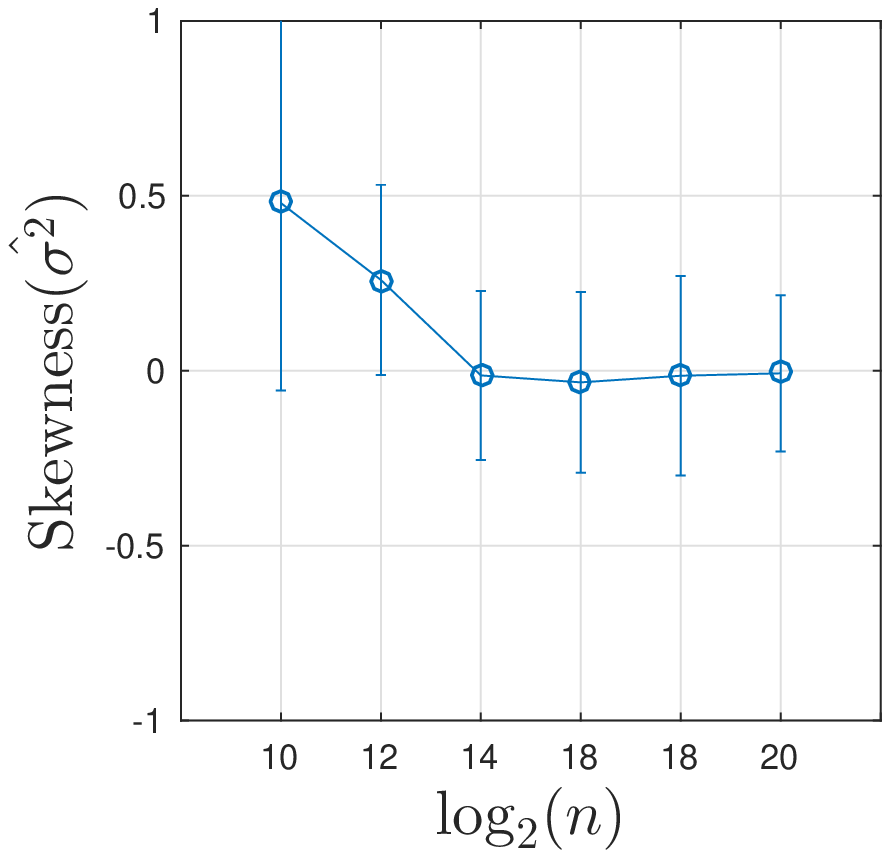}
    \end{minipage}
    \begin{minipage}{.24\linewidth}
        \centering
        \includegraphics[width=\linewidth]{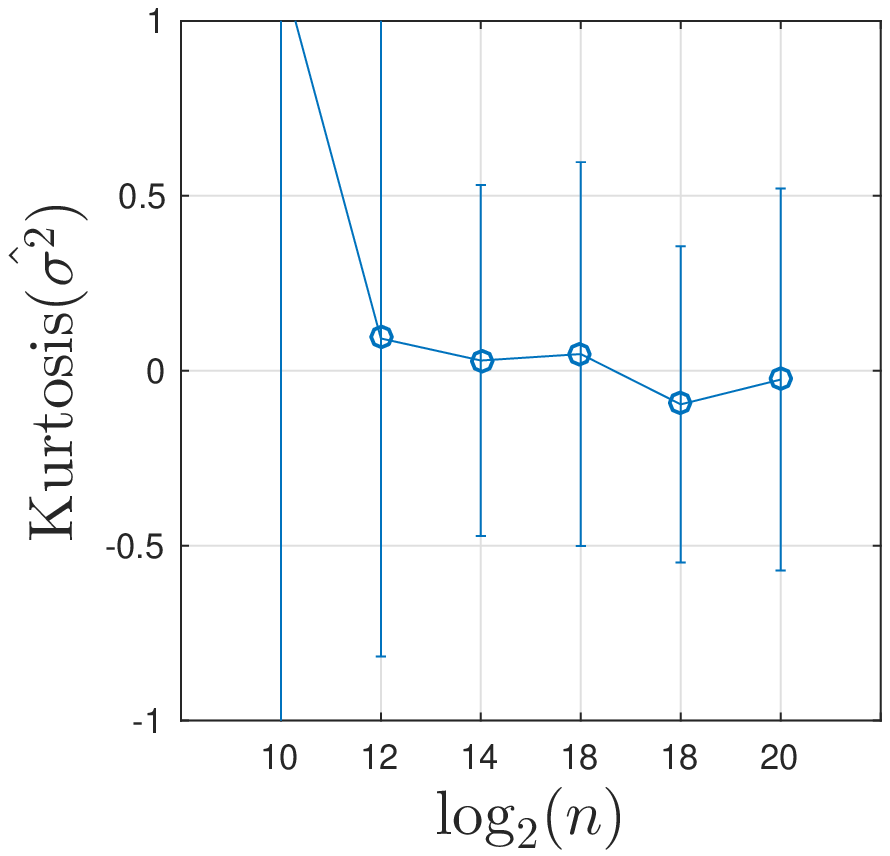}
    \end{minipage}
    \caption{\label{fig:skew_kurt_bias_plots} \textbf{\boldmath  Estimation performance and asymptotic normality of $\widehat{H},\widehat{\lambda},\widehat{\sigma^2}$.} From left to right, each column represents the Monte Carlo bias, standard deviation, skewness, and kurtosis, of the estimators $\widehat{H}$, $\widehat{\lambda}$ and $\widehat{\sigma^2}$, all with $N_\psi=2$ based on 5000 realizations at the sample sizes $n=2^{10},2^{12},\ldots,2^{20}$. The top two rows demonstrate the performance of $\widehat{H}$ with $(H,\lambda,\sigma^2)=(.85,.01,10)$ and $(H,\lambda,\sigma^2)=(.15,.01,10)$, respectively.  The bottom rows correspond to the performance of $\widehat{\lambda}$ and $\widehat{\sigma^2}$ when $(H,\lambda,\sigma^2)=(.15,.01,10)$.  Each of the red dashed lines in the second column correspond to a slope of $-1/2$, indicating the $\sqrt{n}$ convergence of the estimators.}
    \end{figure}

\begin{table}[h]
\centering
\begin{minipage}{.49\linewidth}
\includegraphics[width=\linewidth]{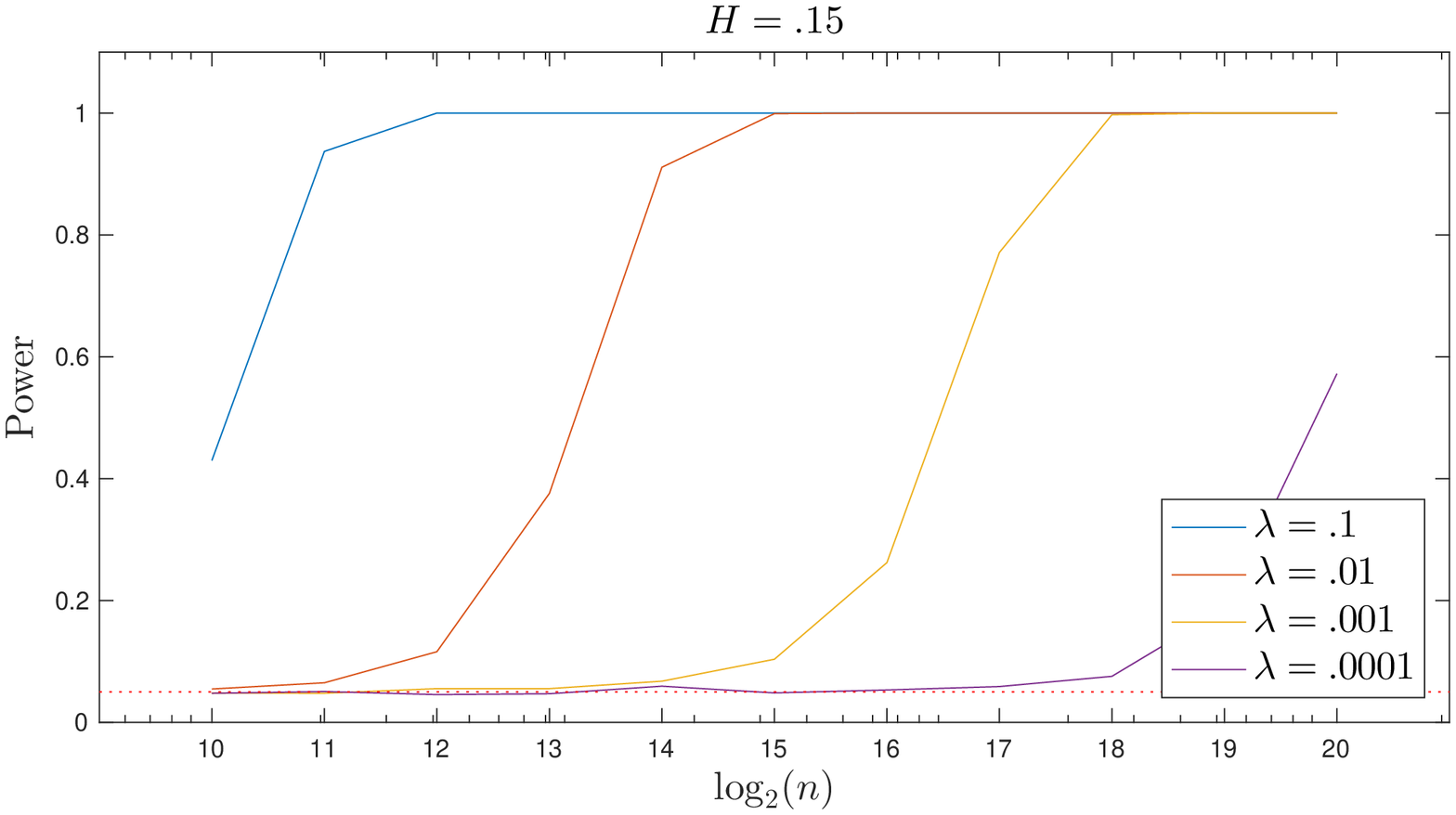}
\resizebox{\linewidth}{!}{
\begin{tabular}{|c||c|c|c|c|c|}
\hline
\multirow{2}{*}{}&\multicolumn{5}{|c|}{under $H_1$, $\epsilon=.05$, $H=.15$} \\\cline{2-6}
&$\lambda=.0001$&$\lambda=.001$&$\lambda=.01$&$\lambda=.1$&$\lambda=1$\\\hhline{|=||=|=|=|=|=|}
\textbf{$n=2^{10}$}&.0478&.0480&.0548&.4298&1.0000\\\hline
\textbf{$n=2^{11}$}&.0506&.0478&.0650&.9370&1.0000\\\hline
\textbf{$n=2^{12}$}&.0456&.0554&.1160&1.0000&1.0000\\\hline
\textbf{$n=2^{13}$}&.0470&.0554&.3758&1.0000&1.0000\\\hline
\textbf{$n=2^{14}$}&.0594&.0676&.9110&1.0000&1.0000\\\hline
\textbf{$n=2^{15}$}&.0484&.1036&.9990&1.0000&1.0000\\\hline
\textbf{$n=2^{16}$}&.0532&.2622&1.0000&1.0000&1.0000\\\hline
\textbf{$n=2^{17}$}&.0588&.7712&1.0000&1.0000&1.0000\\\hline
\textbf{$n=2^{18}$}&.0756&.9972&1.0000&1.0000&1.0000\\\hline
\textbf{$n=2^{19}$}&.1862&1.0000&1.0000&1.0000&1.0000\\\hline
\textbf{$n=2^{20}$}&.5722&1.0000&1.0000&1.0000&1.0000\\\hline
\end{tabular}}\end{minipage}
\begin{minipage}{.49\linewidth}
\includegraphics[width=\linewidth]{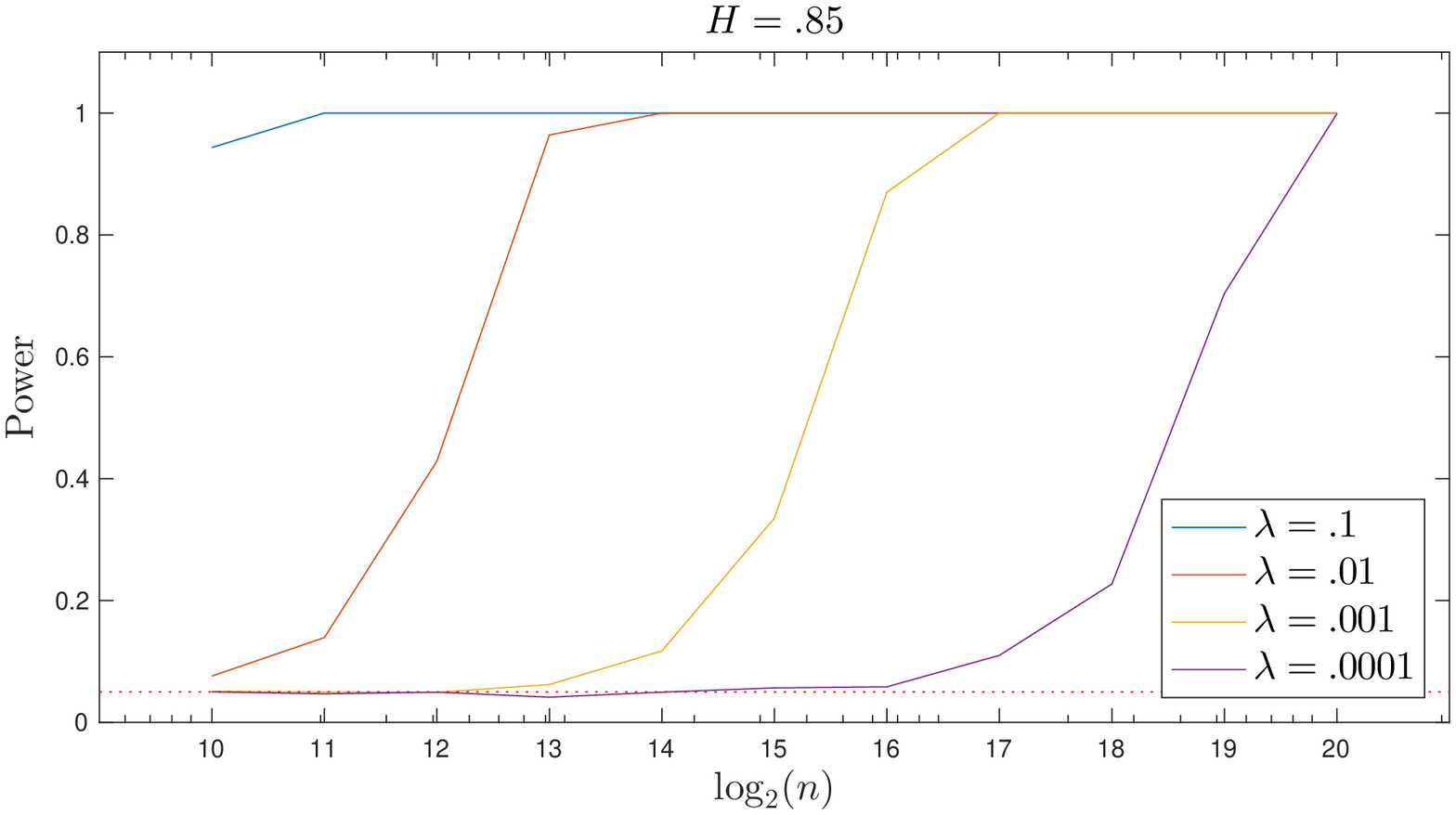}
\resizebox{\linewidth}{!}{
\begin{tabular}{|c||c|c|c|c|c|}
\hline
\multirow{2}{*}{}&\multicolumn{5}{|c|}{under $H_1$, $\epsilon=.05$, $H=.85$} \\\cline{2-6}
&$\lambda=.0001$&$\lambda=.001$&$\lambda=.01$&$\lambda=.1$&$\lambda=1$\\\hhline{|=||=|=|=|=|=|}
\textbf{$n=2^{10}$}&.0502&.0508&.0760&.9434&1.0000\\\hline
\textbf{$n=2^{11}$}&.0468&.0498&.1392&1.0000&1.0000\\\hline
\textbf{$n=2^{12}$}&.0496&.0492&.4288&1.0000&1.0000\\\hline
\textbf{$n=2^{13}$}&.0414&.0622&.9640&1.0000&1.0000\\\hline
\textbf{$n=2^{14}$}&.0496&.1172&1.0000&1.0000&1.0000\\\hline
\textbf{$n=2^{15}$}&.0566&.3352&1.0000&1.0000&1.0000\\\hline
\textbf{$n=2^{16}$}&.0584&.8702&1.0000&1.0000&1.0000\\\hline
\textbf{$n=2^{17}$}&.1100&1.0000&1.0000&1.0000&1.0000\\\hline
\textbf{$n=2^{18}$}&.2272&1.0000&1.0000&1.0000&1.0000\\\hline
\textbf{$n=2^{19}$}&.7040&1.0000&1.0000&1.0000&1.0000\\\hline
\textbf{$n=2^{20}$}&.9990&1.0000&1.0000&1.0000&1.0000\\\hline
\end{tabular}}
\end{minipage}

\caption{\textbf{\boldmath  Test power.} Each entry in the above tables corresponds to the Monte Carlo test power at the given values of $H,\lambda$ based on 5000 runs at the $\epsilon=.05$ level of significance. The red dashed line corresponds to the $\epsilon=.05$ significance level. }
\label{table:pwr}
\end{table}

\noindent \textbf{Hypothesis testing.} To study the finite-sample performance of the test in Theorem \ref{t:test}, 5000 independent realization of tfBm paths were generated over the parameters $(H,\lambda) \in \{0.15,0.85\} \times \{0.0001,0.001,0.01,0.1,1\}$ at the sample sizes $n=2^{10},2^{11},\ldots,2^{20}$. The chosen octaves $j_1,j_2,j_3,j_4$ for the test statistic \eqref{e:test_stat} were $j_1=1$, $j_2=2$, $j_3=\log_2(n)-7$, $j_4=\log_2(n)-3$, and in all cases we set $N_\psi=2$. Simulation studies were conducted to estimate the quantiles used for the test in Table \ref{table:pwr}.  Numerical results indicate that the asymptotic standard deviation \eqref{e:Ralpha} is not very sensitive to $H$ and lies between 0.080 and 0.101 for multiple choices of $H$ (a similar phenomenon is commonly found in the wavelet analysis of fractional stochastic processes; see, for instance, Wendt et al.\ \cite{wendt:didier:combrexelle:abry:2017}, Section 4). So, in the simulation studies of the test we picked values for $\widehat{\tau}_0(H)$ in the indicated range.

As can be seen in Table \ref{table:pwr}, for any fixed value of $\lambda$, reliable testing of fBm vs tfBm alternatives requires a certain sample size. For the moderate values $\lambda = 0.1, 1$, the test is very powerful at the moderate sample size of $2^{11}$ regardless of $H$. However, for the small values $\lambda = 0.0001,0.001, 0.01$, significantly larger sample sizes are required for attaining high power, especially for the small value $H=.15$. This confirms and quantifies what Figure \ref{fig:semilrd}, left plot, visually suggests.

\subsection{River flow data modeling: fBm or tfBm?}\label{s:application}

As an application, we revisit velocity data associated with turbulent supercritical flow in the Red Cedar river, a fourth-order stream in Michigan, USA (Coordinates: 42.72908N, 84.48228W). The data was kindly provided by Prof.\ Mantha S.\ Phanikumar, from Michigan State University, and is also modeled in Meerschaert et al.\ \cite{meerschaert:sabzikar:phanikumar:zeleke:2014} in the Fourier domain. The data set contains flow features over a range of spatial and temporal scales associated with turbulent flows in the natural environment and is believed to be appropriate for the analysis of energy spectra. The measurements ($n=46080$ points) were made at a sampling rate of 50 Hz using a 16 MHz Sontek Micro-ADV (Acoustic Doppler Velocimeter) on May 26, 2014.

Inspection of the sample wavelet spectrum (Figure \ref{fig:prob1_6_1}, right plot) shows that tfBm provides a close fit for turbulent flow data, and that a conspicuous deviation from fBm appears over large octaves. An application of \eqref{e:def_estimator} yields the wavelet-based estimates $\widehat{H}=0.329$, $\widehat{\lambda}=0.121$, $\widehat{\sigma^2}=24.825$. In particular, $\widehat{H}$ is strikingly close to the value $1/3$ predicted by the Kolmogorov scaling model for the inertial range. Figure \ref{fig:ACFs} further compares sample autocorrelation function plots of the flow data and of simulated tfBm based on the fitted values. As expected from a tfBm model, both look compatible with stationarity and the discrepancy between them is tiny. The visual impression that tfBm is a better model for the flow data than fBm is confirmed by the test \eqref{e:Ralpha}. To conduct the test, we made the natural choice of low octaves $j_1=1$ and $j_2=2$, as well as $j_3=8$ and $j_4=10$ for the large octaves. This yielded the value $T_n = 0.8824$ for the test statistic \eqref{e:test_stat}, with an associated $p$-value of the order $10^{-16}$. Hence, and unsurprisingly, there is strong evidence against the null hypothesis of fBm. Other reasonable choices of octaves $j_3$ and $j_4$ lead to the same conclusion.

\begin{figure}[!htb]
    \centering
    \begin{minipage}{\linewidth}
        \centering
        \includegraphics[width=.8\linewidth]{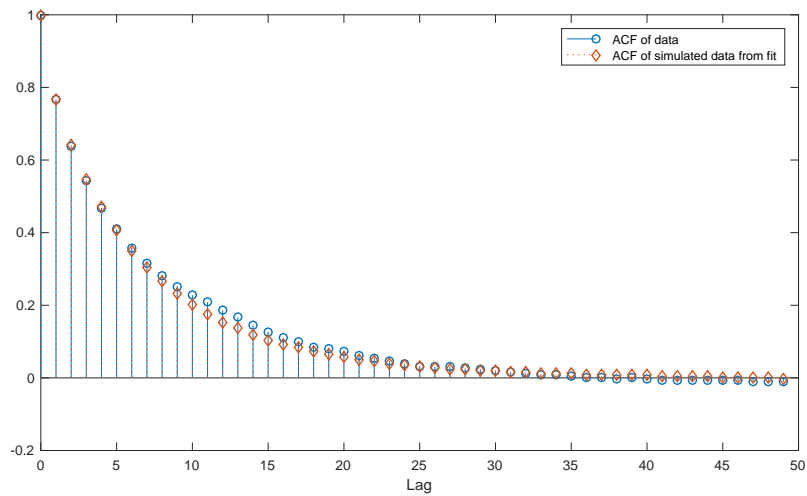}
    \end{minipage}%
    \caption{\label{fig:ACFs} \textbf{Sample autocorrelation function plots} for geophysical flow data and for one simulated tfBm path (of length $2^{20}$) based on the vector parameter value given by the fitted vector of parameters $\widehat{{\boldsymbol \theta}}_n = (\widehat{H}, \widehat{\lambda}, \widehat{\sigma^2})$.}
    \end{figure}

\section{Conclusion}

Tempered fractional Brownian motion (tfBm) is a canonical model that displays the so-named Davenport spectrum. The latter is a modification of the classical Kolmogorov spectrum for the inertial range of turbulence that includes non-scaling behavior at low frequencies. The autocorrelation of the increments of tfBm displays semi-long range dependence, a phenomenon that has been observed in a wide range of applications. In this paper, we use wavelets to construct the first estimator for tfBm and a simple and computationally efficient test for fBm vs tfBm alternatives. We make the realistic assumption that only discrete time measurements are available. The properties of the wavelet estimator and test are mathematically and computationally established. An application of the methodology to geophysical flow data from the Red Cedar river in Michigan, USA, showed that tfBm provides a much closer fit than fBm.

This work also points to a number of open research problems, such as the testing of tfBm vs stationary alternatives, physical modeling for $H > 1$, and the construction of statistical methodology for several other Gaussian or non-Gaussian tempered fractional models.

\appendix

\section{Proofs and auxiliary results}
Throughout the paper, we use the convention
$$
\widehat{f}(x) = \int_\R e^{-ixt}f(t)dt
$$
to denote the Fourier transform for any $f\in L^2(\R)$. In all mathematical statements, we assume the following conditions hold on the underlying wavelet MRA.

\medskip

\noindent {\sc Assumption $(W1)$}: $\psi \in L^2(\bbR)$ is a wavelet function, namely, it satisfies the relations
\begin{equation}\label{e:N_psi}
\int_{\bbR} \psi^2(t)dt = 1 , \quad \int_{\bbR} t^{p}\psi(t)dt = 0, \quad p = 0,1,\hdots, N_{\psi}-1, \quad \int_{\bbR} t^{N_{\psi}}\psi(t)dt \neq 0,
\end{equation}
for some integer (number of vanishing moments) $N_{\psi} \geq 1$.

\medskip

\noindent {\sc Assumption ($W2$)}: the scaling and wavelet functions
\begin{equation}\label{e:supp_psi=compact}
\textnormal{$\phi\in L^1(\bbR)$ and $\psi\in L^1(\bbR)$ are compactly supported }
\end{equation}
and $\widehat{\phi}(0)=1$.

\medskip

\noindent {\sc Assumption $(W3)$}: there is $\alpha > 1$ such that
\begin{equation}\label{e:psihat_is_slower_than_a_power_function}
\sup_{x \in \bbR} |\widehat{\psi}(x)| (1 + |x|)^{\alpha} < \infty.
\end{equation}

\medskip

\noindent {\sc Assumption ($W4$)}: the function
\begin{equation}\label{e:sum_k^m_phi(.-k)}
\sum_{k\in \mathbb{Z}}k^m\phi(\cdot-k)
\end{equation}
is a polynomial of degree $m$ for all $m=0,\hdots,N_{\psi}-1$.

\medskip



Conditions \eqref{e:N_psi} and \eqref{e:supp_psi=compact} imply that $\widehat{\psi}(x)$ exists, is everywhere differentiable and its first $N_{\psi}-1$ derivatives are zero at $x = 0$. Condition \eqref{e:psihat_is_slower_than_a_power_function}, in turn, implies that $\psi$ is continuous (see Mallat \cite{mallat:1999}, Theorem 6.1) and, hence, bounded.

The Daubechies scaling and wavelet functions generally satisfy ($W1-W4$) (see Moulines et al.\ \cite{moulines:roueff:taqqu:2008}, p.\ 1927, or Mallat \cite{mallat:1999}, p.\ 253). Usually, the parameter $\alpha$ increases to infinity as $N_{\psi}$ goes to infinity (see Moulines et al.\ \cite{moulines:roueff:taqqu:2008}, p.\ 1927, or Cohen \cite{cohen:2003}, Theorem 2.10.1). In some results, we will make use of the additional condition
\begin{equation}\label{e:Npsi>alpha}
{N_{\psi} \geq 2},
\end{equation}
which is assumed in Theorem \ref{t:test}.

\subsection{Estimation}

\begin{proposition}\label{p:wavelet_spectrum}
Expressions \eqref{e:discspec} and \eqref{e:disc_slrd2} hold.
\end{proposition}
\begin{proof}
To show \eqref{e:discspec} and \eqref{e:disc_slrd2}, we first establish the general expression
\begin{equation}\label{e:discspec_j,j'}
\bbE d(2^j,k)d(2^{j'},k') = \frac{\sigma^2\Gamma^2\left( H+\frac{1}{2} \right)}{2\pi} \int^{\pi}_{-\pi} e^{ix(2^jk-2^{j'}k')} \sum_{\ell \in \bbZ} \frac{ {\mathcal H}_j(x)\overline{{\mathcal H}_{j'}(x)}}{|\lambda^2 + (x + 2 \pi \ell)^2|^{H+\frac{1}{2}}}dx,
\end{equation}
where $j,j' \in \bbN$ and $k,k' \in \bbZ$. The latter can be shown as Proposition C.1 in Abry et al.\ \cite{abry:didier:li:2018}. For the reader's convenience, we provide the details. Given the initial approximation sequence \eqref{e:a(0,k)}, we can express the wavelet coefficients at octave $j$ as $d(2^j,k)=2^{-j/2}\int_\mathbb{R} \widetilde B_{H,\lambda}(t) \psi(2^{-j}t-k)dt$ (cf.\ \eqref{e:disc2}). Moreover, by making a change of variable $t=u + 2^jk$, $t'=u' + 2^{j'}k'$, followed by $\ell=2^jk-m$, $\ell'=2^{j'}k'-m'$,
$$
\textnormal{Cov}(d(2^j,k),d(2^{j'},k'))
$$
$$
=2^{-\frac{j+j'}{2}}\mathbb{E}\int_\mathbb{R}\int_\mathbb{R} \sum_{\ell\in\Z}\sum_{\ell'\in\Z}\psi(2^{-j}t-k)\psi(2^{-j'}t'-k')\phi(t-\ell)\phi(t'-\ell')B_{H,\lambda}(\ell)B_{H,\lambda}(\ell')dtdt'
$$
\begin{equation}\label{e:Cov(Dtilde,Dtilde)}
=2^{-\frac{j+j'}{2}}\mathbb{E}\int_\mathbb{R}\int_\mathbb{R}\sum_{m\in\Z}\sum_{m'\in\Z} \psi(2^{-j}u)\psi(2^{-j'}u')\phi(u+\ell)\phi(u'+\ell')B_{H,\lambda}(2^j k-m)B_{H,\lambda}(2^{j'} k'-m')dudu'.
\end{equation}
Note that, by conditions \eqref{e:N_psi} and \eqref{e:sum_k^m_phi(.-k)}, 
\begin{equation}\label{eq:vanish}
\int_\mathbb{R} \psi(2^{-j}t)\sum_{\ell\in \bbZ}\phi(t+\ell)\ell^mdt=0, \quad j \geq 0,\quad m=0,\hdots,N_\psi-1.
\end{equation}
Let
\begin{equation}\label{e:g(x)}
g(x) = \frac{\Gamma^2\left( H+\frac{1}{2} \right)}{2\pi} \frac{\sigma^2}{|\lambda^2 + x^2|^{H+\frac{1}{2}}}.
\end{equation}
Recall that ${\mathcal H}_j(x)$ is defined by \eqref{e:Hj}, where the filter sequence is given by \eqref{e:hj,l}. By Proposition 3 in Moulines et al.\ \cite{moulines:roueff:taqqu:2007:JTSA}, $|{\mathcal H}_j(x)|\leq C$, $x\in \bbR$. Thus, by Fubini's theorem, \eqref{e:tfbm_harmonizable} and \eqref{eq:vanish}, we can reexpress \eqref{e:Cov(Dtilde,Dtilde)} as
$$
2^{-\frac{j+j'}{2}}\int_\mathbb{R}\int_\mathbb{R}\int_\mathbb{R} \sum_{\ell\in\Z}\sum_{\ell'\in\Z}\psi(2^{-j}t)\psi(2^{-j'}t')\phi(t+\ell)\phi(t'+\ell')
e^{i(2^jk-\ell)x}e^{-i(2^{j'}k'-\ell')x}g(x)dtdt'dx
$$
\begin{equation}\label{eq:covdis}
= \int_\mathbb{R} {\mathcal H}_j(x)\overline{{\mathcal H}_{j'}(x)}e^{ix(2^jk-2^{j'}k')}g(x)dx,
\end{equation}
where the absolute value of the integral on the right-hand side of \eqref{eq:covdis} is finite. This shows \eqref{e:discspec_j,j'}.

Expression \eqref{e:discspec} is a consequence of \eqref{e:discspec_j,j'} by setting $j = j'$ and $k = k'$.
To show \eqref{e:disc_slrd2}, consider the sequence $\{h_{j,l}\}_{l \in \bbZ}$ as in \eqref{e:hj,l}, and write
\begin{equation}\label{e:Psi-jk}
\psi_{j,k}(t) = 2^{-j/2}\psi(2^{-j}t-k), \quad  \phi_\ell(t)=\phi(t+\ell).
\end{equation}
By \eqref{e:discspec_j,j'} for $j = j'$ and $k = k' = 0$, we can write
\begin{equation}\label{e:Ed^2(2^j,0)_in_terms_of_Hj}
\bbE d^2(2^j,0) = \frac{\sigma^2\Gamma^2\left( H+\frac{1}{2} \right)}{2\pi} \int^{\pi}_{-\pi} \sum_{\ell \in \bbZ} \frac{ |{\mathcal H}_j(x)|^2}{|\lambda^2 + (x + 2 \pi \ell)^2|^{H+\frac{1}{2}}}dx.
\end{equation}
Note that, by Parseval's theorem,
\begin{equation}\label{e:H_Parseval}
\frac{1}{2\pi}\int_{-\pi}^\pi |\mathcal{H}_j(x)|^2dx = \sum_{\ell\in \bbZ}|h_{j,\ell}|^2.
\end{equation}
Since $\psi_{j,k}\in\text{span}\{\phi_\ell\}_{\ell\in \bbZ}$, for all $j\geq 1$, the Pythagorean theorem and expression \eqref{e:Hj} imply that
\begin{equation}\label{e:sum_h^2=1}
\sum_{\ell\in \Z}h_{j,\ell}^2 = \sum_{\ell\in \Z}  |\langle \psi_{j,0},\phi_{\ell}\rangle|^2  = \norm{\psi}^2_{L^2} =1, \quad j \in \bbN.
\end{equation}
Reexpress ${\mathcal H}_j $ as
\begin{equation}\label{e:Hj_expanded_pois_formula}
\mathcal{H}_j(x)  = 2^{j/2}\sum_{\ell \in \Z}  \widehat{\phi}(x+2\pi \ell) \overline{\widehat{\psi}(2^{j}(x+2\pi \ell))}
\end{equation}
(see Moulines et al.\ \cite{moulines:roueff:taqqu:2007:JTSA}, p.\ 180).
For any $0<a<\pi$,
$$
 \int_{[-\pi,\pi]\setminus(-a,a)} |\mathcal{H}_j(x)|^2dx $$
\begin{equation}\label{e:Hj_to_dirac_1}
 \leq \int_{[-\pi,\pi]\setminus(-a,a)}\left||\mathcal{H}_j(x)|^2 - 2^{j}{|\widehat\phi(x)\widehat\psi(2^jx)|^2}\right|dx + \int_{[-\pi,\pi]\setminus(-a,a)}2^{j}{|\widehat{\phi}(x)\widehat\psi(2^jx)|^2}dx.
\end{equation}
 From relation (79) in Moulines et al.\ \cite{moulines:roueff:taqqu:2007:JTSA}, for all $x\in(-\pi,\pi)$,
$$
\left||\mathcal{H}_j(x)|^2 - 2^{j}|\widehat\phi(x)\widehat\psi(2^jx)|^2\right|\leq \frac{2^{j(1+N_\psi-\alpha)}|x|^{2N_\psi}}{(1+2^j|x|)^{N_\psi+\alpha}},
$$
where $\alpha$ is the constant in condition \eqref{e:psihat_is_slower_than_a_power_function}.  Thus, by a change of variable $y = 2^{j}x$, the first summand in \eqref{e:Hj_to_dirac_1} is bounded by
$$
 \int_{[-\pi,\pi]\setminus(-a,a)} \frac{2^{j(1+N_\psi-\alpha)}|x|^{2N_\psi}}{(1+2^j|x|)^{N_\psi+\alpha}}dx = 2\int_{2^{j}a}^{2^j\pi} \frac{2^{j(1+N_\psi-\alpha)}|y/2^j|^{2N_\psi}}{(1+|y|)^{N_\psi+\alpha}}\frac{dy}{2^j}
 $$
 $$=2\int_{2^ja}^{2^j\pi}\frac{2^{j(-N_\psi-\alpha)}|y|^{2N_\psi}}{(1+|y|)^{N_\psi+\alpha}}dy
\leq C \frac{2^{j(N_\psi+1-\alpha)}}{2^{j(N_\psi+\alpha)}} \to 0,
$$
as $j\to\infty$, since $\alpha>1$. Likewise, for the second summand in \eqref{e:Hj_to_dirac_1}, we have
$$
 \int_{[-\pi,\pi]\setminus(-a,a)} 2^{j}{|\widehat{\phi}(x)\widehat\psi(2^jx)|^2}dx =2 \int_{2^ja}^{2^j\pi}{|\widehat{\phi}(2^{-j}y)\widehat\psi(y)|^2}dy \to 0, \quad j \rightarrow \infty,
$$
by the dominated convergence theorem.
Hence, for every $a>0$,
\begin{equation}\label{e:int_restricted_Hj^2(x)dx->0}
\int_{[-\pi,\pi]\setminus(-a,a)} |\mathcal{H}_j(x)|^2dx\to 0, \quad j\to\infty.
\end{equation}
Therefore, by \eqref{e:H_Parseval}, \eqref{e:sum_h^2=1}, and \eqref{e:int_restricted_Hj^2(x)dx->0}, we have
\begin{equation}\label{e:Hlim=1}
\lim_{j\to \infty}\frac{1}{2\pi}\int_{-a}^a |\mathcal{H}_j(x)|^2dx = 1.
\end{equation}
So, define $f(x) = \sum_{\ell\in\bbZ}\frac{\sigma^2 \Gamma^2(H+\frac{1}{2})}{|\lambda^2+(2\pi \ell +x)^2|^{H+\frac{1}{2}}}$, which is continuous on $[-\pi,\pi]$ by the dominated convergence theorem. For a fixed $\varepsilon>0$, choose $0 < a < \pi$ small so that, for $|x|<a$, $|f(x)-f(0)|<\varepsilon$.
Then, by \eqref{e:H_Parseval}, \eqref{e:sum_h^2=1} and \eqref{e:int_restricted_Hj^2(x)dx->0},
$$
\left|\frac{1}{2\pi}\int_{-\pi}^\pi |\mathcal{H}_j(x)|^2f(x)dx-f(0)\right| \leq \frac{1}{2\pi}\int_{-\pi}^\pi |\mathcal{H}_j(x)|^2|f(x)-f(0)| dx
$$
$$
= \frac{1}{2\pi}\int_{-a}^a |\mathcal{H}_j(x)|^2|f(x)-f(0)| dx + \frac{1}{2\pi}\int_{[-\pi,\pi]\setminus[-a,a]} |\mathcal{H}_j(x)|^2|f(x)-f(0)| dx
$$
\begin{equation}\label{e:|int_Hj-f(0)|_bound}
< \frac{\varepsilon}{2\pi} \int_{-a}^a |\mathcal{H}_j(x)|^2 dx  +\frac{1}{\pi}\norm{f}_\infty \int_{[-\pi,\pi]\setminus[-a,a]} |\mathcal{H}_j(x)|^2 dx, 
\end{equation}
where $\|\cdot\|_{\infty}$ is the usual $L^\infty$ norm. By taking $j$ sufficiently large, combined with \eqref{e:Ed^2(2^j,0)_in_terms_of_Hj}, \eqref{e:int_restricted_Hj^2(x)dx->0}, \eqref{e:Hlim=1} and \eqref{e:|int_Hj-f(0)|_bound}, we arrive at \eqref{e:disc_slrd2}.  $\Box$\\
\end{proof}

In the following proposition, we describe the decorrelation property of the wavelet transform for both tfBm \eqref{e:approx_decorrelation} and fBm in discrete time, which is used in Lemmas \ref{fBm_AN} and \ref{Hest_AN}. Hereinafter,
\begin{equation}\label{e:dtfBm_dfBm}
d_{\text{fBm}}(2^j,k), \quad d_{\text{tfBm}}(2^j,k),
\end{equation}
denote the sample wavelet coefficients of fBm and tfBm, respectively. Since the result for fBm is classical, we only provide the proof for tfBm.

\begin{proposition}\label{p:tfbm_decorr}
Let $T=\textnormal{length}(\textnormal{supp}(\psi))$ and choose $j,j' \in \bbN$, $k,k'\in \bbZ$ that satisfy
\begin{equation}\label{e:|2^jk-2^{j'}k'|>=2T*max(2^j,2^{j'})}
\frac{|2^jk-2^{j'}k'|}{\max\{2^j,2^{j'}\}}\geq 2 T.
\end{equation}
Then,
\begin{itemize}
\item[$(i)$] for some constant $C$ that is independent of $j,j',k,k'$,
\begin{equation}\label{e:tfbm_decorr}
|\textnormal{Cov}(d_{\textnormal{tfBm}}(2^j,k),d_{\textnormal{tfBm}}(2^{j'},k'))| \leq C2^{\frac{j+j'}{2}}|2^jk-2^{j'}k'|^{H-1/2}e^{-\lambda|2^jk-2^{j'}k'|};
\end{equation}
\item[$(ii)$] for some constant $C$ that is independent of $j,j',k,k'$,
\begin{equation}\label{e:fbm_decorr}
|\textnormal{Cov}(d_{\textnormal{fBm}}(2^j,k),d_{\textnormal{fBm}}(2^{j'},k'))| \leq C\frac{2^{(j+j')(N_{\psi}+\frac{1}{2})}}{|2^jk-2^{j'}k'|^{2N_{\psi}-2H}}.
\end{equation}
\end{itemize}
\end{proposition}

\begin{proof} We only show $(i)$. Without loss of generality we can assume that $\psi$ and $\phi$ are both supported in $[0,T]$. By \eqref{e:Cov(Dtilde,Dtilde)} and Fubini's theorem, we can write
$$
\textnormal{Cov}(d_{\textnormal{tfBm}}(2^j,k),d_{\textnormal{tfBm}}(2^{j'},k'))
$$
\begin{equation}\label{e:Cov(dtfbm,dtfbm)}
=2^{-\frac{j+j'}{2}}\int_\mathbb{R}\int_\mathbb{R}\sum_{\ell\in\Z}\sum_{\ell'\in\Z} \psi(2^{-j}t)\psi(2^{-j'}t')\phi(t+\ell)\phi(t'+\ell')\mathbb{E}[B_{H,\lambda}(2^j k-\ell)B_{H,\lambda}(2^{j'} k'-\ell')]dtdt'.
\end{equation}
Recall that the covariance function of tfBm is given by \eqref{e:cov}. Then, in expression \eqref{e:Cov(dtfbm,dtfbm)}, by \eqref{eq:vanish},
$$
\int_\mathbb{R}\psi(2^{-j}t)\sum_{\ell\in\Z}\phi(t+\ell)dt \left[C^2_{2^{j'} k'-\ell'} |2^{j'} k'-\ell'|^{2H}\right] = 0,
$$
 and similarly for the term involving $t',\ell',j,k$.  Thus, we can rewrite \eqref{e:Cov(dtfbm,dtfbm)} as
$$
-2^{-\frac{j+j'}{2}}\frac{\sigma^2}{2}\int_\mathbb{R}\int_\mathbb{R}\sum_{\ell\in\Z}\sum_{\ell'\in\Z} \psi(2^{-j}t)\psi(2^{-j'}t')\phi(t+\ell)\phi(t'+\ell')C^2_{2^j k-2^{j'}k'+\ell'-\ell} |2^j k-2^{j'}k'+\ell'-\ell|^{2H}dtdt'
$$
$$
= M2^{-\frac{j+j'}{2}}\int_\mathbb{R}\int_\mathbb{R}\sum_{\ell\in\Z}\sum_{\ell'\in\Z} \psi(2^{-j}t)\psi(2^{-j'}t')\phi(t+\ell)\phi(t'+\ell') |2^j k-2^{j'} k' +\ell' - \ell|^{H}K_{H}(\lambda|2^j k-2^{j'} k' +\ell' - \ell|)dtdt'
$$
$$
=M|r|^{H}2^{-\frac{j+j'}{2}}\int_\mathbb{R}\int_\mathbb{R}\sum_{\ell\in\Z}\sum_{\ell'\in\Z} \psi(2^{-j}t)\psi(2^{-j'}t')\phi(t+\ell)\phi(t'+\ell')
$$
\begin{equation}\label{e:Cov(dtfbm,dtfbm)_after_applying_Npsi}
\times \left| 1 +  \frac{\ell' - \ell}{{r}}\right|^{H}K_{H}\left(\lambda|r| \left| 1 +  \frac{\ell' - \ell}{{r}}\right|\right)dtdt',
\end{equation}
where
\begin{equation}\label{e:r=2^jk-2^j'k'}
r=2^jk-2^{j'}k'
\end{equation}
and $M=\frac{\sigma^2\Gamma(H+\frac{1}{2})}{\sqrt{\pi}(2 \lambda)^{H}} $. Note that $\textnormal{supp}(\phi(\cdot+\ell)) =[-\ell,-\ell + T]$, and $\text{supp}(\psi(2^{-j}\cdot)) = [0,2^j T]$.  Therefore, $\phi(t+l)\psi(2^{-j}t)=0$ if $(-\ell,-\ell + T)\cap (0,2^j T)=\varnothing$, i.e., if $\ell\leq -2^{j}T$ or if $\ell\geq T$, and so each sum ranges over $\ell\in(-2^{j}T,T) =:S$ and $\ell'\in(-2^{j'}T,T)=:S'$. Thus, by assumption \eqref{e:|2^jk-2^{j'}k'|>=2T*max(2^j,2^{j'})},
\begin{equation}\label{e:(l-l')/|r|<.75}
\left| \frac{\ell' - \ell}{{r}}\right|\leq \left| \frac{\ell' - \ell}{2T\max\{2^j,2^{j'}\}}\right|\leq\frac{T(1+\max\{2^j,2^{j'})\}}{2T\max\{2^j,2^{j'}\}}\leq \frac{3}{4}.
\end{equation}
By making use of asymptotic expression
$$
K_H(t)\sim \sqrt{\frac{\pi}{2t}}e^{-t}, \quad t\to \infty
$$
(e.g. Olver et al \cite{olver2010}, p.\ 249, expression 10.25.3), for large $|r|$ we have
$$
MK_{H}\left(\lambda|r| \left| 1 +  \frac{\ell' - \ell}{2^j k-2^{j'} k'}\right|\right)\leq M'|r|^{-1/2}e^{-\lambda|r|},
$$
where the constant $M'$ is independent of $\ell,\ell'$.  Thus, expression \eqref{e:Cov(dtfbm,dtfbm)_after_applying_Npsi} is bounded by
\begin{equation}\label{e:Cov(dtfbm,dtfbm)_bounded}
M'|r|^{H-1/2}e^{-\lambda|r|} 2^{-\frac{j+j'}{2}}\int_\mathbb{R}\int_\mathbb{R}\sum_{\ell\in S}\sum_{\ell'\in S'} |\psi(2^{-j}t)\psi(2^{-j'}t')\phi(t+\ell)\phi(t'+\ell' )|dtdt'.
\end{equation}
Observe that each sum over $\ell$ and $\ell'$ in \eqref{e:Cov(dtfbm,dtfbm)_bounded} contains only $(2^{j}+1)T-1$ and $(2^{j'}+1)T-1$ terms, respectively, and recall that $\psi$ is bounded. Thus, after interchanging the summation and integration, each summand is bounded by the constant $\|\psi\|_{\infty}\|\phi\|_{L^1}$. Thus, \eqref{e:Cov(dtfbm,dtfbm)_bounded} is bounded by
$$
M''|r|^{H-1/2}e^{-\lambda|r|} 2^{-\frac{j+j'}{2}} 2^{j+j'}\|\psi\|^2_{\infty}\|\phi\|^2_{L^1} \leq C2^{\frac{j+j'}{2}}|r|^{H-1/2}e^{-\lambda|r|}.
$$
This establishes \eqref{e:tfbm_decorr}. $\Box$\\

\end{proof}
\begin{remark}
By a similar proof, expression \eqref{e:tfbm_decorr} can also be established starting from continuous time measurements.
\end{remark}

\begin{proposition}\label{p:cont_time_wavelet_spec}
Expression \eqref{e:cont_time_wavelet_spec} holds.
\end{proposition}
\begin{proof}
By the harmonizable representation of tfBm (see \eqref{e:tfbm_harmonizable}), the change of variable $s = 2^{-j}t - k$, condition \eqref{e:N_psi} and the change of variable $2^j x = y$, we can rewrite \eqref{e:d(2^j,k)_cont} as
$$
2^{-j/2}d(2^j,k) = \int_{{\Bbb R}} \psi(s) B_{H,\lambda}(2^j(s+k))ds
\stackrel{{\mathcal L}}=  \frac{\Gamma \left(H + \frac{1}{2}\right)}{\sqrt{2 \pi}}\int_{{\Bbb R}} \psi(s) \Big\{ \int_{{\Bbb R}} \frac{e^{-i2^j(s+k)x}-1}{(\lambda + ix)^{H + \frac{1}{2}}} \widetilde{B}(dx)\Big\}ds
$$
$$
=\frac{\Gamma(H + 1/2)}{\sqrt{2\pi}}\int_{{\Bbb R}} \frac{e^{-i 2^{j}kx}}{(\lambda + ix)^{H + \frac{1}{2}}} \Big\{\int_{{\Bbb R}} e^{-i(2^jx)s}\psi(s)ds\Big\}\widetilde{B}(dx)
$$
\begin{equation}\label{e:d(2j,k)_harmoniz}
= \frac{\Gamma(H + 1/2)}{\sqrt{2\pi}} \int_{{\Bbb R}} \frac{e^{-i 2^{j}kx}}{(\lambda + ix)^{H + \frac{1}{2}}} \widehat{\psi}(2^j x)\widetilde{B}(dx)
\stackrel{{\mathcal L}}= 2^{-j/2}\frac{\Gamma(H+1/2)}{\sqrt{2\pi}} \int_{{\Bbb R}} \frac{e^{-i ky}}{(\lambda + i\frac{y}{2^j})^{H + \frac{1}{2}}} \widehat{\psi}(y)\widetilde{B}(dy).
\end{equation}
The exchanging of integration order in the first equality $\stackrel{{\mathcal L}}=$ in \eqref{e:d(2j,k)_harmoniz} is justified by taking second moments and applying Fubini's theorem (cf.\ Abry et al.\ \cite{abry:didier:li:2018}, Proposition B.1, proof of property $(P2)$). By taking expectations, \eqref{e:cont_time_wavelet_spec} holds. $\Box$\\

\end{proof}

\newcommand{\diag}{\text{diag}}
\newcommand{\deq}{\stackrel{d}{=}}

In the following lemma, we establish the asymptotic normality of the (discrete time) wavelet variances \eqref{e:W(2j)_tfBm} of tfBm.
\begin{lemma}\label{l:ANwspec}
Fix $m \in \bbN$, and let $j_1,\hdots,j_m$ be a range of octaves. Then,
\begin{equation}\label{e:sqrt(nj)(W(2^j)-Ed^2(j,0))->N(0,sigma^2)}
\left(\sqrt{n} \left(W(2^j) - \E d_{\textnormal{tfBm}}^2(2^j,0)\right)\right)_{j=j_1,\ldots,j_m} \stackrel{d}\to \mathcal{N}(\mathbf{0},F), \quad n \rightarrow \infty,
\end{equation}
where $F=(F_{j,j'})_{j,j'=j_1,\ldots,j_m}$,
\begin{equation}\label{e:Fjj'}
F_{j,j'} =  2^{\min\{j,j'\}+1}\sum_{r\in\Z} \zeta^2(j,j', r 2^{\min\{j,j'\}}),
\end{equation}
and
\begin{equation}\label{e:zeta_def}
\zeta(j,j',r): = \frac{\sigma^2\Gamma^2\left( H+\frac{1}{2} \right)}{2\pi} \int^{\pi}_{-\pi} e^{ixr} \sum_{\ell \in \bbZ}  \frac{ {\mathcal H}_j(x)\overline{{\mathcal H}_{j'}(x)}}{|\lambda^2 + (x + 2 \pi \ell)^2|^{H+\frac{1}{2}}}dx.
\end{equation}
\end{lemma}
\begin{proof}
The proof is similar to that of Theorem 3.1 in Abry and Didier \cite{abry:didier:2018}. We provide the details for reader's convenience. The argument makes use of the Cram\'{e}r-Wold device and the Lindeberg central limit theorem.  Define the vector of wavelet coefficients
$$
V_n:=\left(d_{\textnormal{tfBm}}(2^{j_1},1),\ldots,d_{\textnormal{tfBm}}(2^{j_1},n_j),d_{\textnormal{tfBm}}(2^{j_2},1),\ldots,d_{\textnormal{tfBm}}(2^{j_m},n_{j_m})\right)^\top\in  \R^\nu,
$$
where $\nu={n_{j_1}+\ldots+n_{j_m}}$. 
Fix $\mathbf{a} = (a_1,\ldots,a_m)\in\R^m$, and consider the statistic $U_n:=\sum_{j=j_1}^{j_m}a_jW(2^j)$.   Our goal is to establish the limiting distribution of the statistic $\sqrt{n}U_n$.  Write
$$
D_n=\diag \Big(\underbrace{\frac{a_1}{n_{j_1}},\ldots,\frac{a_1}{n_{j_1}}}_{n_{j_1} \text{ terms}},\underbrace{\frac{a_2}{n_{j_2}},\ldots,\frac{a_2}{n_{j_2}}}_{n_{j_2} \text{ terms}},\ldots,\underbrace{\frac{a_m}{n_{j_m}},\ldots,\frac{a_m}{n_{j_m}}}_{n_{j_m}\text{ terms}}\Big) \in \R^{\nu\times \nu},
$$
where $\diag$ denotes a diagonal matrix. Then, $V_n^\top D_n V_n=U_n$. Write $\Gamma_n = \E V_nV_n^\top$, and consider the spectral decomposition of the matrix $\Gamma^{1/2}_nD_n\Gamma^{1/2}_n = O_n^\top\Lambda_n O_n$.  Note $\Gamma^{-1/2}_nV_n \deq Z_n\sim \mathcal{N}(0,I)$, where $I$ is the $\nu\times\nu$ identity matrix. Thus, $U_n=V_n^\top D_n V_n \deq Z_n^\top\Gamma_n^{1/2}D_n\Gamma_n^{1/2}Z_n = Z_n^\top O_n^\top\Lambda_n O_nZ_n \deq Z_n^\top\Lambda_n Z_n =\sum_k \xi_k^2\lambda_k$, where $\lambda_k$ is the $k^{\textnormal{th}}$ diagonal entry of the matrix $\Lambda_n$, and $\xi_k$ is the $k^{\textnormal{th}}$ component of the vector $Z_n$.  Note that $\E(\sum_k \xi_k^2\lambda_k) = \sum_k \lambda_k$. 
By the Lindeberg central limit theorem, it suffices to show that
\begin{equation}\label{e:lambda-max=o(sqrt(VarTn))}
\max_{k = 1,\hdots, \nu}\lambda_k=:\lambda_{\max}(n) =  o\left(\sqrt{\Var U_n}\right),\quad n\to\infty
\end{equation}

Note that
 \begin{align*}
 \Var U_n = \sum_{j=j_1}^{j_m}\sum_{j'=j_1}^{j_m} {a_ja_{j'}} \Cov(W(2^j),W(2^{j'})).
  \end{align*}
  For arbitrary $j,j',k,k'$, the Isserlis theorem (e.g., Vignat \cite{vignat:2012}) implies that
\begin{equation}\label{e:Isserlis=>Cov=2(E)^2}
  \Cov( d^2(2^j,k),d^2(2^{j'},k')) = 2 \left(\E d(2^{j'},k')d(2^j,k)\right)^2.
 \end{equation}
  Hence,
  \begin{equation}\label{e:Cov(W(2^j),W(2^j')} \Cov( W(2^j),W(2^{j'})) = 
    \frac{2}{n_jn_{j'}} \sum_{k=1}^{n_j}\sum_{k'=1}^{n_{j'}} \left(\E d_{\textnormal{tfBm}}(2^{j'},k')d_{\textnormal{tfBm}}(2^j,k)\right)^2.
  \end{equation}
When $|2^jk-2^{j'}k'|$ is large enough,  Propositions \ref{p:wavelet_spectrum} and \ref{p:tfbm_decorr} imply that $\zeta(j,j',2^jk-2^{j'}k')=|\E d_{\textnormal{tfBm}}(2^j,k)d_{\textnormal{tfBm}}(2^{j'},k')|\leq C|2^jk-2^{j'}k'|^{H-1/2}e^{-\lambda|2^jk-2^{j'}k'|}$, regardless of the sample size $n$.
Consequently, $\sum_{r\in\Z}\zeta^2(j,j',r)<\infty$. Thus, as $n\to\infty$, i.e., $n_j,n_{j'}\to\infty$, by Lemma B.3 in Abry and Didier \cite{abry:didier:2018},
 \begin{equation}\label{e:where_B3_is_used}
  \frac{2^{j+j'}}{n}\sum_{k=1}^{n_j}\sum_{k'=1}^{n_{j'}} \zeta^2(j,j',{2^jk-2^{j'}k'}) \to 2^{\min\{j,j'\}}\sum_{r\in\Z} \zeta^2(j,j',r2^{\min\{j,j'\}}).
  \end{equation}
  This gives
  $$
 n\Cov( W(2^j),W(2^{j'})) \to  2^{\min\{j,j'\}+1}\sum_{r\in\Z} \zeta^2(j,j',r2^{\min\{j,j'\}}), \quad n \rightarrow \infty,
  $$
  i.e.,
  \begin{equation}\label{e:nVarTn_conv}
  n \Var U_n  \to   \sum_{j=j_1}^{j_m}\sum_{j'=j_1}^{j_m} a_ja_{j'} F_{j,j'}, \quad n \rightarrow \infty.
  \end{equation}
Thus, $\sqrt{\Var U_n} = \mathcal O(\frac{1}{\sqrt{n}})$. Now, for $\lambda_{\max}(n)$ as on the left-hand side of \eqref{e:lambda-max=o(sqrt(VarTn))},
 $$
 \lambda_{\max}(n) = \sup_{\|u\|=1} u^\top\Gamma_n^{1/2}D_n\Gamma_n^{1/2}u\leq \max_{k} |(D_n)_{kk}| \sup_{\|u\|=1}u^\top\Gamma_nu.
 $$
Further noting that as $n\to\infty$, $\max_{k} |(D_n)_{kk}| \sim C n^{-1}$, we need only to study the eigenvalues of the matrix $\Gamma_n$.  Consider the bound
\begin{equation}\label{ANmaxbound}
 \sup_{\|u\|=1}u^\top\Gamma_nu  \leq \max_{\ell=1,\ldots,\nu} \sum_{k=1}^{\nu} |(\Gamma_n)_{k,\ell}|
 \leq m \max_{j,j'=j_1,\ldots,j_m}\max_{k=1,\ldots,n_j} \sum_{k'=1}^{n_{j'}} | \E d_{\textnormal{tfBm}}(2^j,k)d_{\textnormal{tfBm}}(2^{j'},k')|.
 \end{equation}
Now, take $M$ large so that \eqref{e:approx_decorrelation} holds whenever $|2^jk-2^{j'}k'|>M$.  Note that, for any $k'$, $j$, and $j'$,
$$
\#\{k:|2^jk-2^{j'}k'|<M\} =  \#\{k:  2^{j'-j}k'- 2^{-j}M<k<2^{j'-j}k'+ 2^{-j}M\}
$$
$$
\leq 2^{1-j}M+1  \leq 2M+1.
$$
So, for large $n$, all but finitely many terms satisfy the bound \eqref{e:approx_decorrelation}, and the maximum number of which do not satisfy the bound is independent of $k'$, $j$ and $j'$.  Furthermore, for arbitrary $k,k',j,j'$,
$$
|\E d_{\textnormal{tfBm}}(2^j,k)d_{\textnormal{tfBm}}(2^{j'},k')| \leq |\E d_{\textnormal{tfBm}}^2(2^j,0)\E d_{\textnormal{tfBm}}^2(2^{j'},0)|^{1/2}$$$$ \leq {\sup_{j,j'}}\hspace{1mm}|\E d_{\textnormal{tfBm}}^2(2^j,0)\E d_{\textnormal{tfBm}}^2(2^{j'},0)|^{1/2}<\infty
$$
by \eqref{e:disc_slrd2}, i.e., the finitely many terms in the sum on the right-hand side of \eqref{ANmaxbound} that do not satisfy the bound are bounded by a constant irrespective of $n$.  Thus, for some $C>0$,
\begin{equation}\label{e:mmaxmax_sum<infty}
m \max_{j,j'=j_1,\ldots,j_m}\max_{k=1,\ldots,n_j} \sum_{k'=1}^{n_{j'}} | \E d_{\textnormal{tfBm}}(2^j,k)d_{\textnormal{tfBm}}(2^{j'},k')|
\leq C \sum_{r\neq 0} |r|^{H-1/2}e^{-\lambda|r|} <\infty.
 \end{equation}
Hence, by \eqref{ANmaxbound} and \eqref{e:mmaxmax_sum<infty} the eigenvalues of the sequence of matrices $\{\Gamma_n\}_{n\in\N}$ are bounded, and thus $\lambda_{\text{max}}(n)\sim C n^{-1}$. This establishes \eqref{e:sqrt(nj)(W(2^j)-Ed^2(j,0))->N(0,sigma^2)}. $\Box$\\
\end{proof}

\noindent {\sc Proof of Theorem \ref{t:asympt}}: First, we show (\emph{i}). Let $f_n$ be as in \eqref{e:def_estimator}. Also, {let $B(j)$ be as in \eqref{e:B(j)}}. Note that, for each $j$,
{\begin{equation}\label{e:B(j)=o(1/n)}
B(j) = \mathcal O(n^{-1}).
\end{equation}
(e.g. Olver et al \cite{olver2010}, p.\ 140, expression 5.11.2). Let ${\widehat{\boldsymbol \theta}}_{n}$ be a} minimum of the objective function $f_n$ restricted to $\overline{U}_0$. Since $f_n$ is continuous and the closure $\overline{U}_0$ of the neighborhood is compact, such minimum exists.  By \eqref{e:B(j)=o(1/n)} and Lemma \ref{l:ANwspec},
$f_n({\boldsymbol \theta}_0)\pto 0$, and thus
$$
0\leq f_n({\widehat{\boldsymbol \theta}}_{n})\leq f_n({\boldsymbol \theta}_0)\pto 0, \quad n \rightarrow \infty,
$$
where the inequality stems from the fact that ${\boldsymbol \theta}_0 \in \overline{U}_0$. Hence, $ f_n({\widehat{\boldsymbol \theta}}_{n})\pto 0$.

 By way of contradiction, suppose the convergence does not hold, i.e., there exists some $c>0$ and a subsequence $\widehat{\boldsymbol \theta}_{n_k}$ with $\P(\|\widehat{\boldsymbol \theta}_{n_k}-{\boldsymbol \theta}_0\|>c)>\varepsilon>0$ for all $k$. 
If necessary, further refine the subsequence so that $\P(f_{n_k}(\widehat{\boldsymbol \theta}_{n_k})<\frac{1}{k})>1-\frac{\varepsilon}{2}$. Let $\Omega^k_0 = \{\omega: \|\widehat{{\boldsymbol \theta}}_{n_k}-{\boldsymbol \theta}_0\|>c,\ \  f_{n_k}(\widehat{\boldsymbol \theta}_{n_k})<\frac{1}{k} \}$ and note that $\P(\Omega^k_0)\geq \varepsilon/2$ for every $k$.

Write $S=\{\bth\in \overline{U}_0:\|\bth-\bth_0\|\geq c\}$. Define the real-valued function
$$
G({\boldsymbol \theta}) = \left\| \left(\E_{\bth}d^2(2^{j_1},0),\ldots,\E_{\bth}d^2(2^{j_m},0)\right)-\left(\E_{\bth_0}d^2(2^{j_1},0),\ldots,\E_{\bth_0}d^2(2^{j_m},0)\right) \right\|.
$$
Since $\bth\mapsto (\E_{\bth} d^2(2^{j_1},0), \ldots,\E_{\bth} d^2(2^{j_m},0))$ is continuous, then so is $G({\boldsymbol \theta})$. Therefore, since $S$ is compact, $\inf_{\bth\in S} G({\boldsymbol \theta})$ attained at some $\bth' \in S$.  However, at $\bth'$, by the identifiability assumption \eqref{e:local_identifiability}, $ G({\boldsymbol \theta}') > 0$. In other words, for some $c' > 0$,
$$
\inf_{\bth\in S} G({\boldsymbol \theta})>c'.
$$
In particular, for every $\bth\in S$ there is some $j^*=j^*(\bth)\in \{j_1,\ldots,j_m\}$ such that \begin{equation}\label{e:|eta_j*(theta)-eta_j*(theta0)|>delta}
|\eta_{j^*}(\bth)-\eta_{j^*}(\bth_0)| = \left|\E_{\bth}d^2(2^{j^*},0) - \E_{\bth_0}d^2(2^{j^*},0)\right|>\delta:=\frac{c'}{\sqrt{m}}.
\end{equation}
Thus, for every $k$ and every $\omega\in\Omega^k_0$, we have $|\eta_{j^*}(\widehat{\boldsymbol \theta}_{n_k})-\eta_{j^*}({\boldsymbol \theta}_0)|>\delta$.

Furthermore, for $\delta > 0$ as in \eqref{e:|eta_j*(theta)-eta_j*(theta0)|>delta}, Lemma \ref{l:ANwspec} implies that there is some $N$ such that, for $n_k>N$, $\P(\bigcap^{j_m}_{j=j_1}\{|\log_2W(2^j)- \eta_j({\boldsymbol \theta}_0)|<\delta/2\})>1-\frac{\varepsilon}{4}$. So, let $\Omega^k_1 = \Omega^k_0\cap \bigcap^{j_m}_{j=j_1}\{|\log_2W(2^j)- \eta_j({\boldsymbol \theta}_0)|<\delta/2\}$, and note that $\P(\Omega^k_1)\geq \frac{\varepsilon}{4}$ when $n_k>N$, i.e., when $k$ is sufficiently large. Let $w^* = \min_{j\in\{j_1,\ldots,j_m\}}w_j$. For all $k$ and every $\omega\in\Omega^k_1$,
$$
\frac{1}{k}  >\sum_{j={j_1}}^{j_m} w_j\left(\log_2W(2^j)-\eta_j({\boldsymbol \theta}_{n_k})\right)^2 = \sum_{j={j_1}}^{j_m} w_j\left(\log_2W(2^j)- \eta_j({\boldsymbol \theta}_0)+ \eta_j({\boldsymbol \theta}_0)-\eta_j({\boldsymbol \theta}_{n_k})\right)^2
$$
$$
\geq w^*\left(\log_2W(2^{j^*})- \eta_{j^*}({\boldsymbol \theta}_0)+ \eta_{j^*}({\boldsymbol \theta}_0)-\eta_{j^*}({\boldsymbol \theta}_{n_k})\right)^2
$$
$$
 \geq w^*(\delta/2 - |\eta_{j^*}({\boldsymbol \theta}_0)-\eta_{j^*}({\boldsymbol \theta}_{n_k})|)^2\geq w^*( \delta/2)^2.
$$
By taking the limit $k \rightarrow \infty$, we arrive at a contradiction. Hence, $\widehat{{\boldsymbol \theta}}_n \stackrel{P}\rightarrow {\boldsymbol \theta}_0$. This shows \eqref{e:consistency}.\\

We proceed to establish (\emph{ii}).  For notational simplicity, write $\bth=(\bth_1,\bth_2,\bth_3)=(H,\lambda,\sigma^2)$, $D = \frac{d}{d\bth},$ and $\partial_\ell=\frac{\partial}{\partial\bth_\ell}$, $\ell=1,2,3$. Recall that $f_n$ denotes the objective function \eqref{e:def_estimator}. By Lemma \ref{l:ANwspec}, for $i,\ell=1,2,3$,
$$
\partial_i\partial_\ell f_n(\bth) = 2\sum_{j=j_1}^{j_m}w_j\Big[ \left(\eta_j(\bth)-\log_2W(2^j)\right)\partial_i\partial_\ell \eta_j(\bth) + \partial_i \eta_j(\bth)\partial_\ell \eta_j(\bth)\Big]
$$
\begin{equation}\label{e:Hessian_conv}
\pto 2\sum_{j=j_1}^{j_m} \frac{w_j\partial_\ell \E_{\bth} d^2(2^j,0)\partial_i\E_{\bth} d^2(2^j,0) }{[\E_{\bth} d^2(2^j,0)]^2 \log^2(2)}, \quad i,\ell = 1,2,3, \quad n \rightarrow \infty.
\end{equation}
In view of \eqref{e:Hessian_conv}, as a consequence of assumption \eqref{e:det>0} there is some constant $c=c(\bth_0)$ such that
\begin{equation}\label{D2psd}
\P(\det D^2f_n(\bth_0)>c>0)\to 1, \quad n \rightarrow \infty.
\end{equation}
Furthermore, by considering a Taylor expansion of $D f_n(\bth)$ around $\bth_0$ and rearranging the terms,
\begin{equation}\label{e:Taylor}
 D^2 f_n({\boldsymbol \theta}_0)({\boldsymbol \theta}-{\boldsymbol \theta}_0) = Df_n({\boldsymbol \theta}) -Df_n({\boldsymbol \theta}_0) +o(\|{\boldsymbol \theta}-{\boldsymbol \theta}_0\|) \hspace{1mm}\in \bbR^3.
\end{equation}
At $\bth=\widehat{\bth}_n$, $Df_n(\widehat{\bth}_n)=0$ with probability going to 1 since $\Xi$ is an open set (see \eqref{e:theta0_in_intXi}). Now, by \eqref{D2psd}, with probability going to 1, the Hessian matrix $D^2f_n(\bth_0)$ is invertible, and $\|(D^2f_n(\bth_0))^{-1}\|<c'<\infty$.  Writing $g(\bth)$ for the remainder term in \eqref{e:Taylor}, for all $\omega\in \{\omega:\det D^2f_n(\bth_0)>c>0\}$, we have $\|(D^2f_n(\bth_0))^{-1}g(\widehat{\bth}_n)\|\leq c'\|g( \widehat{\bth}_n)\| = o_p(\|\widehat{\boldsymbol \theta}_n-{\boldsymbol \theta}_0\|)$, i.e., 
\begin{equation}\label{e:th3taylor}
\sqrt{n}(\widehat{\boldsymbol \theta}_n-{\boldsymbol \theta}_0)  = -\left[D^2f_n(\bth_0)\right]^{-1}\sqrt{n}Df_n({\boldsymbol \theta}_0) + o_p(\sqrt{n}\|\widehat{\boldsymbol \theta}_n-{\boldsymbol \theta}_0\|).
\end{equation}
Note that, for any fixed $j$, \eqref{e:B(j)=o(1/n)} implies that $\sqrt{n_j}B(j)\to 0$. Thus, by Lemma \ref{l:ANwspec} and Slutsky's theorem,
$$
\left(\sqrt{n_j}\left[\log_2 W(2^j)-\eta_j(\bth_0)\right]\right)_{j=j_1}^{j_m}
$$
$$
= \left(\sqrt{n_j}\left[\frac{\left(W(2^j)-\E d^2(2^j,0)\right)}{ \E d^2(2^j,0) \log 2} + o_p\left(W(2^j)-\E d^2(2^j,0)\right)\right]-\sqrt{n_j}B(j)\right)_{j=j_1,\ldots,j_m}
$$
$$
\stackrel{d}\to \mathcal{N}\left(0,\left(\frac{F_{j,j'}}{\E d^2(2^j,0)\E d^2(2^{j'},0)\log^2 2}\right)_{j,j'=j_1,\ldots, j_m}\right), \quad n \rightarrow \infty.
$$
Hence, by the Cram\'er-Wold device,
\begin{equation}\label{e:sqrt(n)Dfn_conv}
\sqrt{n}Df_n(\bth_0) = 2\sum_{j=j_1}^{j_m}\sqrt{n}\left(\eta_j(\bth_0)-\log_2 W(2^j)\right)w_jD\eta_{j}(\bth_0) \stackrel{d}\to \mathcal{N}(\mathbf{0},\Upsilon(\bth_0)), \quad n \rightarrow \infty.
\end{equation}
In \eqref{e:sqrt(n)Dfn_conv}, the asymptotic covariance matrix is given by
$$
\Upsilon(\bth_0) =\left(\sum_{j=j_1}^{j_m}\sum_{j'=j_1}^{j_m} \frac{2^{\frac{j+j'}{2}+2}w_jw_{j'}\partial_\ell\Edj\partial_i\Edjp}{[\Edj\Edjp]^2\log^4 2 }F_{j,j'}\right)_{i,\ell=1,2,3},
$$
where $F_{j,j'}$ is as in \eqref{e:Fjj'}. From \eqref{e:Hessian_conv} and the smoothness of $D^2f_n$, we may write
\begin{equation}\label{e:inv_second_deriv_conv}
\left[D^2f_n(\bth_0)\right]^{-1} \pto Q(\bth_0)^{-1},
\end{equation}
where
$$
Q(\bth)=\left(2\sum_{j=j_1}^{j_m} \frac{w_j\partial_\ell\E_{\bth} d^2(2^j,0)\partial_i\E_{\bth} d^2(2^j,0) }{[\E_{\bth} d^2(2^j,0) ]^2 \log^2(2)}\right)_{i,\ell =1,2,3}.
$$
Thus, by \eqref{e:sqrt(n)Dfn_conv} and taking limits in $n$ in \eqref{e:th3taylor}, expression \eqref{e:asympt_normality} holds, where
$$
\Sigma(\bth)= Q(\bth)^{-1}\Upsilon(\bth)(Q(\bth)^{-1})^{-\top}. \quad \Box
$$

\subsection{Hypothesis testing}

Throughout this section, recall that $d_{\text{fBm}}(2^j,k)$ and $d_{\text{tfBm}}(2^j,k)$ as in \eqref{e:dtfBm_dfBm} denote the (discrete time) wavelet coefficients of fBm and tfBm, respectively.

Let $\widetilde{H}(2^j,2^{j'})$ be the wavelet regression estimator defined in \eqref{eq:def_wavereg} and let $\{a(n)\}_{n\in\N}$ be a dyadic sequence of integers satisfying the growth condition \eqref{e:seq_decay_assumption}.
The asymptotic distribution of wavelet variances or wavelet-based estimators is discussed and established in many works (for instance, Veitch and Abry \cite{veitch:abry:1999}, Bardet \cite{bardet:2002}, Bardet et al.\ \cite{Bardet2003a}, Moulines et al.\ \cite{moulines:roueff:taqqu:2008}). In particular, for a pair of octaves $j_1,j_2$,
\begin{equation}\label{e:fbm_a_AN}
\sqrt{\frac{n}{a(n)}}\left(\widetilde{H}\left(a(n)2^{j_1},a(n)2^{j_2}\right)-H\right) \stackrel{d}\to \mathcal{N}(0,\gamma_0^2), \quad n \rightarrow \infty,
\end{equation}
where
\begin{equation}\label{e:gamma0}
\gamma_0^2 = \frac{1}{4\log^2 2}\sum_{j = j_1,j_2}\sum_{j'= j_1,j_2}\varpi_j\varpi_{j'}\frac{\widetilde{G}_{j,j'}}{\varrho_\infty(j,j,0)\varrho_\infty(j',j',0)},
\end{equation}
\begin{equation}\label{e:varrho_inf}
\varrho_{\infty}(j,j',r) :=\frac{\sigma^22^{(j+j')/2}}{C^2(H)}\int_{\R} \frac{\widehat{\psi}(2^jy)\overline{\widehat\psi(2^{j'}y)}}{|y|^{2H+1}}e^{iyr}dx
\end{equation}
(cf.\ expression \eqref{e:Cov(W(a2j)W(a2j'))_fbm}), and in \eqref{e:gamma0}, the matrix $\widetilde{G} = (\widetilde{G}_{j,j'})_{j,j'=j_1,j_2}$ is given by
\begin{equation}\label{e:Gtilde}
\widetilde{G}_{j,j'} = 2^{\min\{j,j'\}+1}\sum_{r\in\Z}
\varrho_\infty^2(j,j',r2^{\min\{j,j'\}}),
\end{equation}

 However, the asymptotic distribution of wavelet variances of fBm \textit{both} in discrete time and for fixed scales is not generally emphasized. Since we use such asymptotic distribution in Lemma \ref{Hest_AN}, for the reader's convenience we provide it in the following lemma.
\begin{lemma}\label{fBm_AN}
Let $j_1,\ldots,j_m$ be a set of octaves, and let $W(2^j)$ be the sample wavelet variance of a fBm with Hurst parameter $H$ and variance $\sigma^2$. Then, {under condition \eqref{e:Npsi>alpha}},
\begin{equation}\label{e:sqrt(n)(WfBm-EWfBm)->Normal}
\sqrt{n}\left(W(2^j)-\E d_{\textnormal{fBm}}^2(2^j,0)\right)_{j=j_1,\ldots, j_m}\stackrel{d}\to \mathcal{N}(0,G),
\end{equation}
as $n \rightarrow \infty$, where $G=(G_{j,j'})_{j,j'=j_1,\ldots,j_m}, G_{j,j'}=2^{\min\{j,j'\}+1}\sum_{r\in\Z} \varrho^2(j,j',r2^{\min\{j,j'\}})$, and
\begin{equation}\label{e:varrho(j,j',r)_def}
\varrho(j,j',r):= \frac{\sigma^2}{C^2(H)} \int^{\pi}_{-\pi} e^{ixr} \sum_{\ell \in \bbZ}  \frac{ {\mathcal H}_j(x)\overline{{\mathcal H}_{j'}(x)}}{|x + 2 \pi \ell |^{2H+1}}dx.
\end{equation}
\end{lemma}
\begin{proof}
Following the proof of Lemma \ref{l:ANwspec}, we need to establish the expression of the matrix $G$, and control the quantity
\begin{equation}\label{e:fbm_eigbound}
m \max_{j,j'=j_1,\ldots,j_m}\max_{k=1,\ldots,n_j} \sum_{k'=1}^{n_{j'}} | \E d_{\textnormal{fBm}}(2^j,k)d_{\textnormal{fBm}}(2^{j'},k')|.
\end{equation}

Note that, by \eqref{e:fbm_decorr}, for any fixed $j,j'$,
$$
|\E d_{\text{fBm}}(2^j,k)d_{\text{fBm}}(2^{j'},k')| = \mathcal{O}\left(\frac{1}{|2^j k - 2^{j'}k'|^{2N_{\psi}-2H}}\right), \quad |2^j k - 2^{j'}k'|\to\infty
$$
Write $\varrho(j,j',2^jk-2^{j'}k')=\E d_{\text{fBm}}(2^j,k)d_{\text{fBm}}(2^{j'},k')$. Then, $\sum_{r}\varrho^2(j,j',r)<\infty$. So, by the proof of Lemma \ref{l:ANwspec}
$$
n\Cov( W(2^j),W(2^{j'})) \to 2^{\min\{j,j'\}+1}\sum_{r\in\Z} \varrho^2(j,j',r2^{\min\{j,j'\}}).
$$
Writing $j_*=\max\{j_1,\ldots,j_m\}$, we have
$$
|\E d_{\textnormal{fBm}}(2^j,k)d_{\textnormal{fBm}}(2^{j'},k')| \leq |\E d_{\textnormal{fBm}}^2(2^j,0)\E d_{\textnormal{fBm}}^2(2^{j'},0)|^{1/2} \leq \E d_{\textnormal{fBm}}^2(2^{j_{*}},0).
$$
Similarly as in the proof of Lemma \ref{l:ANwspec}, for large $n$ all but finitely many of the terms in the sum \eqref{e:fbm_eigbound} satisfy the bound \eqref{e:fbm_decorr}. Thus, {by \eqref{e:Npsi>alpha}}, for some constant $C$ expression \eqref{e:fbm_eigbound} is bounded by
$$
m \max_{j,j'=j_1,\ldots,j_m}\max_{k=1,\ldots,n_j} \sum_{k'=1}^{n_{j'}} | \E d_{\text{fBm}}(2^j,k)d_{\text{fBm}}(2^{j'},k')|
\leq C \sum_{r\neq 0} \frac{1}{|r|^{2N_{\psi}-2H}}<\infty.
$$
Hence, the limit in distribution \eqref{e:sqrt(n)(WfBm-EWfBm)->Normal} holds. $\Box$\\
\end{proof}

The following lemma establishes a bound, uniform in $k,k'$, on the correlation between wavelet coefficients when one of the octaves grows indefinitely.
\begin{lemma} \label{single_scale_limit_decorr}
Let $k,k',j,j'$ be fixed, and let $\{a(n)\}_{n\in\N}$ be a dyadic sequence as in \eqref{e:seq_decay_assumption}. Then,
\begin{equation}\label{e:EdtfBm(2^j,k)dtfBm(2^{j'},k')}
|\E d_{\textnormal{tfBm}}(2^j,k)d_{\textnormal{tfBm}}(a(n)2^{j'},k')| = o\Big(\frac{1}{\sqrt{a(n)}}\Big),\quad n\to\infty
\end{equation}
Moreover, under condition \eqref{e:Npsi>alpha} and $\alpha$ as in \eqref{e:psihat_is_slower_than_a_power_function}, we have
\begin{equation}\label{e:EdfBm(2^j,k)dfBm(2^{j'},k')}
|\E d_{\textnormal{fBm}}(2^j,k)d_{\textnormal{fBm}}(a(n)2^{j'},k')| \leq
 \mathcal{O}\left(a(n)^{2H+\frac{1}{2}-N_\psi}\right) + o\Big(\frac{1}{\sqrt{a(n)}}\Big),\quad n\to\infty
\end{equation}
In \eqref{e:EdtfBm(2^j,k)dtfBm(2^{j'},k')} and \eqref{e:EdfBm(2^j,k)dfBm(2^{j'},k')}, the residual terms $\mathcal{O}(\cdot)$, $o(\cdot)$ do not depend on $k,k'$.
\end{lemma}
\begin{proof}
We first show \eqref{e:EdtfBm(2^j,k)dtfBm(2^{j'},k')}. Write $f(x) = \sum_{\ell \in\Z}\sigma^2\Gamma^2\left( H+\frac{1}{2} \right)|\lambda^2 + (x + 2 \pi \ell)^2|^{-(H+\frac{1}{2})}$, and let
$$
{j'_n}=j'+\log_2 a(n).
$$
Then,
$$
\E d_{\textnormal{tfBm}}(2^j,k)d_{\textnormal{tfBm}}(a(n)2^{j'},k') = \int_{-\pi}^\pi e^{ix(2^jk-a(n)2^{j'}k')}\mathcal{H}_j(x)\overline{\mathcal{H}_{{j'_n}}(x)}f(x)dx.
$$
Consider $\mathcal{H}_{j_n'}(x)$ as in \eqref{e:Hj_expanded_pois_formula} with $j_n'$ in place of $j$. Thus, using a change of variable $y=x+2\pi k$,
$$
\E d_{\textnormal{tfBm}}(2^j,k)d_{\textnormal{tfBm}}(a(n)2^{j'},k') = 2^{j_n'/2}\int_\R e^{iy(2^jk-a(n)2^{j'}k')}\mathcal{H}_j(y)\overline{\widehat{{\phi}}(y)} \widehat{\psi}(2^{{j'_n}}y)f(y)dx.
$$
After a change of variable $v=2^{j'_n}y$, and noting by Fourier inversion $|\widehat{\phi}(x)| \leq  \|\phi\|_{L^1}$, a.e., we obtain
$$
2^{j_n'/2}|\E d_{\textnormal{tfBm}}(2^j,k)d_{\textnormal{tfBm}}(a(n)2^{j'},k')| \leq
\int_{\R} \left|\mathcal{H}_j(2^{-j'_n}v)\widehat{\phi}(2^{-j'_n}v) \widehat{\psi}(v)f(2^{-j'_n}v)\right|dv
$$
$$
\leq  \|\widehat{\phi}\|_{\infty}\|f\|_{\infty}\int_{\R} |\mathcal{H}_j(2^{-j_n'}v) \widehat{\psi}(v)|dv\to 0,\quad n\to \infty,
$$
by the dominated convergence theorem, since \eqref{eq:vanish} implies
$$\mathcal{H}_{j}(0)=\sum_{\ell\in\Z}h_{j,\ell}=2^{-j/2}\int_\mathbb{R} \psi(2^{-j}t)\sum_{\ell\in \bbZ}\phi(t+\ell) dt=0.$$
Noting that $2^{j'_n/2}\sim a^{1/2}(n)$, the conclusion follows.

We now turn to \eqref{e:EdfBm(2^j,k)dfBm(2^{j'},k')}. Write $g(x) = \sum_{\ell\in\Z}\frac{\sigma^2/C^2(H)}{|x+2\pi \ell|^{2H+1}}$.  Then,
$$
|\E d_{\textnormal{fBm}}(2^j,k)d_{\textnormal{fBm}}(a(n)2^{j'},k')| = \left|\int_{-\pi}^\pi e^{ix(2^jk-a(n)2^{j'}k')}\mathcal{H}_j(x)\overline{\mathcal{H}_{{j'_n}}(x)}g(x)dx\right|
$$
$$
= \left|\int_{-\pi}^\pi e^{i2^{j'_n}x((2^j/a(n)2^{j'})k-k')}\mathcal{H}_j(x)\overline{\mathcal{H}_{{j'_n}}(x)}g(x)dx\right|
$$
$$
= \left|\int_{-\pi 2^{j'_n}}^{\pi 2^{j'_n}} e^{iy((2^j/2^{j'_n})k-k')}\mathcal{H}_j\Big(\frac{y}{2^{j'_n}}\Big)\overline{\mathcal{H}_{{j'_n}}\Big(\frac{y}{2^{j'_n}}\Big)}
g\Big(\frac{y}{2^{j'_n}}\Big)\frac{dy}{2^{j'_n}}\right|
$$
$$
\leq 2^{-j'_n}\int_{-\pi 2^{j'_n}}^{\pi 2^{j'_n}}\left|\mathcal{H}_j\Big(\frac{y}{2^{j'_n}}\Big)\overline{\mathcal{H}_{{j'_n}}\Big(\frac{y}{2^{j'_n}}\Big)}
g\Big(\frac{y}{2^{j'_n}}\Big)\right| dy
$$
$$
\leq 2^{j'_n 2H}\int_{-\pi 2^{j'_n}}^{\pi 2^{j'_n}}\left|\mathcal{H}_j\Big(\frac{y}{2^{j'_n}}\Big)\overline{\mathcal{H}_{{j'_n}}\Big(\frac{y}{2^{j'_n}}\Big)}
\frac{\sigma^2/C^2(H)}{|y|^{2H+1}}\right| dy
$$
\begin{equation}\label{e:EdfBmdfBm_bound}
+ 2^{-j'_n}\int_{-\pi 2^{j'_n}}^{\pi 2^{j'_n}}\Big|\mathcal{H}_j\Big(\frac{y}{2^{j'_n}}\Big)\overline{\mathcal{H}_{{j'_n}}\Big(\frac{y}{2^{j'_n}}\Big)}\Big|
\sum_{\ell \neq 0}\frac{\sigma^2/C^2(H)}{|(y/2^{j'_n})+2\pi \ell|^{2H+1}} dy,
\end{equation}
where we made the change of variable $y = 2^{j'_n}x$. Recast the first term in the bound \eqref{e:EdfBmdfBm_bound} as
\begin{equation}\label{e:EdfBmdfBm_bound_1st_term}
2^{j'_n 2H}\Big\{\int_{|y| \leq 1} + \int_{1 < |y| \leq \pi 2^{j'_n}} \Big\}\left|\mathcal{H}_j\Big(\frac{y}{2^{j'_n}}\Big)\overline{\mathcal{H}_{{j'_n}}\Big(\frac{y}{2^{j'_n}}\Big)}
\frac{\sigma^2/C^2(H)}{|y|^{2H+1}}\right| dy.
\end{equation}
By Moulines et al.\ \cite{moulines:roueff:taqqu:2007:JTSA}, relation (78),
\begin{equation}\label{e:MoulinesetalJTSA_rel78}
|{\mathcal H}_{j}(x)| \leq C 2^{j/2} \frac{|2^j x|^{N_{\psi}}}{(1 + 2^j |x|)^{\alpha + N_{\psi}}}, \quad x \in [-\pi,\pi).
\end{equation}
Therefore, for $|y| \leq 2^{j'_n}\pi$,
$$
\Big|{\mathcal H}_{j}\Big(\frac{y}{2^{j'_n}}\Big){\mathcal H}_{j'_n}\Big(\frac{y}{2^{j'_n}}\Big)\Big|
$$
\begin{equation}\label{e:|Hj(y/2^j')Hj'(2/2^j')|_bound}
\leq C 2^{j'_n/2}\frac{|y/2^{j'_n}|^{N_{\psi}}}{(1 + 2^j|y/2^{j'_n}|)^{\alpha + N_{\psi}}}\frac{|2^{j'_n}y/2^{j'_n}|^{N_{\psi}}}{(1 + 2^{j'_n}|y/2^{j'_n}|)^{\alpha + N_{\psi}}}
\leq C \frac{(2^{j'_n})^{1/2}}{(2^{j'_n})^{N_{\psi}}}  \frac{y^{2N_{\psi}}}{(1 + |y|)^{\alpha + N_{\psi}}}.
\end{equation}
Thus, by \eqref{e:|Hj(y/2^j')Hj'(2/2^j')|_bound}, the first term in \eqref{e:EdfBmdfBm_bound_1st_term} is bounded by

\begin{equation}\label{e:EdfBmdfBm_bound_1st_term_a}
C \frac{(2^{j'_n})^{2H+1/2}}{(2^{j'_n})^{N_\psi}} \int^{1}_{-1} \frac{1}{|y|^{2H+1}} \frac{|y|^{2 N_{\psi}}}{(1+|y|)^{\alpha + N_{\psi}}}dy \leq C \frac{(2^{j'_n})^{2H+1/2}}{(2^{j'_n})^{N_{\psi}}}.
\end{equation}
Also, {by \eqref{e:|Hj(y/2^j')Hj'(2/2^j')|_bound}}, the second term in \eqref{e:EdfBmdfBm_bound_1st_term} is bounded by
\begin{equation}\label{e:EdfBmdfBm_bound_1st_term_b}
C \frac{(2^{j'_n})^{2H+1/2}}{(2^{j'_n})^{N_{\psi}}} \int^{2^{j'_n}\pi}_{1} \frac{y^{2N_{\psi}}}{(1 + |y|)^{\alpha + N_{\psi}}} \frac{1}{|y|^{2H+1}}dy$$$$
\leq C \frac{(2^{j'_n})^{2H+1/2}}{(2^{j'_n})^{N_{\psi}}}\times \begin{cases}
C' +  (2^{j'_n})^{N_{\psi}-\alpha -2H}, & N_{\psi}-\alpha -2H \neq 0;\\
C' +  \log 2^{j'_n} ,  & N_{\psi}-\alpha -2H  = 0.\\
\end{cases}
\end{equation}
By relations \eqref{e:EdfBmdfBm_bound_1st_term_a} and \eqref{e:EdfBmdfBm_bound_1st_term_b}, \eqref{e:EdfBmdfBm_bound_1st_term} is bounded by
\begin{equation}\label{e:EdfBmdfBm_bound_1st_term_overall_bound}
\begin{cases}
 {C}\left((2^{j'_n})^{\frac{1}{2}-\alpha} + (2^{j'_n})^{ 2H+\frac{1}{2}-N_\psi}\right), & N_{\psi}-\alpha -2H \neq 0;\\
   {C} \log (2^{j'_n}) \hspace{1mm} (2^{j'_n})^{ 2H+\frac{1}{2}-N_\psi}, & N_{\psi}-\alpha -2H = 0.
 \end{cases}
\end{equation}
On the other hand, $\sup_{|y| \leq \pi 2^{j'_n}}\sum_{\ell \neq 0} 1/|(y/2^{j'_n})+ 2 \pi \ell|^{2H+1} < \infty$. Let $\delta\in (0,\alpha-1)$. By \eqref{e:|Hj(y/2^j')Hj'(2/2^j')|_bound}, the second term in \eqref{e:EdfBmdfBm_bound} is bounded by
\begin{equation}\label{e:EdfBmdfBm_bound_2nd_term_overall_bound}
C\frac{2^{-j'_n}2^{j'_n/2}}{(2^{j'_n})^{N_{\psi}}}\int^{2^{j'_n}\pi}_{-2^{j'_n}\pi} \frac{|y|^{2N_\psi}}{(1 + |y|)^{\alpha + N_{\psi}}}dy
\leq \frac{C(2^{j'_n})^{N_\psi-\delta}}{(2^{j'_n})^{N_{\psi}+1/2}}\int^{2^{j'_n }\pi}_{-2^{j'_n }\pi} \frac{|y|^{N_\psi+\delta}}{(1 + |y|)^{\alpha + N_{\psi}}}dy \leq \frac{C}{(2^{j'_n})^{\frac{1}{2}+\delta}}.
\end{equation}
Noting that
\begin{equation}\label{e:1/2-alpha<-1/2}
\frac{1}{2}-\alpha<-\frac{1}{2},
\end{equation}
and that if $N_{\psi}-\alpha -2H = 0$,
$$
\log 2^{j'_n} \hspace{1mm} (2^{j'_n})^{ 2H+\frac{1}{2}-N_\psi} =  \log 2^{j'_n} \hspace{1mm} (2^{j'_n})^{ \frac{1}{2}-\alpha} = o(a(n)^{-\frac{1}{2}}),
$$
relations \eqref{e:EdfBmdfBm_bound_1st_term_overall_bound} and \eqref{e:EdfBmdfBm_bound_2nd_term_overall_bound} yield \eqref{e:EdfBm(2^j,k)dfBm(2^{j'},k')}. $\Box$\\
\end{proof}

The following lemma, which is needed in the proof of Lemma \ref{Hest_AN}, establishes a limit used to obtain the asymptotic covariance of the sample wavelet variances $W(a(n)2^j)$ for both tfBm and fBm when the scale factor $a(n)$ goes to infinity. {The result is classical for fBm, but since in this case the additional condition \eqref{e:Npsi>alpha} is needed, we provide the argument for the reader's convenience.}
\begin{lemma}\label{Cov_scale_to_infty}
Let $j,j'$ be fixed, and let $\{a(n)\}_{n\in\N}$ be a dyadic sequence as in \eqref{e:seq_decay_assumption}. Also let
$$
n_{a,j} = \frac{n}{a(n)2^j}
$$
and $n_{a,j'}$ be the number of wavelet coefficients at scales $a(n)2^j$ and $a(n)2^{j'}$, respectively. 
\begin{itemize}
\item[(i)] As $n\to\infty$,

\begin{equation}\label{e:Cov(W(a2j)W(a2j'))_tfbm}
\frac{a(n)}{n}\sum_{k=1}^{n_{a,j}}\sum_{k'=1}^{n_{a,j'}}\left(\E d_{\textnormal{tfBm}}(a(n)2^j,k)d_{\textnormal{tfBm}}(a(n)2^{j'},k')\right)^2\to 2^{\min\{j,j'\}-(j+j')}\sum_{r\in\mathcal \Z}\zeta_{\infty}^2(j,j',r 2^{\min\{j,j'\}}),
\end{equation}
where
\begin{equation}\label{e:zeta_inf}
\zeta_{\infty}(j,j',r)=\frac{\sigma^2\Gamma^2(H+\frac{1}{2})}{2\pi}\sum_{\ell\in\Z}\frac{2^{(j+j')/2}}{|\lambda^2+(2\pi\ell)^2|^{H+\frac{1}{2}}}\int_{\R} \widehat{\psi}(2^jy)\overline{\widehat\psi(2^{j'}y)}e^{ixr}dx
\end{equation}
\item[(ii)]  Under condition \eqref{e:Npsi>alpha}, as $n\to\infty$,

\begin{equation}\label{e:Cov(W(a2j)W(a2j'))_fbm}
\frac{a(n)}{n}\sum_{k=1}^{n_{a,j}}\sum_{k'=1}^{n_{a,j'}}\left(\frac{\E d_{\textnormal{fBm}}(a(n)2^j,k)d_{\textnormal{fBm}}(a(n)2^{j'},k')}{a(n)^{2H+1}}\right)^2\to 2^{\min\{j,j'\}-(j+j')}\sum_{r\in\mathcal \Z}\varrho^2_{\infty}(j,j',r 2^{\min\{j,j'\}}),
\end{equation}
where $\varrho_\infty$ is given by \eqref{e:varrho_inf}.
\end{itemize}
\end{lemma}
\begin{proof}
We start with $(i)$.  Write $\mathcal{H}_j^{(0)}(x)= 2^{j/2}\widehat\phi(x)\widehat\psi(2^jx)$.  Note
$$
\left|{\mathcal H}_{j}(x)\overline{{\mathcal H}_{j'}(x)} -\mathcal{H}_j^{(0)}(x)\overline{\mathcal{H}_{j'}^{(0)}(x)}\right|
$$
\begin{equation}\label{e:Hj-Hj0_1}
\leq \left|{\mathcal H}_{j}(x)\right|\left| \mathcal{H}_{j'}(x)-\mathcal{H}_{j'}^{(0)}(x)\right| + \left|{\mathcal H}_{j'}^{(0)}(x)\right|\left| \mathcal{H}_{j}(x)-\mathcal{H}_{j}^{(0)}(x)\right|.
\end{equation}
By Moulines et al.\ \cite{moulines:roueff:taqqu:2007:JTSA}, Proposition 3, for all $x\in[-\pi,\pi)$,
$$
|\mathcal{H}_{j}(x) |\leq C2^{j/2} \frac{|2^jx|^{N_\psi}}{(1+2^j|x|)^{N_\psi+\alpha}},\quad |\mathcal{H}^{(0)}_{j}(x) |\leq C'2^{j/2} \frac{|2^jx|^{N_\psi}}{(1+2^j|x|)^{N_\psi+\alpha}},
$$
and
$$
\left| \mathcal{H}_{j}(x)-\mathcal{H}_{j}^{(0)}(x)\right| \leq C'' 2^{j(\frac{1}{2}-\alpha)}|x|^{N_\psi}.
$$
Thus, \eqref{e:Hj-Hj0_1} is bounded by
\begin{equation}\label{e:Hj-Hj0_bound1}
\widetilde C\frac{2^{j/2} |2^jx|^{N_\psi}}{(1+2^j|x|)^{N_\psi+\alpha}}2^{j'(\frac{1}{2}-\alpha)}|x|^{N_\psi} + \widetilde C'\frac{2^{j'/2} |2^{j'}x|^{N_\psi}}{(1+2^{j'}|x|)^{N_\psi+\alpha}}2^{j(\frac{1}{2}-\alpha)}|x|^{N_\psi} =: M_{j,j'}(x).
\end{equation}
Consider the function $\zeta(j,j',r)$ as in \eqref{e:zeta_def}, i.e.,
\begin{equation}\label{e:zeta(j,j',r)}
\zeta(j,j',r) = \bbE d_{\textnormal{tfBm}}(2^j,k)d_{\textnormal{tfBm}}(2^{j'},k'), \quad r = 2^{j}k - 2^{j'}k'.
\end{equation}
Write $f(x) = \sum_{\ell\in\Z}\frac{(2\pi)^{-1}\sigma^2\Gamma^2(H+\frac{1}{2})}{|\lambda^2+(x+2\pi\ell)^2|^{H+\frac{1}{2}}}$, $j_n = \log_2(a(n)2^j)$, $j'_n = \log_2(a(n)2^{j'})$. Then,
$$
\E d(a(n)2^j,k)d(a(n)2^{j'},k')  = \int_{-\pi}^\pi {\mathcal H}_{j_n}(x)\overline{{\mathcal H}_{j'_n}(x)}e^{ixa(n)(2^jk-2^{j'}k')}f(x)dx
$$
$$
= \zeta\left(j_n,j'_n,a(n)(2^jk-2^{j'}k')\right).
$$
By \eqref{e:Hj-Hj0_1} and \eqref{e:Hj-Hj0_bound1},
$$
\left| \int_{-\pi}^\pi {\mathcal H}_{j_n}(x)\overline{{\mathcal H}_{j'_n}(x)}e^{ixa(n)(2^jk-2^{j'}k')}f(x)dx -  \int_{-\pi}^\pi {\mathcal H}^{(0)}_{j_n}(x)\overline{{\mathcal H}^{(0)}_{j'_n}(x)}e^{ixa(n)(2^jk-2^{j'}k')}f(x)dx \right|
$$
$$
\leq\int_{-\pi}^\pi  \left|{\mathcal H}_{j}(x)\overline{{\mathcal H}_{j'}(x)} -\mathcal{H}_j^{(0)}(x)\overline{\mathcal{H}_{j'}^{(0)}(x)}\right|f(x)dx
$$
\begin{equation}\label{e:H-H^0_lim}
\leq \int_{-\pi}^\pi M_{j_n,j'_n}(x) f(x)dx
\to 0,\qquad n\to\infty.
\end{equation}
In view of \eqref{e:1/2-alpha<-1/2}, the limit in \eqref{e:H-H^0_lim} holds by dominated convergence, since the righthand integrand is bounded by
$$
\|f\|_\infty \left(\widetilde C\frac{2^{j_n/2} |2^{j_n}x|^{N_\psi}}{(2^{j_n}|x|)^{N_\psi+\alpha}}2^{j'_n(\frac{1}{2}-\alpha)}|x|^{N_\psi} + \widetilde C'\frac{2^{j'_n/2} |2^{j'_n}x|^{N_\psi}}{(2^{j'_n}|x|)^{N_\psi+\alpha}}2^{j_n(\frac{1}{2}-\alpha)}|x|^{N_\psi} \right)
$$
$$
=\|f\|_\infty \left(\widetilde C''2^{(j_n+j'_n)(\frac{1}{2}-\alpha)}|x|^{N_\psi-\alpha} \right)
$$
$$
={\|f\|_\infty }\left(\widetilde C''a(n)^{1-2\alpha}2^{(j+j')(\frac{1}{2}-\alpha)}|x|^{N_\psi-\alpha} \right) \leq \|f\|_\infty \left(\widetilde C''2^{(j+j')(\frac{1}{2}-\alpha)}|x|^{N_\psi-\alpha} \right).
$$
On the other hand, by a change of variable $y = a(n)x$,
 $$
  \int_{-\pi}^\pi {\mathcal H}^{(0)}_{j_n}(x)\overline{{\mathcal H}^{(0)}_{j'_n}(x)}e^{ixa(n)(2^jk-2^{j'}k')}f(x)dx
  $$$$=  2^{(j_n+j_n')/2}\int_{-\pi}^\pi |\widehat{\phi}(x)|^2 \widehat{\psi}(a(n)2^jx)\overline{\widehat\psi(a(n)2^{j'}x)}e^{ixa(n)(2^jk-2^{j'}k')}f(x)dx
  $$
  $$
  =2^{(j+j')/2}\int_{-a(n)\pi}^{a(n)\pi} \Big| \widehat{\phi}\Big(\frac{y}{a(n)}\Big) \Big|^2 \widehat{\psi}(2^jy)\overline{\widehat\psi(2^{j'}y)}e^{iy(2^jk-2^{j'}k')}f\Big(\frac{y}{a(n)}\Big)dy
  $$
\begin{equation}\label{e:H^0_lim}
  \to2^{(j+j')/2}\int_{\R} \widehat{\psi}(2^jy)\overline{\widehat\psi(2^{j'}y)}e^{iy(2^jk-2^{j'}k')}f(0)dx,\quad n\to\infty,
\end{equation}
   where the limit is a consequence of dominated convergence. Thus, combining \eqref{e:H-H^0_lim} and \eqref{e:H^0_lim}, for any $r=2^jk-2^{j'}k'$, we have
\begin{equation}\label{zeta_to_zeta_inf}
   \lim_{n\to\infty}\zeta(j_n,j'_n,a(n)r) = 2^{(j+j')/2}f(0)\int_{\R} \widehat{\psi}(2^jy)\overline{\widehat\psi(2^{j'}y)}e^{ixr}dx=\zeta_{\infty}(j,j',r).
\end{equation}
Now consider the sum
 $$
 \frac{a(n)}{n}\sum_{k=1}^{n_{a,j}}\sum_{k'=1}^{n_{a,j'}}\zeta^2(j_n,j'_n,a(n)(2^jk-2^{j'}k')).
 $$
Let $\nu_n = \frac{n}{a(n)2^{j+j'}}$, and note $\nu_n\to\infty$ as $n\to\infty$.  For a given $r\in\mathcal{R}:= 2^j\N-2^{j'}\N=2^{\min\{j,j'\}}\Z$, by Lemma B.2 in Abry and Didier, 
 the number $\xi(n)$ of pairs $(k,k')\in\Z^2$  with $2^jk-2^{j'}k'= r$ and
   $$
  1\leq k \leq {2^{j'}\nu_n},\qquad 1\leq k'\leq {2^{j}\nu_n}
    $$
  satisfies
    $$
    \lim_{n\to\infty}\frac{\xi(n)}{\nu_n} = \textnormal{gcd}(2^j,2^{j'}) = 2^{\min\{j,j'\}}.
  $$
Let $B(n)=\{r: 2^jk-2^{j'}k'=r,\quad   1\leq k \leq {2^{j'}\nu_n},\quad 1\leq k'\leq {2^{j}\nu_n}\}$.  Note that {$B(n)\uparrow \mathcal R$} as $n\to\infty$, and thus
$$
  \frac{a(n)}{n}\sum_{k=1}^{n_{a,j}}\sum_{k'=1}^{n_{a,j'}}\zeta^2(j_n,j'_n,a(n)(2^jk-2^{j'}k'))=\frac{a(n)}{n}\sum_{k=1}^{2^{j'}\nu_n}\sum_{k'=1}^{{2^{j}\nu_n}}\zeta^2(j_n,j'_n,a(n)(2^jk-2^{j'}k'))
 $$$$
  = \frac{1}{2^{j+j'}\nu_n}\sum_{r\in B(n)}  \xi(n) \zeta^2(j_n,j'_n,a(n)r) \to 2^{\min\{j,j'\}-(j+j')}\sum_{r\in\mathcal R}\zeta^2_{\infty}(j,j',r)
  $$
  \begin{equation}\label{e:2^minjj'sum_r_zeta^2infty}
  = 2^{\min\{j,j'\}-(j+j')}\sum_{r\in\Z}\zeta^2_{\infty}(j,j',r2^{\min\{j,j'\}}).
\end{equation}
In \eqref{e:2^minjj'sum_r_zeta^2infty}, the limit holds due to dominated convergence. In fact, consider $j$, $j'$ and $a(n)$ such that $a(n)r \geq 2T\max\{a(n)2^j,a(n)2^{j'}\}$. By Proposition \ref{p:tfbm_decorr} and in view of \eqref{e:zeta(j,j',r)},
$$
|\zeta(j_n,j'_n,a(n)r)| \leq C2^{\frac{j_n+j'_n}{2}}|a(n)r|^{H-1/2}e^{-\lambda|a(n)r|}= C2^{\frac{j+j'}{2}}a(n)|a(n)r|^{H-1/2}e^{-\lambda|a(n)r|}
$$
$$
= C' \frac{{a(n)^{H+\frac{1}{2}}|r|^{H-\frac{1}{2}}}}{e^{\lambda|a(n)r|}}\leq \frac{C'}{\lambda^2} a(n)^{H-\frac{3}{2}}r^{H-\frac{5}{2}} \leq C''r^{H-\frac{5}{2}}.
$$
Therefore, $|\zeta(j_n,j'_n,a(n)r)|^2 \leq C''^2r^{2H-5}\in \ell^1(\Z)$. Moreover, for every $n$,
$$
\#\{r: ra(n) < 2T\max\{a(n)2^j,a(n)2^{j'}\}\} =\#\{r: r < 2T\max\{2^j,2^{j'}\}\} \leq 4T\max\{2^j,2^{j'}\}+1.
$$
In other words, the sequence $|\zeta(j_n,j'_n,a(n)r)|^2$ is bounded by an absolutely summable sequence in $r$ that does not depend on $n$. Hence, the limit in \eqref{e:2^minjj'sum_r_zeta^2infty} holds, as claimed.\\

We now turn to $(ii)$. Consider $\varrho(j,j',r)$ as in \eqref{e:varrho(j,j',r)_def}, namely,
\begin{equation}\label{e:varrho(j,j',r)}
\varrho(j,j',r)= \frac{\sigma^2}{C^2(H)} \int^{\pi}_{-\pi} e^{ixr} \sum_{\ell \in \bbZ}  \frac{ {\mathcal H}_j(x)\overline{{\mathcal H}_{j'}(x)}}{|x + 2 \pi \ell |^{2H+1}}dx.
\end{equation}
For notational simplicity, let $g(x) = \sum_{\ell\in\Z}\frac{\sigma^2/C^2(H)}{|x+2\pi \ell|^{2H+1}}$.  We first show that
$$
\frac{1}{a(n)^{2H+1}}\Big| \int_{-\pi}^\pi {\mathcal H}_{j_n}(x)\overline{{\mathcal H}_{j'_n}(x)}e^{ixa(n)(2^jk-2^{j'}k')}g(x)dx
$$
\begin{equation}\label{e:H-H_0_1_fbm}
-  \int_{-\pi}^\pi {\mathcal H}^{(0)}_{j_n}(x)\overline{{\mathcal H}^{(0)}_{j'_n}(x)}e^{ixa(n)(2^jk-2^{j'}k')}g(x)dx \Big| \to 0, \quad n \rightarrow \infty.
\end{equation}
In fact, by \eqref{e:Hj-Hj0_1} and \eqref{e:Hj-Hj0_bound1},
$$
\frac{1}{a(n)^{2H+1}}\left| \int_{-\pi}^\pi \Big[{\mathcal H}_{j_n}(x)\overline{{\mathcal H}_{j'_n}(x)}- {\mathcal H}^{(0)}_{j_n}(x)\overline{{\mathcal H}^{(0)}_{j'_n}(x)}\Big] e^{ixa(n)(2^jk-2^{j'}k')}g(x)dx \right|
$$
$$
\leq \frac{1}{a(n)^{2H+1}}\int_{-\pi}^\pi M_{j_n,j'_n}(x)g(x)dx
$$
\begin{equation}\label{e:H0-H_fbm_M_jj'_terms}
= \frac{1}{a(n)^{2H+1}}\int_{-\pi}^\pi M_{j_n,j'_n}(x) \sum_{\ell\neq 0}\frac{\sigma^2/C^2(H)}{|x+2\pi \ell|^{2H+1}}dx + \frac{1}{a(n)^{2H+1}}\int_{-\pi}^\pi M_{j_n,j'_n}(x)\frac{\sigma^2/C^2(H)}{|x|^{2H+1}}dx.
\end{equation}
Since the function $\sum_{\ell\neq 0}\frac{\sigma^2/C^2(H)}{|x+2\pi \ell|^{2H+1}}dx$ in the integrand of the first term of \eqref{e:H0-H_fbm_M_jj'_terms} is bounded, by following a similar argument as in part $(i)$ the first term in \eqref{e:H0-H_fbm_M_jj'_terms} goes to 0.  For the second term, note that, by a change of variable $y = a(n)x$,
$$
 \frac{1}{a(n)^{2H+1}}\int_{-\pi}^\pi M_{j_n,j'_n}(x)\frac{\sigma^2/C^2(H)}{|x|^{2H+1}}dx =  C \frac{1}{a(n)}\int_{-a(n)\pi}^{a(n)\pi} M_{j_n,j'_n}\left(\frac{y}{a(n)}\right)\frac{1}{|y|^{2H+1}}dy
$$
$$
=\frac{1}{a(n)} \int_{-a(n)\pi}^{a(n)\pi}\bigg(\widetilde C\frac{2^{j_n/2} |2^{j_n}y/a(n)|^{N_\psi}}{(1+2^{j_n}|y/a(n)|)^{N_\psi+\alpha}}2^{j_n'(\frac{1}{2}-\alpha)}\left|\frac{y}{a(n)}\right|^{N_\psi}$$$$ +\  \widetilde C'\frac{2^{j_n'/2} |2^{j_n'}y/a(n)|^{N_\psi}}{(1+2^{j'_n}|y/a(n)|)^{N_\psi+\alpha}}2^{j_n(\frac{1}{2}-\alpha)}\left|\frac{y}{a(n)}\right|^{N_\psi}\bigg)\frac{1}{|y|^{2H+1}}dy
$$
$$
=\frac{1}{a(n)^{\alpha+N_\psi}}\int_{-a(n)\pi}^{a(n)\pi}\left(\widetilde C\frac{2^{j/2} |2^jy|^{N_\psi}}{(1+2^j|y|)^{N_\psi+\alpha}}2^{j'(\frac{1}{2}-\alpha)}|y|^{N_\psi} + \widetilde C'\frac{2^{j'/2} |2^{j'}y|^{N_\psi}}{(1+2^{j'}|y|)^{N_\psi+\alpha}}2^{j(\frac{1}{2}-\alpha)}|y|^{N_\psi}\right)\frac{1}{|y|^{2H+1}}dy
$$
$$
\leq \frac{1}{a(n)^{\alpha+N_\psi}}\int_{-a(n)\pi}^{a(n)\pi} C'\frac{ |y|^{2N_\psi}}{(1+C''|y|)^{N_\psi+\alpha}}\frac{1}{|y|^{2H+1}}dy
$$
$$
\leq \frac{1}{a(n)^{\alpha}}\int_{\R} C'\frac{ |y|^{N_\psi}}{(1+C''|y|)^{N_\psi+\alpha}}\frac{1}{|y|^{2H+1}}dy\to 0, \quad n\to\infty,
$$
{where $\int_{\R} \frac{ |y|^{N_\psi}}{(1+C''|y|)^{N_\psi+\alpha}}\frac{1}{|y|^{2H+1}}dy < \infty$ by \eqref{e:Npsi>alpha}}. In other words, \eqref{e:H-H_0_1_fbm} holds.

On the other hand,
 $$
 \frac{1}{a(n)^{2H+1}}  \int_{-\pi}^\pi {\mathcal H}^{(0)}_{j_n}(x)\overline{{\mathcal H}^{(0)}_{j'_n}(x)}e^{ixa(n)(2^jk-2^{j'}k')}g(x)dx
  $$
  $$=  \frac{2^{(j_n+j_n')/2}}{a(n)^{2H+1}} \int_{-\pi}^\pi |\widehat{\phi}(x)|^2 \widehat{\psi}(a(n)2^jx)\overline{\widehat\psi(a(n)2^{j'}x)}e^{ixa(n)(2^jk-2^{j'}k')}g(x)dx
  $$
\begin{equation}\label{e:H-H0_fbm_terms1and2}
   \frac{2^{(j_n+j_n')/2}}{a(n)^{2H+1}} \int_{-\pi}^\pi |\widehat{\phi}(x)|^2 \widehat{\psi}(a(n)2^jx)\overline{\widehat\psi(a(n)2^{j'}x)}e^{ixa(n)(2^jk-2^{j'}k')}\sum_{\ell\neq 0}\frac{\sigma^2/C^2(H)}{|x+2\pi\ell|^{2H+1}}dx
$$$$
   + \frac{2^{(j_n+j_n')/2}}{a(n)^{2H+1}} \int_{-\pi}^\pi |\widehat{\phi}(x)|^2 \widehat{\psi}(a(n)2^jx)\overline{\widehat\psi(a(n)2^{j'}x)}e^{ixa(n)(2^jk-2^{j'}k')}\frac{\sigma^2/C^2(H)}{|x|^{2H+1}}dx.
\end{equation}
The first term in \eqref{e:H-H0_fbm_terms1and2} vanishes as $n\to\infty$ since the integrand is bounded ($\|\widehat{\phi}\|_\infty\leq\|\phi\|_{L^1}$). For the second term in \eqref{e:H-H0_fbm_terms1and2}, by a change of variable $y = a(n)x$,
  $$
2^{(j+j')/2}\int_{-a(n)\pi}^{a(n)\pi} \Big|\widehat{\phi}\Big(\frac{y}{a(n)}\Big)\Big|^2 \widehat{\psi}(2^jy)\overline{\widehat\psi(2^{j'}y)}e^{iy(2^jk-2^{j'}k')}\frac{\sigma^2/C^2(H)}{|y|^{2H+1}}dy
  $$
  \begin{equation}\label{e:varrho_limit}
  \to \frac{\sigma^22^{(j+j')/2}}{C^2(H)}\int_{\R} \frac{\widehat{\psi}(2^jy)\overline{\widehat\psi(2^{j'}y)}}{|y|^{2H+1}}e^{iy(2^jk-2^{j'}k')}dx,\quad n\to\infty,
  \end{equation}
  by dominated convergence. Consider the function $\varrho$ as in \eqref{e:varrho(j,j',r)}. Thus, for every $r=2^jk-2^{j'}k'$, \eqref{e:H-H_0_1_fbm} and \eqref{e:varrho_limit} imply that
\begin{equation}\label{varrho_to_varrho_inf}
\frac{\varrho(j_n,j'_n,a(n)r)}{a(n)^{2H+1}} \to \varrho_\infty(j,j',r) = \frac{\sigma^22^{(j+j')/2}}{C^2(H)}\int_{\R} \frac{\widehat{\psi}(2^jy)\overline{\widehat\psi(2^{j'}y)}}{|y|^{2H+1}}e^{iyr}dx, \quad n \rightarrow \infty.
\end{equation}
  Now, consider the sum
  $$
  \frac{a(n)}{n}\sum_{k=1}^{n_{a,j}}\sum_{k'=1}^{n_{a,j'}}\left(\frac{\E d_{\textnormal{fBm}}(a(n)2^j,k)d_{\textnormal{fBm}}(a(n)2^{j'},k')}{a(n)^{2H+1}}\right)^2. 
  $$
  Similar to the argument in part $(i)$, Proposition \ref{p:tfbm_decorr} implies that, for all but finitely many $r$,
  $$
  \left|\frac{\varrho(a(n)j,a(n)j',ra(n))}{a(n)^{2H+1}}\right|^2\leq \left(\frac{C}{a(n)^{2H+1}}\frac{2^{(j_n+j_n')(N_{\psi}+\frac{1}{2})}}{|a(n)r|^{2N_{\psi}-2H}}\right)^2
  =\left( C\frac{2^{(j+j')(N_{\psi}+\frac{1}{2})}}{|r|^{2N_{\psi}-2H}}\right)^2 \in \ell^1(\Z).
  $$
  Thus, as in part $(i)$, by dominated convergence, as $n\to\infty$,
  $$
   \frac{a(n)}{n}\sum_{k=1}^{n_{a,j}}\sum_{k'=1}^{n_{a,j'}}\left(\frac{\E d_{\textnormal{fBm}}(a(n)2^j,k)d_{\textnormal{fBm}}(a(n)2^{j'},k')}{a(n)^{2H+1}}\right)^2 \to 2^{\min\{j,j'\}-(j+j')}\sum_{r\in\Z}\varrho^2_{\infty}(j,j',r2^{\min\{j,j'\}}),
  $$
  as was to be shown.
$\Box$

\end{proof}

Recall that the hypotheses $H_0$ and $H_1$ are defined by \eqref{e:H0}. In the following lemma, we establish the asymptotic joint normality of the log-regression estimator \eqref{eq:def_wavereg} and the $M$-estimator \eqref{eq:def_Mest_fbm} under both $H_0$ and $H_1$, where the scale range on which the former estimator is based grows indefinitely. Under $H_1$, namely, for tfBm, the wavelet spectrum will resemble that of an fBm for initial octaves provided $\lambda$ is not too large. Thus, over a pair of initial octaves $j_1,j_2$, $\{\E d_{\text{tfBm}}^2(2^j,0)\}_{j = j_1,j_2}\approx\{\E d_{\text{fBm}}^2(2^j,0)\}_{j= j_1,j_2}$.  Therefore, for a tfBm with parameters $H,\lambda$, $\widehat{H}(2^{j_1},2^{j_2})\approx H$ over the initial scales but $\widetilde{H}(a(n)2^{j_3},a(n)2^{j_4})<H$ over large scales due to the presence of the tempering factor. This statement is made rigorous in the proof of part $(ii)$ of the lemma below.

\begin{lemma}\label{Hest_AN}  Let $j_1<j_2<j_3<j_4$ be a set of octaves, and let $(\widehat{H}_n,\widehat{\sigma^2_n}):=(\widehat{H}(2^{j_1},2^{j_2}),\widehat{\sigma^2}(2^{j_1},2^{j_2}))$ be the M-estimator defined in \eqref{eq:def_Mest_fbm}, and let $\widetilde{H}_n:=\widetilde{H}(a(n)2^{j_3},a(n)2^{j_4})$ be the wavelet regression estimator as defined in \eqref{eq:def_wavereg}, where the sequence $\{a(n)\}_{n\in\N}$ satisfies \eqref{e:seq_decay_assumption}.
\begin{itemize}
\item[$(i)$]Under $H_0$ and condition \eqref{e:Npsi>alpha}, 
\begin{equation}\label{e:test_stat_under_H0}
\sqrt{n}\Big((\widehat{H}_n,\widehat{\sigma^2_n},a^{-1/2}(n)\widetilde{H}_n) - (H,\sigma^2,H) \Big) \stackrel{d}\to \mathcal{N}(0,M), \quad n \rightarrow \infty.
\end{equation}
In \eqref{e:test_stat_under_H0}, $M$ is the $3\times3$ matrix given by
\begin{equation}\label{e:matrix_M}
M = \begin{pmatrix}
N & \mathbf{0}\\
\mathbf{0} & \gamma_0^2
\end{pmatrix},
\end{equation}
$\gamma_0^2$ is given by \eqref{e:gamma0}, and $N$ is the $2\times 2$ matrix
$$
N=N(\bth_0)$$$$=\widetilde Q(\bth_0)^{-1}\left(\sum_{j= j_1,j_2} \sum_{j'= j_1,j_2}\frac{2^{\frac{j+j'}{2}+2}w_jw_{j'}\partial_\ell \Edjf \partial_i\Edjpf}{[\Edjf\Edjpf]^2\log^4 2 }G_{j,j'}\right)_{i,\ell=1,2}
$$
\begin{equation}\label{e:N(theta0)}
\cdot \hspace{1mm}(\widetilde{Q}(\bth_0)^{-1})^{-\top},
\end{equation}
where
\begin{equation}\label{e:def_widetilde_Q}
\widetilde{Q}(\bth)= \left(2\sum_{j=j_1,j_2} \frac{w_j\partial_\ell\E_{\bth} d_{\textnormal{fBm}}^2(2^j,0)\partial_i\E_{\bth} d_{\textnormal{fBm}}^2(2^j,0) }{[\E_{\bth} d_{\textnormal{fBm}}^2(2^j,0)]^2 \log^2(2)}\right)_{i,\ell =1,2},
\end{equation}
and $(G_{j,j'})_{j,j'= j_1,j_2}$ are the entries of the matrix $G$ in \eqref{e:sqrt(n)(WfBm-EWfBm)->Normal}.
\item[$(ii)$] Alternatively, under $H_1$ and assumption \eqref{e:slope_assumption}, there exists a $\bth'=(H',\sigma'^2)$, with $H'<H$, such that
\begin{equation}\label{e:Hhat,sigma^2hat,Htilde_converge}
\sqrt{n}((\widehat{H}_n,\widehat{\sigma^2_n},a^{-1/2}(n)\widetilde{H}_n)-(H',\sigma'^2,-1/2)) \stackrel{d}\to \mathcal{N}(0,L), \quad n \rightarrow \infty.
\end{equation}
In \eqref{e:Hhat,sigma^2hat,Htilde_converge}, $L$ is the $3\times3$ symmetric positive semidefinite matrix
$$
L = \begin{pmatrix}
K & \mathbf{0}\\
\mathbf{0} & \gamma_1^2
\end{pmatrix},
$$
where
$\gamma_1^2 := \frac{1}{4\log^2(2)\zeta_\infty^2(0,0,0)}\sum_{j,j'=j_3}^{j_4}\varpi_j\varpi_{j'}\widetilde{F}_{j,j'}$,
\begin{equation}\label{e:F-tilde_j,j'}
\widetilde{F}_{j,j'} =  2^{\min\{j,j'\}+1}\sum_{r\in\Z} \zeta_{\infty}^2(j,j',{r2^{\min\{j,j'\}}}),
\end{equation}
$\zeta_\infty(j,j',r)$ is given by \eqref{e:zeta_inf},
\begin{equation}\label{e:matrix_K}
K=K(\bth')
$$$$= \widetilde{Q}(\bth')^{-1} \left(\sum_{j= j_1,j_2}\sum_{j'= j_1,j_2}\frac{2^{\frac{j+j'}{2}+2}w_jw_{j'}\partial_\ell\Edjfa\partial_i\Edjpfa}{[\Edjfa\Edjpfa]^2\log^4 2 }F_{j,j'}\right)_{i,\ell=1,2}\widetilde{Q}(\bth')^{-\top}
\end{equation}
and $\widetilde{Q}(\bth)$ is given by \eqref{e:def_widetilde_Q}, and each entry of the matrix $(F_{j,j'})_{j,j'=j_1,j_2}$ is given by \eqref{e:Fjj'}.
\end{itemize}
\end{lemma}

\begin{proof}
To establish $(i)$, first we will show that, with appropriate normalization, the vectors
\begin{equation}\label{e:wavelet_var_fixed_coarse_scales_H0}
\left(\sqrt{n}W(2^{j_1}),\sqrt{n}W(2^{j_2})\right) \textnormal{ and }\left(\sqrt{n/a(n)}W(a(n)2^{j_3}),\ldots,\sqrt{n/a(n)}W(a(n)2^{j_4})\right)
\end{equation}
are asymptotically decorrelated. Afterwards, we will establish the joint asymptotic normality of said vectors, from which the expression \eqref{e:test_stat_under_H0} will follow using Slutsky's theorem. Note that, by \eqref{e:fbm_decorr}, for all large $n$, if $|2^jk-a(n)2^{j'}k'|\geq2^{j'+1}a(n)T$, we have
$$
|\E d_{\text{fBm}}(2^j,k)d_{\text{fBm}}(a(n)2^{j'},k')| \leq C\frac{a(n)^{N_\psi+\frac{1}{2}}}{|2^jk-a(n)2^{j'}k'|^{2N_{\psi}-2H}}\leq C' a(n)^{\frac{1}{2}-N_\psi + 2H}.
$$
Define the normalized sample wavelet variance
$$
W_a(2^j):=\frac{1}{n_{a,j}}\sum_{k=1}^{n_{a,j}}\frac{d^2(a(n)2^j,k)}{a(n)^{2H+1}},
$$
where
$$
n_{a,j} = \frac{n}{a(n)2^j}
$$
is the number of coefficients available at scale $a(n)2^j$ (cf.\ \eqref{e:nj=n/2^j}).

Let
\begin{equation}\label{e:b(n)}
b(n)=\lfloor a(n)\rfloor^{\frac{5}{2}}.
\end{equation}
Note that, for fixed $k'$,
$$
\#\{k: |2^j k - a(n)2^{j'}k'|\leq b(n)\}
$$
$$
=\#\{k: a(n)2^{j'}k'-b(n)\leq 2^jk\leq a(n)2^{j'}k' +b(n)\}\leq 2 \lfloor b(n)/2^j \rfloor +1.
$$
Consider the set of indices
\begin{equation}\label{e:A_n}
A_n(b(n)):= \{(k,k'):|2^j k - a(n)2^{j'}k'|\leq b(n)\},
\end{equation}
where, for notational simplicity, we write $A_n$. Then,
\begin{equation}\label{bn_count}
\sum_{k=1}^{n_j}\sum_{k'=1}^{n_{a,j'}}1_{A_n}(k,k') \leq n_{a,j'} ( 2 \lfloor b(n)/2^j \rfloor +1).
\end{equation}
Recast
$$
\Cov\left(\sqrt{n}W(2^j),\sqrt{n/a(n)}W_a(2^{j'})\right) = \frac{n}{a^{1/2}(n)}\Cov\left(W(2^j),W_a(2^{j'})\right)
$$
$$
= \frac{2na(n)^{-(2H+3/2)}}{n_{j}n_{a,j'}}\sum_{k=1}^{n_j}\sum_{k'=1}^{n_{a,j'}}\left(\E d_{\textnormal{fBm}}(2^j,k)d_{\textnormal{fBm}}(a(n)2^{j'},k')\right)^2
$$
$$
= \frac{2na(n)^{-(2H+3/2)}}{n_{j}n_{a,j'}}\sum_{k=1}^{n_j}\sum_{k'=1}^{n_{a,j'}}(1_{A_n}(k,k')+1_{A_n^c}(k,k'))\left(\E d_{\textnormal{fBm}}(2^j,k)d_{\textnormal{fBm}}(a(n)2^{j'},k')\right)^2,
$$
where we make use of the Isserlis relation (cf.\ \eqref{e:Isserlis=>Cov=2(E)^2}). On the set $A_n$ as in \eqref{e:A_n}, by applying relations \eqref{bn_count} and \eqref{e:EdfBm(2^j,k)dfBm(2^{j'},k')} we obtain
$$
 \frac{2na(n)^{-(2H+3/2)}}{n_{j}n_{a,j'}}\sum_{k=1}^{n_j}\sum_{k'=1}^{n_{a,j'}}1_{A_n}(k,k')\left(\E d_{\textnormal{fBm}}(2^j,k)d_{\textnormal{fBm}}(a(n)2^{j'},k')\right)^2
$$

\begin{equation}\label{e:Cov(sqrt(n)W(2j),sqrt(n/a)W(a2j'))->0_1a}
\leq Ca(n)^{-(2H+3/2)}( 2 \lfloor b(n)/2^j \rfloor +1)\left[ \mathcal{O}\left(a(n)^{2H+\frac{1}{2}-N_\psi}\right) + o\left(a(n)^{-\frac{1}{2}}\right)\right]^2.
\end{equation}
On one hand, if $2H+\frac{1}{2}-N_\psi\geq -\frac{1}{2}$, \eqref{e:Cov(sqrt(n)W(2j),sqrt(n/a)W(a2j'))->0_1a} is bounded by
\begin{equation}\label{e:Cov(sqrt(n)W(2j),sqrt(n/a)W(a2j'))->0_1b}
Ca(n)^{-(2H+3/2)}( 2 \lfloor b(n)/2^j \rfloor +1)\left(a(n)^{2H+\frac{1}{2}-N_\psi}\right)^2 \sim a(n)^{2H+2 - 2N_\psi} \to 0,
\end{equation}
as $n \rightarrow \infty$, since $N_\psi\geq 2$ by condition \eqref{e:Npsi>alpha}. On the other hand, if $2H+\frac{1}{2}-N_\psi<-\frac{1}{2}$, \eqref{e:Cov(sqrt(n)W(2j),sqrt(n/a)W(a2j'))->0_1a} is bounded by
\begin{equation}\label{e:Cov(sqrt(n)W(2j),sqrt(n/a)W(a2j'))->0_1c}
Ca(n)^{-(2H+3/2)}( 2 \lfloor b(n)/2^j \rfloor +1)o\Big(\frac{1}{\sqrt{a(n)}}\Big)^2\sim o\left(a(n)^{-2H}\right) \to 0.
\end{equation}
Combining \eqref{e:Cov(sqrt(n)W(2j),sqrt(n/a)W(a2j'))->0_1b} and \eqref{e:Cov(sqrt(n)W(2j),sqrt(n/a)W(a2j'))->0_1c},
\begin{equation}\label{e:Cov(sqrt(n)W(2j),sqrt(n/a)W(a2j'))->0_1}
 \frac{2na(n)^{-(2H+3/2)}}{n_{j}n_{a,j'}}\sum_{k=1}^{n_j}\sum_{k'=1}^{n_{a,j'}}1_{A_n}(k,k')\left(\E d_{\textnormal{fBm}}(2^j,k)d_{\textnormal{fBm}}(a(n)2^{j'},k')\right)^2\to 0,
\end{equation}
as $n \rightarrow \infty$. Moreover, expression \eqref{e:fbm_decorr} holds on $A_n^c$. Therefore,
$$
 \frac{2na(n)^{-(2H+3/2)}}{n_{j}n_{a,j'}}\sum_{k=1}^{n_j}{\sum_{k'=1}^{n_{a,j'}}}1_{A^c_n}(k,k')\left(\E d_{\textnormal{fBm}}(2^j,k)d_{\textnormal{fBm}}(a(n)2^{j'},k')\right)^2
$$
\begin{equation}\label{e:Cov(sqrt(n)W(2j),sqrt(n/a)W(a2j'))->0_2}
\leq C\frac{n}{a(n)^{2H+3/2}}\Big(\frac{a(n)^{N_\psi+\frac{1}{2}}}{|b(n)|^{2N_\psi-2H}}\Big)^{2}\sim \frac{n}{a(n)^{8N_\psi-8H+\frac{1}{2}}} \to 0,
\end{equation}
by \eqref{e:b(n)} and \eqref{e:seq_decay_assumption}, since $8N_\psi-8H+\frac{1}{2} > \frac{17}{2}>  2 \beta + 1$. Thus, we have shown, for any $j,j'$,
$$
\Cov\left(\sqrt{n}W(2^j),\sqrt{n/a(n)}W_a(2^{j'})\right) \to 0,\quad n\to\infty.
$$
In particular,
\begin{equation}\label{e:Cov(sqrt(n)W(2j),sqrt(n/a)W(a2j'))->0_3}
\Cov\left(\sqrt{n}\left(W(2^{j_1}),W(2^{j_2})\right),\sqrt{\frac{n}{a(n)}}\left(W_a(2^{j_3}),\ldots,W_a(2^{j_4})\right)\right) \to \mathbf{0},\quad n\to\infty.
\end{equation}
Recall that, for any set of octaves $j_3,\ldots,j_4$,
\begin{equation}\label{e:wavelet_var_fBm_weak_limit}
\sqrt{\frac{n}{a(n)}}\Big(W_a(2^j)-C2^{j(2H+1)} \Big)_{j=j_3,\ldots,j_4} \stackrel{d}\rightarrow {\mathcal N}(\mathbf{0},\widetilde{G}^{j_3,j_4}), \quad n\to\infty,
\end{equation}
where $C=\frac{\sigma^2}{C^2(H)}\int_\R\frac{|\widehat{\psi}(x)|^2}{|x|^{2H+1}}dx$ and  $\widetilde{G}^{j_3,j_4} = \{\widetilde{G}_{j,j'}\}_{j,j'= j_3,\ldots,j_4}$ are given by \eqref{e:Gtilde} {(see Remark \ref{r:sqrt(n/a(n))(Wa(2j)-C2j(2H+1))_asympt_norm}).} So, write
$$
\mathbf{v}_{j_1,j_2} = \left(W(2^{j_1})-\E d^2(2^{j_1},0),W(2^{j_2})-\E d^2(2^{j_2},0)\right),
$$
$$
\mathbf{v}_{j_3,j_4} = \left({W_a(2^{j_3})-C2^{j_3(2H+1)}},\ldots,W_a(2^{j_4})-C2^{j_4(2H+1)}\right),
$$
where $C$ is as in \eqref{e:wavelet_var_fBm_weak_limit}. By relation \eqref{e:Cov(sqrt(n)W(2j),sqrt(n/a)W(a2j'))->0_3},
\begin{equation}\label{e:Cov(sqrt{n}v_{j_1,j_2},sqrt{n/a(n)}v_{j_3,j_4})}
\Cov\left(\sqrt{n}\mathbf{v}_{j_1,j_2},\sqrt{\frac{n}{a(n)}}\mathbf{v}_{j_3,j_4}\right)\to \mathbf0, \quad n \rightarrow \infty,
\end{equation}
This shows that, with appropriate normalization, the random vectors \eqref{e:wavelet_var_fixed_coarse_scales_H0} are asymptotically uncorrelated, as claimed. To establish the joint asymptotic normality of \eqref{e:wavelet_var_fixed_coarse_scales_H0}, we now adapt the proofs of Lemmas \ref{l:ANwspec} and \ref{fBm_AN}. In fact, take
$$
V_n=\Big(d(2^{j_1},1),\ldots,d(2^{j_1},n_{j_1}),d(2^{j_2},1),\ldots,d(2^{j_2},n_{j_2}),$$$$\frac{d(a(n)2^{j_3},1)}{a(n)^{H+\frac{1}{2}}},\ldots,\frac{d(2^{j_3},n_{a,j_3})}{a(n)^{H+\frac{1}{2}}},\ldots,\frac{d(a(n)2^{j_4},1)}{a(n)^{H+\frac{1}{2}}},\ldots,\frac{d(2^{j_4},n_{a,j_4})}{a(n)^{H+\frac{1}{2}}}\Big)\in \R^\nu,
$$
where $\nu= n_{j_1}+n_{j_2}+n_{a,j_3}+\ldots+n_{a,j_4}$, and with the number of octaves $m=j_4-j_3+3$. Let
\begin{equation}\label{e:new_Dn}
D_n =\diag \Big(\underbrace{\frac{a_1}{n_{j_1}},\ldots,\frac{a_1}{n_{j_1}}}_{n_{j_1} \text{ terms}},\underbrace{\frac{a_2}{n_{j_2}},\ldots,\frac{a_2}{n_{j_2}}}_{n_{j_2} \text{ terms}},\underbrace{\frac{a_3}{n_{j_3}\sqrt{{a(n)}}},\ldots,\frac{a_3}{n_{j_3}\sqrt{{a(n)}}}}_{n_{a,j_3}\text{ terms}},\ldots,\underbrace{\frac{a_m}{n_{j_4}\sqrt{{a(n)}}},\ldots,\frac{a_m}{n_{j_4}\sqrt{{a(n)}}}}_{n_{a,j_4}\text{ terms}}\Big)
\end{equation}
and take $U_n= \sum_{j=j_1,j_2} a_jW(2^j) +\sum_{j=j_3}^{j_4} \frac{a_j}{\sqrt{a(n)}}W_a(2^j)$. Then, again we arrive at $\sqrt{\Var(U_n)}=\mathcal{O}(n^{-1/2})$ (cf.\ \eqref{e:nVarTn_conv}). Moreover, analogous to \eqref{e:fbm_eigbound} (see also \eqref{ANmaxbound}), the maximum eigenvalue $\lambda_{\text{max}}(n)$ of $\Gamma_n=\E V_nV_n^\top$ is bounded by %

\begin{equation}\label{e:lambda_max_fbm_with_mixed_scales}
\lambda_{\textnormal{max}}(n)\leq
(j_4-j_3+3)\max\bigg\{{\max_{j,j'=j_1,j_2}}\max_{k=1,\ldots,n_j} \sum_{k'=1}^{n_{j'}} | \E d_{\textnormal{fBm}}(2^j,k)d_{\textnormal{fBm}}(2^{j'},k')|,$$
$$ \max_{j=j_1,j_2}\max_{j'=j_3,\ldots,j_4}\max_{k=1,\ldots,n_j} \sum_{k'=1}^{n_{j'}}  \frac{|\E d_{\textnormal{fBm}}(2^j,k)d_{\textnormal{fBm}}(a(n)2^{j'},k')|}{a(n)^{H+\frac{1}{2}}},$$
$$ \max_{j=j_1,j_2}\max_{j'=j_3,\ldots,j_4}\max_{k=1,\ldots,n_j} \sum_{k'=1}^{n_{j'}}  \frac{|\E d_{\textnormal{fBm}}(a(n)2^j,k)d_{\textnormal{fBm}}(2^{j'},k')|}{a(n)^{H+\frac{1}{2}}},$$
$$ \max_{j,j'=j_3,\ldots,j_4}\max_{k=1,\ldots,n_j} \sum_{k'=1}^{n_{j'}}   \frac{|\E d_{\textnormal{fBm}}(a(n)2^j,k)d_{\textnormal{fBm}}(a(n)2^{j'},k')|}{a(n)^{2H+1}}\bigg\}.
\end{equation}
In turn, the right-hand side of \eqref{e:lambda_max_fbm_with_mixed_scales} is bounded as $n\to\infty$ due to both \eqref{e:fbm_decorr} and \eqref{e:EdfBm(2^j,k)dfBm(2^{j'},k')}, i.e., $\lambda_{\text{max}}(n)=o\left(\sqrt{\Var{U_n}}\right)$  (cf.\ \eqref{e:lambda-max=o(sqrt(VarTn))}). Thus, by the proof of Lemma \ref{fBm_AN}, we obtain
\begin{equation}\label{e:v_j1,..v_j4_AN_under_H0}
\bbR^{j_4-j_3+3} \ni \sqrt{n}
\begin{pmatrix}
\mathbf{v}_{j_1,j_2} \\
a^{-1/2}(n)\mathbf{v}_{j_3,j_4}
\end{pmatrix}
\stackrel{d}{\to}
\mathcal{N}\left(\mathbf0,
\begin{pmatrix}G^{j_1,j_2} & \mathbf 0\\
\mathbf{0} &\widetilde{G}^{j_3,j_4} \\
\end{pmatrix}
\right), \quad n \rightarrow \infty,
\end{equation}
where $G^{j_1,j_2} = \{G_{j,j'}\}_{j,j'= j_1,j_2}$, $\widetilde{G}^{j_3,j_4} = \{\widetilde{G}_{j,j'}\}_{j,j'= j_3,\ldots,j_4}$, $G$ is the matrix appearing in \eqref{e:sqrt(n)(WfBm-EWfBm)->Normal}, and $\widetilde{G}$ is the matrix \eqref{e:Gtilde}.

Now, to establish \eqref{e:test_stat_under_H0}, we will reexpress the vector of estimators $(\widehat{H}_n, \widehat{\sigma^2_n},\widetilde{H}_n)$ in terms of wavelet variances. We start off with $(\widehat{H}_n, \widehat{\sigma^2_n})$. Let $f_n$ and $\widetilde{\eta}_j$ be as in \eqref{e:def_estimator} and \eqref{e:eta-tilde}, respectively. By the mean value theorem,
$$
\sqrt{n}((\widehat{H}_n,\widehat{\sigma^2_n})-(H,\sigma^2))  = -\left[D^2f_n(H^*_n,{\sigma^2}^*_n)\right]^{-1}2\sum_{j= j_1,j_2}\sqrt{n}\left(\widetilde{\eta_j}(H,\sigma^2)-\log_2 W(2^j)\right)w_jD\widetilde\eta_{j}(H,\sigma^2)
$$
\begin{equation}\label{e:fbm_Mest_meanval}
= -\left[D^2f_n(H^*_n,{\sigma^2}^*_n)\right]^{-1}2\sum_{j= j_1,j_2}\sqrt{n}\left(B(j) +\frac{W(2^j)-\E d_{\textnormal{fBm}}^2(2^j,0)}{W^*(2^j)\log(2)}\right)w_jD\widetilde\eta_{j}(H,\sigma^2).
\end{equation}
In \eqref{e:fbm_Mest_meanval}, $(H^*_n,{\sigma^2}^*_n)$ and $W^*(2^j)$ lie between $(\widehat{H}_n,\widehat{\sigma^2_n})$ and $(H,\sigma^2)$, and between $W(2^j)$ and $\E d_{\textnormal{fBm}}^2(2^j,0)$, respectively (cf.\ expression \eqref{e:th3taylor}).  By following an argument similar to that of Theorem \ref{t:asympt}, $(i)$, we obtain $(\widehat{H}_n,\widehat{\sigma^2_n})\pto (H,\sigma^2)$. Further noting that $W(2^j)-\eta_j(\bth)\pto 0$, by continuity (cf.\ expressions \eqref{e:Hessian_conv} and \eqref{e:inv_second_deriv_conv}),
$$
\left[D^2f_n(H^*_n,{\sigma^2}^*_n)\right]^{-1}\pto \widetilde{Q}(\bth_0)^{-1},
$$
where $\widetilde{Q}(\bth)$ is given by \eqref{e:def_widetilde_Q}.
Likewise, by \eqref{e:sqrt(n)(WfBm-EWfBm)->Normal}, $W^*(2^j)\pto \E d_{\textnormal{fBm}}^2(2^j,0)$. 
We now turn to $\widetilde{H}_n$. Note that from \eqref{e:wavelet_var_fBm_weak_limit} and the limit \eqref{varrho_to_varrho_inf},
\begin{equation}\label{e:Wa(2j)->C2^(j(2H+1))}
W_a(2^j)\pto \varrho_{\infty}(j,j,0)= C2^{j(2H+1)}, \quad n \rightarrow \infty.
\end{equation}
By relation \eqref{e:sum-varpi=0,sum-j*varpij=1}, the mean value theorem and the same $C > 0$ as in \eqref{e:Wa(2j)->C2^(j(2H+1))},
$$
\widetilde{H}_n -H= \sum^{j_4}_{j=j_3} \frac{\varpi_{j}}{2} \left( \log_2 W(a(n)2^j) - (2H+1) \right)
$$
$$
= \sum^{j_4}_{j=j_3} \frac{\varpi_{j}}{2}\left(\log_2 W_a(2^j) - j(2H+1) - \log_2C\right) = \sum^{j_4}_{j=j_3} \frac{\varpi_{j}}{2} \left(\log_2 W_a(2^j) - \log_2(C2^{j(2H+1)}) \right)
$$
\begin{equation}\label{e:fbm_wavereg_meanval}
= \sum^{j_4}_{j=j_3} \frac{\varpi_{j}}{2}\left(\frac{W_a(2^j)-C2^{j(2H+1)}}{W_a^*(2^j)\log(2)}\right),
\end{equation}
where $W_a^*(2^j)$ is between $W_a(2^j)$ and $C2^{j(2H+1)}$. In view of \eqref{e:Wa(2j)->C2^(j(2H+1))}, $W_a^*(2^j)\pto C2^{j(2H+1)}$, $n \rightarrow \infty$.

Thus, by considering \eqref{e:fbm_Mest_meanval} and \eqref{e:fbm_wavereg_meanval}, we may rewrite
\begin{equation}\label{e:fbm_estim_matrix}
\sqrt{n}\begin{pmatrix}
\widehat{H}_n-H\\
\widehat{\sigma^2_n}-\sigma^2\\
a(n)^{-\frac{1}{2}}(\widetilde{H}_n-H)\end{pmatrix}=
\sqrt{n}
\begin{pmatrix}
\mathbf{u}_1^{\top}\mathbf{v}_{j_1,j_2}  +u_B \\
\mathbf{u}_2^{\top}\mathbf{v}_{j_1,j_2}+u_B\\
a(n)^{-\frac{1}{2}}\mathbf{u}_3^{\top} \mathbf{v}_{j_3,j_4}
\end{pmatrix}.
\end{equation}
In \eqref{e:fbm_estim_matrix}, $\mathbf{u}_1$, $\mathbf{u}_2$ and $\mathbf{u}_3$ are random vectors satisfying
\begin{equation}\label{e:u1,u2,u3}
\mathbf{u}_1=\mathbf{u}_1(H^*_n,{\sigma^2}^*_n,W^*(2^{j_1}),W^*(2^{j_2}))\pto \mathbf{c}_1, \quad \mathbf{u}_2=\mathbf{u}_2(H^*_n,{\sigma^2}^*_n,W^*(2^{j_1}),W^*(2^{j_2}))\pto \mathbf{c}_2,$$$$  \quad \mathbf{u}_3=\mathbf{u}_3(W_a^*(2^{j_3}),\ldots,W_a^*(2^{j_4}))\pto \mathbf{c}_3,
\end{equation}
as $n \rightarrow \infty$ for vector constants $\mathbf{c}_1$, $\mathbf{c}_2$, $\mathbf{c}_3$. 
Moreover,
$$
u_B=u_B(B(j_1),B(j_2),H^*_n,{\sigma^2}^*_n,W^*(2^{j_1}),W^*(2^{j_2}))
$$
in \eqref{e:fbm_estim_matrix} is a residual term satisfying
\begin{equation}\label{e:sqrt(n)u_B->0}
\sqrt{n}u_B\pto 0, \quad n \rightarrow \infty,
\end{equation}
which is a consequence of \eqref{e:B(j)=o(1/n)}.
By following a proof nearly identical to that of Theorem \ref{t:asympt}, $(ii)$, one has
\begin{equation}\label{e:widehat(H_n)_AN}
n\Var(\widehat{H}_n,\widehat{\sigma^2_n})\to N(\bth_0),\quad n\to\infty
\end{equation}
where $N(\bth_0)$ is the matrix  \eqref{e:N(theta0)}.  Likewise, from \eqref{e:fbm_a_AN}
\begin{equation}\label{e:widetilde(H_n)_AN}
n\Var(\widetilde{H}_n)\to \gamma_0^2,\quad n\to\infty.
 \end{equation}
Thus, in view of 
\eqref{e:v_j1,..v_j4_AN_under_H0},  \eqref{e:fbm_estim_matrix}, \eqref{e:widehat(H_n)_AN}, and \eqref{e:widetilde(H_n)_AN} by Slutsky's theorem, the vector $\sqrt{n}((\widehat{H}_n,\widehat{\sigma^2_n},a^{-1/2}(n)\widetilde{H}_n)-(H,\sigma^2,H))$ is asymptotically normal with covariance matrix
$$
M = \begin{pmatrix}
N & \mathbf{0}\\
\mathbf{0} & \gamma_0^2
\end{pmatrix}.
$$
This shows \eqref{e:test_stat_under_H0}. \\

\noindent We turn to $(ii)$.
First, let $\bth_0=(H,\lambda,\sigma^2)$ be the true parameter vector of the underlying tfBm.  Note the functions
$$
(H,\sigma^2) \mapsto \E d_{\text{fBm}}^2(2^j,0),\quad (H,\lambda,\sigma^2)\mapsto  \E d_{\text{tfBm}}^2(2^j,0),\quad j\in\N
$$
are continuous. {Thus, from assumption \eqref{e:slope_assumption} there exists an $H'=H'(H,\lambda)\in(0,H)$ such that
$$\log_2\E_{H'} d^2_{\text{fBm}}(2^{j_2},0)-\log_2\E_{H'}d^2_{\text{fBm}}(2^{j_1},0) =\log_2\E_{H,\lambda} d^2_{\text{tfBm}}(2^{j_2},0)-\log_2\E_{H,\lambda}d^2_{\text{tfBm}}(2^{j_1},0),$$
 i.e., there is a pair $(H',\sigma'^2)(\neq (H,\sigma^2))$ as in \eqref{e:(H',sigma'^2)} such that \eqref{e:Ed2_fbm=Ed2_tfBm} holds.

 Now, as in the proof of part $(i)$, we first show that, with appropriate normalization, the vectors $\left(\sqrt{n}W(2^{j_1}),\sqrt{n}W(2^{j_2})\right)$ and $\left(\sqrt{n/a(n)}W(a(n)2^{j_3}),\ldots,\sqrt{n/a(n)}W(a(n)2^{j_4})\right)$ are asymptotically decorrelated, and then establish their joint asymptotic normality. Proposition \ref{p:tfbm_decorr} implies that, for large $n$,
$$
\textnormal{Cov}(d(2^j,k),d(a(n)2^{j'},k')) \leq Ca^{1/2}(n)2^{\frac{j+j'}{2}}|2^jk-a(n)2^{j'}k'|^{H-1/2}e^{-\lambda|2^jk-a(n)2^{j'}k'|}
$$
under the condition $|2^jk-a(n)2^{j'}k'|\geq2^{j'+1}a(n)T$. Let
\begin{equation}\label{e:b_*(n)}
b_*(n) = 2^{j'+1}a(n)T
\end{equation}
(cf.\ expression \eqref{e:b(n)}). For fixed $k'$,
$$
\#\{k: |2^j k - a(n)2^{j'}k'|\leq b_*(n)\}
$$
$$
=\#\{k: a(n)2^{j'}k'-b_*(n)\leq 2^jk\leq a(n)2^{j'}k' +b_*(n)\} \leq 2 \lfloor b_*(n)/2^j \rfloor +1.
$$
Now, consider the set $A_n(b_*(n))$, where $A_n(\cdot)$ is given by \eqref{e:A_n}. For notational simplicity, we write, again, $A_n$. Note that, for any $j= j_1,j_2$ and any $j'= j_3,\ldots,j_4$,
$$
\Cov\left(\sqrt{n}W(2^j),\sqrt{n/a(n)}W(a(n)2^{j'})\right) = \frac{n}{a^{1/2}(n)}\Cov\left(W(2^j),W(a(n)2^{j'})\right)
$$
\begin{equation}\label{e:Cov(sqrt{n}W(2^j),sqrt{n/a(n)}W(a(n)2^{j'}))_two_terms}
= \frac{2n}{a^{1/2}(n) n_j n_{a,j}} \sum_{k=1}^{n_{j}}\sum_{k'=1}^{n_{a,j'}}(1_{A_n}(k,k')+1_{A_n^c}(k,k')) \left(\E d_{\text{tfBm}}(2^j,k')d_{\text{tfBm}}(a(n)2^{j'},k)\right)^2.
\end{equation}
Thus, by using \eqref{e:EdtfBm(2^j,k)dtfBm(2^{j'},k')} and relation \eqref{bn_count}, the first double summation term on the right-hand side of \eqref{e:Cov(sqrt{n}W(2^j),sqrt{n/a(n)}W(a(n)2^{j'}))_two_terms} can be bounded by
$$
\frac{2n}{a^{1/2}(n) n_j n_{a,j}} \sum_{k=1}^{n_{j}}\sum_{k'=1}^{n_{a,j'}}1_{A_n}(k,k') \left(\E d_{\text{tfBm}}(2^j,k')d_{\text{tfBm}}(a(n)2^{j'},k)\right)^2
$$
\begin{equation}\label{e:tfbm_Cov(sqrt(n)W(2j),sqrt(n/a)W(a2j'))->0_1}
\leq C\frac{n}{a^{1/2}(n) n_j} ( 2 \lfloor b_*(n)/2^j \rfloor +1) \hspace{1mm}o\Big(\Big(\frac{1}{\sqrt{a(n)}}\Big)^2\Big)\sim a^{1/2}(n)\hspace{1mm}o\Big(\frac{1}{a(n)}\Big)\to 0, \quad n \rightarrow \infty.
\end{equation}
In turn, by \eqref{e:tfbm_decorr}, the second double summation term on the right-hand side of \eqref{e:Cov(sqrt{n}W(2^j),sqrt{n/a(n)}W(a(n)2^{j'}))_two_terms} can be bounded by
$$
\frac{2n}{a^{1/2}(n) n_j n_{a,j}} \sum_{k=1}^{n_{j}}\sum_{k'=1}^{n_{a,j'}}1_{A^c_n}(k,k') \left(\E d_{\text{tfBm}}(2^j,k')d_{\text{tfBm}}(a(n)2^{j'},k)\right)^2
$$
\begin{equation}\label{e:tfbm_Cov(sqrt(n)W(2j),sqrt(n/a)W(a2j'))->0_2}
\leq C\frac{n}{a^{1/2}(n)} \left(a(n)^{1/2}b_*(n)^{H-\frac{1}{2}} {e^{-\lambda b_*(n)}}\right)^2 \to 0, \quad n \rightarrow \infty,
\end{equation}
i.e., for every $j,j'\in\N$, $\Cov\left(\sqrt{n}W(2^j),\sqrt{n/a(n)}W(a(n)2^{j'})\right)\to 0$ as $n\to\infty$. Thus,
\begin{equation}\label{e:Cov(sqrt(n)W(2j),sqrt(n/a)W(a2j'))->0_tfBm}
\Cov\left(\sqrt{n}\left(W(2^{j_1}),W(2^{j_2})\right),\sqrt{\frac{n}{a(n)}}\left(W_a(2^{j_3}),\ldots,W_a(2^{j_4})\right)\right) \to \mathbf{0},\quad n\to\infty.
\end{equation}
Now, writing
$$
\mathbf{v}_{j_1,j_2} = \left(W(2^{j_1})-\E_{H',\sigma'^2} d_{\textnormal{fBm}}^2(2^{j_1},0),{W(2^{j_2})-\E_{H',\sigma'^2} d_{\textnormal{fBm}}^2(2^{j_2},0)}\right)
$$
$$
 = \left({W(2^{j_1})-\E_{\bth_0} d_{\textnormal{tfBm}}^2(2^{j_1},0)},W(2^{j_2})-\E_{\bth_0} d_{\textnormal{tfBm}}^2(2^{j_2},0)\right)
$$
$$
\mathbf{v}_{j_3,j_4} = \left({W(a(n)2^{j_3})-\sum_{\ell\in\bbZ}\frac{\sigma^2 \Gamma^2(H+\frac{1}{2})}{|\lambda^2+(2\pi \ell)|^2}},\ldots,{W(a(n)2^{j_4})-\sum_{\ell\in\bbZ}\frac{\sigma^2 \Gamma^2(H+\frac{1}{2})}{|\lambda^2+(2\pi \ell)|^2}}\right).
$$
By relation \eqref{e:Cov(sqrt(n)W(2j),sqrt(n/a)W(a2j'))->0_tfBm}, the same limit \eqref{e:Cov(sqrt{n}v_{j_1,j_2},sqrt{n/a(n)}v_{j_3,j_4})} holds. {In regard to asymptotic normality,} similarly to part $(i)$, we adapt the argument in Lemmas \ref{l:ANwspec} and \ref{fBm_AN}. Let
$$
V_n=\Big(d(2^{j_1},1),\ldots,d(2^{j_1},n_{j_1}),d(2^{j_2},1),\ldots,d(2^{j_2},n_{j_2}),
$$
$$
d(a(n)2^{j_3},1),\ldots,d(a(n)2^{j_3},n_{a,j_3}),\ldots,d(a(n)2^{j_4},1),\ldots,d(a(n)2^{j_4},n_{a,j_4})\Big),
$$
define $D_n$ as in \eqref{e:new_Dn}, and let $U_n= \sum_{j=j_1,j_2} a_jW(2^j) +\sum_{j=j_3}^{j_4} \frac{a_j}{\sqrt{a(n)}}W(a(n)2^j)$. The analogous expression to \eqref{e:lambda_max_fbm_with_mixed_scales} is bounded due to \eqref{e:tfbm_decorr} and \eqref{e:EdtfBm(2^j,k)dtfBm(2^{j'},k')}. Therefore,
\begin{equation}\label{e:mixed_scales_AN_under_H1}
\bbR^{j_4-j_3+3} \ni \sqrt{n}
\begin{pmatrix}
\mathbf{v}_{j_1,j_2} \\
a^{-1/2}(n)\mathbf{v}_{j_3,j_4}
\end{pmatrix}
\stackrel{d}{\to}
\mathcal{N}\left(\mathbf0,
\begin{pmatrix}F^{j_1,j_2} & \mathbf 0\\
\mathbf{0} &\widetilde{F}^{j_3,j_4} \\
\end{pmatrix}
\right), \quad n \rightarrow \infty.
\end{equation}
In \eqref{e:mixed_scales_AN_under_H1}, $F^{j_1,j_2} = \{F_{j,j'}\}_{j,j'= j_1,j_2}$, $\widetilde{F}^{j_3,j_4} = \{\widetilde{F}_{j,j'}\}_{j,j'= j_3,\ldots,j_4}$, $F_{j,j'}$ is given by \eqref{e:Fjj'} and $\widetilde{F}_{j,j'}$ is given by \eqref{e:F-tilde_j,j'}, since by the Isserlis relation and Lemma \ref{Cov_scale_to_infty},
\begin{equation}\label{H1_cov_mat}
\frac{n}{a(n)}\Cov\left(W(2^ja(n)),W(2^{j'}a(n))\right) \to 2^{\min\{j,j'\}+1}\sum_{r\in\Z} \zeta_{\infty}^2(j,j',{r2^{\min\{j,j'\}}}) = \widetilde{F}_{j,j'}.
\end{equation}
As in the proof of part $(i)$, we now reexpress the vector of estimators $(\widehat{H}_n, \widehat{\sigma^2_n},\widetilde{H}_n)$ in terms of wavelet variances.  By the mean value theorem and \eqref{e:Ed2_fbm=Ed2_tfBm},
$$
\sqrt{n}((\widehat{H}_n,\widehat{\sigma^2_n}) -(H',{\sigma^2}'))
$$
\begin{equation}\label{e:tfbm_Mest_meanval}
= -\left[D^2f_n(H^*_n,{\sigma^2_n}^*)\right]^{-1}2\sum_{j= j_1,j_2}\sqrt{n}\left(B(j) +\frac{W(2^j)-\E_{H',{\sigma^2}'} d_{\textnormal{fBm}}^2(2^j,0)}{W^*(2^j)\log(2)}\right)w_jD\widetilde\eta_{j}(H',{\sigma^2}'),
\end{equation}
where $\E_{H',{\sigma^2}'} d_{\textnormal{fBm}}^2(2^j,0) = \E_{\bth_0} d_{\textnormal{tfBm}}^2(2^j,0)$, $(H^*_n,{\sigma^2_n}^*)$ is between $(\widehat{H}_n,\widehat{\sigma^2_n})$ and $(H',{\sigma^2}')$, and $W^*(2^j)$ is between $W(2^j)$ and $\E_{\bth_0} d_{\textnormal{tfBm}}^2(2^j,0)$.   By following an argument similar to that of Theorem \eqref{t:asympt}, $\sqrt{n}(\widehat{H}_n-H,\widehat{\sigma^2_n}-\sigma^2)$ is asymptotically normal, and in particular, $(\widehat{H}_n,\widehat{\sigma^2_n})\pto (H,\sigma^2)$. Thus, similarly as in part $(i)$, by continuity,
$$
\left[D^2f_n(H^*_n,{\sigma^2}^*_n)\right]^{-1}\pto \widetilde{Q}(\bth')^{-1},
$$
where $\widetilde{Q}(\bth)$ is given by \eqref{e:def_widetilde_Q}. Likewise, by Theorem \ref{t:asympt},
$$
W(2^j)\pto \E_{\bth_0}d_{\textnormal{tfBm}}^2(2^j,0), \quad W^*(2^j)\pto \E_{\bth_0} d_{\textnormal{tfBm}}^2(2^j,0), \quad n \rightarrow \infty.
$$
By relation \eqref{e:sum-varpi=0,sum-j*varpij=1} and the mean value theorem,
$$
\widetilde{H}_n - (-1/2)= \frac{1}{2}\sum^{j_4}_{j=j_3} \varpi_{j} \log_2 W(a(n)2^j)
$$
$$
= \frac{1}{2}\sum^{j_4}_{j=j_3} \varpi_{j}\left(\log_2 W(a(n)2^j) - \log_2\sum_{\ell\in\bbZ}\frac{\sigma^2 \Gamma^2(H+\frac{1}{2})}{|\lambda^2+(2\pi \ell)^2|^{H+\frac{1}{2}}}\right)
$$
\begin{equation}\label{e:tfbm_wavereg_meanval}
= \frac{1}{2}\sum^{j_4}_{j=j_3} \varpi_{j}\Big(\frac{W(a(n)2^j)- \sum_{\ell\in\bbZ} \sigma^2 \Gamma^2(H+\frac{1}{2})/|\lambda^2+(2\pi \ell)^2|^{H+\frac{1}{2}}}{W^*(a(n)2^j)\log(2)} \Big),
\end{equation}
where $W^*(a(n)2^j)$ is between  $W(a(n)2^j)$ and $\sum_{\ell\in\bbZ}\frac{\sigma^2 \Gamma^2(H+\frac{1}{2})}{|\lambda^2+(2\pi \ell)^2|^{H+\frac{1}{2}}}$.  From  the limits \eqref{zeta_to_zeta_inf} and \eqref{e:mixed_scales_AN_under_H1}, $W(a(n)2^j)\pto \zeta_\infty(j,j,0)=\sum_{\ell\in\bbZ}\frac{\sigma^2 \Gamma^2(H+\frac{1}{2})}{|\lambda^2+(2\pi \ell)^2|^{H+\frac{1}{2}}}=\zeta_\infty(0,0,0)$, i.e., $W^*(a(n)2^j)\pto \zeta_\infty(0,0,0)$. Now, by considering \eqref{e:tfbm_Mest_meanval} and \eqref{e:tfbm_wavereg_meanval}, we may write out a random vector analogous to \eqref{e:fbm_estim_matrix}, where
the limits analogous to \eqref{e:u1,u2,u3} and \eqref{e:sqrt(n)u_B->0} hold.

In view of the convergence of $\sqrt{n}\mathbf{v}_{j_1,j_2}$ in \eqref{e:mixed_scales_AN_under_H1}, by following an argument analogous to that of Theorem \ref{t:asympt}, writing $\bth'=(H',{\sigma^2}')$, we obtain
\begin{equation}\label{e:Hhat_underH1_AN}
n\Var(\widehat{H}_n,\widehat{\sigma^2_n})\to K(\bth').
\end{equation}
In view of the convergence of $\sqrt{n}\mathbf{v}_{j_3,j_4}$ in \eqref{e:mixed_scales_AN_under_H1}, and in view of \eqref{e:tfbm_wavereg_meanval},
as $n\to\infty$,
\begin{equation}\label{e:eq_Var_widetilde_H_tfBm}
\frac{n}{a(n)}\Var(\widetilde{H}_n) \to \frac{1}{\zeta_\infty^2(0,0,0)4\log^2 2 }\sum_{j=j_3}^{j_4}\sum_{j'=j_3}^{j_4}\varpi_j\varpi_{j'}\widetilde{F}_{j,j'} =\gamma_1^2.
\end{equation}}
Thus, starting from the random vector \eqref{e:fbm_estim_matrix} and relations \eqref{e:u1,u2,u3} and \eqref{e:sqrt(n)u_B->0} redefined for tfBm, in view of \eqref{e:mixed_scales_AN_under_H1}, \eqref{e:Hhat_underH1_AN} and \eqref{e:eq_Var_widetilde_H_tfBm}, Slutsky's theorem implies that the random vector $\sqrt{n}((\widehat{H}_n,\widehat{\sigma^2_n},a^{-1/2}(n)\widetilde{H}_n)-(H,\sigma^2,H))$ is asymptotically normal with covariance matrix
$$
L = \begin{pmatrix}
K & \mathbf{0}\\
\mathbf{0} & \gamma_1^2
\end{pmatrix}. \quad \Box\\
$$
\end{proof}

\begin{remark}\label{r:sqrt(n/a(n))(Wa(2j)-C2j(2H+1))_asympt_norm}
Under condition \eqref{e:seq_decay_assumption}, expression \eqref{e:wavelet_var_fBm_weak_limit} can be established by making use of Lemma A.4 and adapting Lemma A.1, or alternatively by a simple adaptation of the argument for showing Proposition D.3 in Abry et al.\ \cite{abry:didier:li:2018}, which includes the simplification that, for a univariate process, no demixing matrix needs to be considered -- see the reference for details.
\end{remark}

\noindent {\sc Proof of Theorem \ref{t:test}}:
For (\emph{i}), reexpress the matrix \eqref{e:matrix_M} as $M = (M_{i\ell})_{i,\ell=1,2,3}$ and write $\tau^2_0(H):=M_{11}+\gamma_0^2$. Lemma \ref{Hest_AN} implies that
$$
\sqrt{n}T_n \stackrel{d}\to \mathcal{N}(0,\tau^2_0(H)), \quad n \rightarrow \infty.
$$
For $(ii)$, reexpress the matrix \eqref{e:matrix_K} as $K = (K_{i\ell})_{i,\ell=1,2}$.   Also write
$$
\tau^2_1(H,\lambda):=  K_{11} + \gamma_1^2.
$$
By Lemma \ref{Hest_AN}, $(ii)$, $\widehat{H}_n \pto H'$ for some $H'\in(0,H)$ depending on both $H$ and $\lambda$, and $\widetilde{H}_n\pto -\frac{1}{2}$, i.e. $T_n\pto H'+\frac{1}{2}$.  So let $\vartheta = H'+\frac{1}{2}$.  Again by Lemma \ref{Hest_AN},
$$
\sqrt{n}(T_n-\vartheta) \stackrel{d}\to \mathcal{N}(0,\tau^2_1(H,\lambda)), \quad n \rightarrow \infty. \quad \Box\\
$$

\bibliography{tempered_fBm}

\small

\bigskip

\noindent \begin{tabular}{lr}
B. Cooper\ Boniece and Gustavo Didier & \hspace{6cm} Farzad Sabzikar \\
Mathematics Department & Department of Statistics \\
Tulane University  & Iowa State University \\
6823 St.\ Charles Avenue & 2438 Osborn Drive  \\
New Orleans, LA 70118, USA & Ames, IA 50011-1090, USA \\
{\it bboniece@tulane.edu} & {\it sabzikar@iastate.edu}  \\
{\it gdidier@tulane.edu}       &   \\
\end{tabular}\\

\smallskip

\end{document}